\documentclass[a4paper]{article}

\usepackage{mystyle}

\title{Spherical CR uniformization of Dehn surgeries of the Whitehead link complement}
\author{Miguel ACOSTA \\
UPMC, 
IMJ-PRG\\
 and Université de Lorraine - IECL\\ 
France}
\begin{document}
\maketitle

\begin{abstract}
We apply a spherical CR Dehn surgery theorem in order to obtain infinitely many Dehn surgeries of the Whitehead link complement that carry spherical CR structures. We consider as starting point the spherical CR uniformization of the Whitehead link complement constructed by Parker and Will, using a Ford domain in the complex hyperbolic plane $\h2c$.
We deform the Ford domain of Parker and Will in $\h2c$ in a one parameter family. On the one side, we obtain infinitely many spherical CR uniformizations on a particular Dehn surgery on one of the cusps of the Whitehead link complement. On the other side, we obtain spherical CR uniformizations for infinitely many Dehn surgeries on the same cusp of the Whitehead link complement. These manifolds are parametrized by an integer $n \geq 4$, and the spherical CR structure obtained for $n = 4$ is the Deraux-Falbel spherical CR uniformization of the Figure Eight knot complement.
\end{abstract}

\tableofcontents

\setcounter{section}{-1}
\section{Introduction}

The present work takes place in the frame of the study of geometric structures on manifolds as well as in complex hyperbolic geometry. 
The Thurston geometrization conjecture, recently proved by Perelman, confirms that the study of the geometric structures carried by manifolds is extremely useful in order to understand their topology: any 3-dimensional manifold can be cut into pieces that carry a geometric structure.
Among the 3-dimensional structures, we find the spherical CR structures: they are not on the list of the eight 3-dimensional Thurston geometries but have an interesting behavior, and there are relatively few general facts known about them.
 More precisely, a spherical CR structure is a $(G,X)$-structure, where $X = \dh2c \simeq S^3$ is the boundary at infinity of the complex hyperbolic plane $\h2c$ and $G = \pu21$ is the group of holomorphic isometries of $\h2c$.
Hence, the study of discrete subgroups of $\pu21$ is closely related to the understanding of spherical CR structures.
An approach to construct such discrete subgroups is to consider triangle groups. The $(p,q,r)$ triangle group is the group $\Delta_{(p,q,r)}$ with presentation
\[\langle (\sigma_1,\sigma_2,\sigma_3) \mid \sigma_1^2 = \sigma_2^2 = \sigma_3^2 =  (\sigma_2 \sigma_3)^p = (\sigma_3\sigma_1)^q = (\sigma_1\sigma_2)^r = \mathrm{Id} \rangle .\]

 If $p$, $q$ or $r$ equals $\infty$, then the corresponding relation does not appear. The representations of triangle groups into $\pu21$ where the images of $\sigma_1$, $\sigma_2$ and $\sigma_3$ are complex reflexions have been widely studied. 
 For example, in \cite{goldman_complex_1992}, Goldman and Parker study the representations of $(\infty , \infty , \infty )$ triangle groups, and show that they are parametrized, up to conjugation, by a real number $s \in [0, + \infty [$. They conjecture a condition on the parameter for having a discrete and faithful representation. The conjecture, proved by Schwartz in \cite{schwartz_ideal_2001} and \cite{schwartz_better_2005}, can be summarized as follows:
\begin{thm*}[Goldman-Parker, Schwartz]
 A representation of the $(\infty,\infty,\infty)$ triangle group into $\pu21$ is discrete and faithful if and only if the image of $\sigma_1\sigma_2\sigma_3$ is non-elliptic.
\end{thm*}

A more complete picture on complex hyperbolic triangle groups can be found in the survey of Schwartz \cite{schwartz_complex_2002}, where he states the following conjecture:

\begin{conj*}[Schwartz]
 Let $\Delta_{(p,q,r)}$ a triangle group with $p \leq q \leq r$. Then, a representation of $\Delta_{(p,q,r)}$ into $\pu21$ where the images of the generators are complex reflexions is discrete and faithful if
and only if the images of $\sigma_1\sigma_2\sigma_1\sigma_3$ and $\sigma_1\sigma_2\sigma_3$ are not elliptic. Furthermore:
\begin{enumerate}
 \item If $p < 10$ then the representation is discrete and faithful if and only if image of $\sigma_1\sigma_2\sigma_1\sigma_3$ is
nonelliptic.
 \item If $p > 13$ then the representation is discrete and faithful if and only if image of $\sigma_1\sigma_2\sigma_3$ is
nonelliptic.
\end{enumerate}
\end{conj*}

More recently, in \cite{deraux_deforming_2006}, Deraux studies the representations of $(4,4,4)$-triangle groups, and shows that the representation for which the image of $\sigma_1\sigma_2\sigma_1\sigma_3$ is of order 5 is a lattice in $\pu21$.
 In \cite{parker_wang_xie}, Parker, Wang and Xie study the representations of $(3,3,n)$-triangle groups, and prove Schwartz conjecture in this case. Namely, the discrete and faithful representations are the ones for which the image of $\sigma_1\sigma_2\sigma_1\sigma_3$ is
nonelliptic. These representations will appear naturally in this article.

 Back to the geometric structures, determining if a manifold carries a spherical CR structure or not is a difficult question. The only negative result known to us is due to Goldman in \cite{goldman_tore}, and concerns the torus bundles over the circle. On the side of the known structures on manifolds, we can obtain spherical CR structures on quotients of $S^3$ of the form
 $\Gamma \backslash \dh2c$, containing the lens spaces.
 In \cite{falbel-gusevskii}, Falbel and Gusevskii construct spherical CR structures on circle bundles on hyperbolic surfaces with arbitrary Euler number $e \neq 0$.
 
 Constructing discrete subgroups of $\pu21$ can be used to construct spherical CR structures on manifolds.
Among general $(G,X)$-structures, the structures obtained as $\Gamma \backslash X$ for $\Gamma < G$ are called \emph{complete}, and are specially interesting since all the information is given by the group $\Gamma$. In the case of the spherical CR structures, we are interested in a more general class of structures, called \emph{uniformizable};
we say that a spherical CR structure on a manifold $M$ is uniformizable if it is obtained as as $\Gamma \backslash \Omega_\Gamma$, where $\Omega_\Gamma \subset \dh2c$ is the set of discontinuity of $\Gamma$.

In order to show that a spherical CR structure is uniformizable, the proofs often extend the structure to $\h2c$ and use the Poincaré polyhedron theorem as stated for example in \cite{parker_complex_2017a}. Besides the examples cited below, there are mainly three known uniformizable spherical CR structures on cusped manifolds.
There are two different uniformizable spherical CR structures of the Whitehead link complement. The first one is constructed
 in \cite{schwartz}, by R. Schwartz. The second one is constructed by Parker and Will in \cite{parker_complex_2017a}. 
  In \cite{falbel}, Deraux and Falbel construct a spherical CR uniformization of Figure eight knot complement. This uniformization can be deformed in a one parameter family of uniformizations, as shown by Deraux in \cite{deraux_uniformizations}.

%TECHNIQUES ? \cite{deraux_spherical_2015}

On the other hand, as in the real hyperbolic case treated in the notes of Thurston \cite{gt3m}, we can expect to construct spherical CR uniformizations of other manifolds by performing a Dehn surgery on a cusp of one of the examples above. In \cite{schwartz}, Schwartz proves a spherical CR Dehn surgery theorem, stating that, under some convergence conditions, the representations close to the holonomy representation of a uniformizable structure give spherical CR uniformizations of Dehn surgeries of the initial manifold.

 This can be applied to the first uniformization of the Whitehead link complement. It leads to an infinity of uniformizable manifolds, parametrized by some rational points in an open set of a deformation space. However, the hypotheses of the Schwartz surgery theorem contain a condition on the \emph{porosity} of the limit set of the holonomy representation, that we were unable to check in the two other cases. In \cite{acosta_spherical_2016}, we show another spherical CR Dehn surgery theorem, with weaker hypotheses and weaker conclusions, giving spherical CR structures but not the uniformizability. We apply the theorem to the Figure Eight knot complement in \cite{acosta_spherical_2016}, and we will apply it to the Parker-Will structure in section \ref{sect_ch_cr_wlc} of this paper.

If we use Dehn surgeries to construct spherical CR uniformizations on manifolds there are two main difficult points. The first one is to apply a theorem or to prove the uniformizability of a given structure. The second one is that the two spherical CR Dehn surgery theorems give structures parametrized by the points of an open set of a space of deformations of representations that is not explicit.
Two questions rise then naturally:

\begin{enumerate}
 \item Can we give explicitly an open set of representations giving spherical CR structures on Dehn surgeries of the Whitehead link complement ?
 \item Are these structures uniformizable ?
\end{enumerate}

 The aim of this article is to answer, at least partially and in a particular case, to the the two questions above. We will take as starting point the Parker-Will uniformization of the Whitehead link complement. We will use as space of deformations the representations constructed by Parker and Will in \cite{parker_complex_2017a}, that factor through the group $\z3z3$. This corresponds to consider a slice of the character variety $\mathcal{X}_{\su21}(\z3z3)$ as defined in \cite{acosta_character_2016}. Furthermore, we will be considering the representations of a whole component of the character variety, since Guilloux and Will show in \cite{guilloux_will_2016} that a whole component of the $\sl3c$-character variety of the fundamental group of the Whitehead link complement corresponds only to representations that factor through $\z3z3$.
 
 The slice that we will consider is parametrized by a single complex number $z$, and the representation of the Parker-Will uniformization has parameter $z=3$. However, we will consider the parametrization of representations used by Parker and Will in \cite{parker_complex_2017a}, given by a pair of angles $(\alpha_1, \alpha_2) \in ]-\frac{\pi}{2}, \frac{\pi}{2} [^2$. With these parameters, the representation of the Parker-Will uniformization has parameter $(0, \alpha_2^{\lim})$. For the deformations of the representation having parameter $(0, \alpha_2)$, we show the two following theorems:
   
   \newtheorem*{thm:associativity}{Theorem \ref{thm_ch_dehn_eff_ell}}
\begin{thm:associativity}
Let $n\geq 4$. Let $\rho_n$ be the representation with parameter $(0, \alpha_2)$ such that $z = 8 \cos^2(\alpha_2) = 2\cos(\frac{2\pi}{n}) + 1$ in the Parker-Will parametrization. Then, $\rho_n$ is the holonomy representation of a \CR {} structure on the Dehn surgery of the Whitehead link complement on $T_1$ of type $(1,n-3)$ (i.e. of slope $\frac{1}{n-3}$). 
\end{thm:associativity}

\newtheorem*{thm:associativity2}{Theorem \ref{thm_ch_dehn_eff_lox}}
\begin{thm:associativity2}
 Let $\alpha_2 \in ]0 , \alpha_2^{\lim}[$. Let $\rho$ be the representation with parameter $(0, \alpha_2)$ in the Parker-Will parametrization. Then $\rho$ is the holonomy representation of a \CR {} structure on the Dehn surgery of the Whitehead link complement on $T_1$ of type $(1,-3)$ (i.e. of slope $-\frac{1}{3}$).
\end{thm:associativity2} 

 The corresponding representations have been studied previously by Parker and Will in \cite{parker_complex_2017a} and by Parker, Wang and Xie in \cite{parker_wang_xie}. In these two articles, the authors prove that the groups are discrete using the Poincaré polyhedron theorem in $\h2c$, but they do not identify the topology of the manifolds at infinity. In this article, we give a new proof of this facts, but using different and more geometrical techniques, and we establish the topology of the manifolds at infinity.

 On the one hand, in \cite{parker_complex_2017a}, Parker and Will study a region $\mathcal{Z} \subset ]-\frac{\pi}{2},\frac{\pi}{2}[^2$ parametrizing representations of $\z3z3$ with values in $\su21$. The region is given in Figure \ref{peche2}, and contains the parameters that appear in the statement of Theorem \ref{thm_ch_dehn_eff_lox}.
 On the other hand, in \cite{parker_wang_xie}, Parker, Wang and Xie study the representations that appear in the statement of Theorem \ref{thm_ch_dehn_eff_ell}, since their images are index two subgroups of a $(3,3,n)$ triangle group in $\su21$. They use the Poincaré polyhedron theorem to show that the groups are discrete. The Dirichlet domain used to apply the Poincaré polyhedron theorem is very similar to the domain that we use in this article. However, they do not identify the topology of the manifold at infinity and there is no visible link between the Dirichlet domain of \cite{parker_wang_xie} and the Ford domain of \cite{parker_complex_2017a} that we establish in this article.

\paragraph{Outline of the article}
 This article has three main parts. 
 
 In Part \ref{part_geom_background}, we give the geometric background needed to state and prove the results. We will set notation and describe the complex hyperbolic plane and several objects related to this space in Section \ref{sect_h2c}, and specially in the visual sphere of a point in $\cp2$ in Section \ref{sect_sphere_visuelle}. We will then focus in the definition and properties of the equidistant hypersurfaces of two points, called \emph{bisectors} and their continuation to $\cp2$, called \emph{extors}, as well as some of their intersections in Section \ref{sect_extors_biss_spinal}.
 
 In Part \ref{part_surgeries_wlc}, we consider some spherical CR structures on the Whitehead link complement and on manifolds obtained from it by Dehn surgeries. We recall the spherical CR uniformizations of Schwartz and Parker-Will, and describe a space of deformations of the corresponding holonomy representations. At last, we apply the surgery theorem of \cite{acosta_spherical_2016}, and identify the expected Dehn surgeries that would have a spherical CR structure if the open set of the surgery theorem is large enough.
 
 In Part \ref{part_eff_defor}, which is the core of the article, we give an explicit deformation of the Ford domain in $\h2c$ constructed by Parker and Will in \cite{parker_complex_2017a}, and that is bounded by bisectors. We recall the construction of Parker and Will, that gives the spherical CR uniformization of the Whitehead link complement when restricted to the boundary at infinity. We consider the deformations of the holonomy representation with parameters $(0,\alpha_2)$, and deform the bisectors that border the Ford domain. By studying carefully their intersections, we show that if a particular element $[U]$ in the group is either loxodromic or elliptic of finite order $\geq 4$, then the bisectors border a domain in $\h2c$ with a face pairing. We identify the manifold obtained by restricting the construction to $\dh2c$ as the expected Dehn surgery of the Whitehead link complement. For the parameters for which $[U]$ is an elliptic element of finite order and for some of the parameters for which $[U]$ is loxodromic we apply the Poincaré polyhedron theorem as stated in \cite{parker_complex_2017a}, and show that the spherical CR structures obtained are uniformizable.

 In Section
 \ref{sect_strategie_de_preuve}, we will state the results on surgeries and uniformization and we will give the strategy of the proof. The rest of Part \ref{part_eff_defor} will be devoted to prove these statements. Section \ref{sect_notat_combi_initiale} fixes the notation and describes the construction of Parker and Will in detail. We will prove the statements in Section \ref{sect_fin_preuve}, but admitting some technical conditions that we will prove in the three last sections. We will check the conditions on the faces of the domain: a condition on the topology of the faces in Section \ref{sect_topo_faces}, a local combinatorics condition in Section \ref{sect_combi_locale}, and we will show that the global combinatorics of the intersection of the faces is the expected one in Section \ref{sect_combi_globale}.

\paragraph*{Acknowledgements}
The author would like to acknowledge his advisors
Martin Deraux and Antonin Guilloux, as well as Pierre Will for many discussions about the subject.

\newpage
\part{Geometric background} \label{part_geom_background}

\section{The complex hyperbolic plane and its isometries}\label{sect_h2c}
In this section, we set notation and recall the definition of the complex hyperbolic space $\h2c$ and its boundary at infinity $\dh2c$. We also describe briefly its isometries and the geometric structure modeled on $\dh2c$. The main reference for these objects is the book of Goldman \cite{goldman}.

\subsection{Definition and models}

Throughout this article, we will use objects belonging to complex vector spaces and their projectivizations. For a complex vector space $V$ and a vector $v \in V$, we will denote by $[v]$ its image in $\mathbb{P}V$. The same notation holds for matrix groups. For example, the image of a matrix $M \in \su21$ in the group $\pu21$ will be denoted by $[M]$.

Let $V$ be a complex vector space of dimension $3$.
Let $\Phi$ be a Hermitian form of signature $(2,1)$ on $V$, and define:
\begin{eqnarray*}
 V_{-} &=& \{ v \in V \mid \Phi(v)<0 \} \\
 V_{0} &=& \{ v \in V \mid \Phi(v) = 0 \}
\end{eqnarray*} 

 \begin{defn}
  The complex hyperbolic plane is the space $\h2c = \mathbb{P}V_{-}$ endowed with the Hermitian metric induced by $\Phi$. Its boundary at infinity is the set $\dh2c = \mathbb{P}V_{0}$. We denote by $\con{\h2c}$ the set $\h2c \cup \dh2c \in \cp2$.
 \end{defn}

 The space $\h2c$ is homeomorphic to a ball $B^4$, and $\dh2c$ is homeomorphic to the sphere $\mathbb{S}^3$.
Comment: curvature, topology, S3

\begin{defn}
 If $V = \mathbb{C}^3$ and the Hermitian form $\Phi$ has matrix
 \[ \begin{pmatrix}
 1 & 0 & 0 \\
 0 & 1 & 0 \\
 0 & 0 & -1
\end{pmatrix}  \]
we obtain the ball model. In this model, we identify $\h2c$ and $\dh2c$ as follows:

 \begin{eqnarray*}
  \h2c &=& \left\{ \begin{bmatrix} z_1 \\ z_2 \\ 1 \end{bmatrix} \in \cp2 \mid |z_1|^2 + |z_2|^2 < 1 \right\} \\
  \dh2c &=& \left\{ \begin{bmatrix} z_1 \\ z_2 \\ 1 \end{bmatrix} \in \cp2 \mid |z_1|^2 + |z_2|^2 = 1 \right\}
 \end{eqnarray*}
 
\end{defn}

\begin{defn}
 If $V = \mathbb{C}^3$ and the Hermitian form $\Phi$ has matrix
 \[ \begin{pmatrix}
 0 & 0 & 1 \\
 0 & 1 & 0 \\
 1 & 0 & 0
\end{pmatrix}  \]
we obtain the Siegel model. In this model, we identify $\h2c$ and $\dh2c$ as follows:

 \begin{eqnarray*}
  \h2c &=& \left\{ \begin{bmatrix} -\frac{1}{2}(|z|^2 + w) \\ z \\ 1 \end{bmatrix}  \mid (z,w)\in \CC^2 , \Re(w)<0 \right\} \\
  \dh2c &=& \left\{ \begin{bmatrix} -\frac{1}{2}(|z|^2 + it) \\ z \\ 1 \end{bmatrix} \mid (z,t)\in \CC \times \RR \right\} \cup \left\{ \begin{bmatrix} 1 \\ 0 \\ 0 \end{bmatrix} \right\}
 \end{eqnarray*}
 In this case, we identify $\dh2c$ with $\CC \times \RR \cup \{ \infty \}$.
\end{defn}

In \cite{goldman}, Goldman shows that the totally geodesic subspaces of $\h2c$ are points, real geodesics, copies of $\mathbb{H}^1_{\CC}$, copies of $\mathbb{H}^2_{\mathbb{R}}$ and $\h2c$ itself. The copies of $\mathbb{H}^1_{\CC}$ are the intersections of linear subspaces of $\mathbb{P}V$ with $\h2c$, and are called \emph{complex geodesics}. Notice that given two distinct points of $\h2c$ there is a unique complex geodesic containing them, as well as a unique real geodesic containing them. The boundary at infinity of a complex geodesic is called a \emph{$\CC$-circle}, and the boundary at infinity of a copy of $\mathbb{H}^2_\RR$ is called an \emph{$\RR$-circle}: they are unknotted circles in $\dh2c \simeq S^3$.  The group $\pu21$ acts transitively on each kind of subspace.

\subsection{Isometries}
 The group of holomorphic isometries of $\h2c$ is the projectivized of the unitary group for the Hermitian form $\Phi$: we denote it by $\pu21$. Notice that the definition of the group depends on the choice of $\Phi$, and may change depending on the model that we consider.
 
 We will often consider matrices in the group $\su21$ instead of elements of $\pu21$.
  Every element of $\pu21$ admits exactly three lifts to the group $\su21$ of unitary matrices for $\Phi$ of determinant one. If $\omega$ is a cube root of $1$ and $U \in \su21$, then $U$, $\omega U$ and $\omega^2 U$ are the three lifts of $[U]$ to $\su21$.
 
 As in real hyperbolic geometry, the elements of $\pu21$ are classified by their fixed points in $\h2c$ and $\dh2c$; and we can refine the classification by dynamical considerations.

 \begin{defn}
  An isometry $[U] \in \pu21$ is
     \emph{elliptic} if it fixes a point in $\h2c$, \emph{parabolic} if it is not elliptic and fixes exactly one point in $\dh2c$ and \emph{loxodromic} if it is not elliptic and it has two fixed points in $\dh2c$.
 \end{defn} 
 
 We can state this classification in terms of the eigenvalues and eigenspaces:
 \begin{prop}
  Let $U \in \su21 \setminus \{\mathrm{Id}\}$. Then $U$ is in one of the three following cases:
  \begin{enumerate}
\item $U$ has an eigenvalue $\lambda$ of modulus different from $1$. Then $[U]$ is loxodromic.
\item $U$ has an eigenvector $v \in V_{-}$. Then $[U]$ is elliptic and its eigenvalues have
modulus equal to $1$ but are not all equal.
\item All eigenvalues of $U$ have modulus $1$ and $U$ has an eigenvector $v \in V_{0}$. Then
$[U]$ is parabolic.
  \end{enumerate}
 \end{prop} 
 
 We say that an element is \emph{regular} if it has three different eigenvalues, and unipotent if it is not the identity and has three equal eigenvalues (hence equal to a cube root of $1$).
 This last definition and the proposition above extend easily to $\pu21$.
 
 It is possible to recognize the type of a regular element only by considering its trace, using the following proposition, given by Goldman in \cite{goldman}. Notice that the function $f$ satisfies $f(z) = f(\omega z)$, and hence $f \circ \mathrm{tr}$ is well defined on elements of $\pu21$.
 
 \begin{prop}
  Let
  $f(z) = |z|^4 - 8 \Re(z^3)+18 |z|^2 - 27$.
Let $U \in \su21$. Then $U$ is regular if and only if $f (tr(U )) \neq 0$. Furthermore, if
$f (tr(U )) < 0$ then $[U]$ is regular elliptic, and if $f (tr(U )) > 0$ then $[U]$ is loxodromic.
 \end{prop}

 We recall some dynamical properties of regular elliptic elements, that we need to classify some of them.
 For a detailed description of the dynamics of elements of $\pu21$ on $\dh2c$ see \cite{acosta_spherical_2016}.
 
 A regular elliptic element $[U]$ stabilizes two complex geodesics on $\h2c$ intersecting at the fixed point of $[U]$, and two linked $\CC$-circles in $\dh2c$. In this case, $[U]$ belongs to a one parameter subgroup of $\pu21$: the orbits of such a subgroup are the two stable $\CC$-circles and torus knots turning around the two circles. In some cases, we say that an element is of type $(\frac{p}{n}, \frac{q}{n})$:

 \begin{defn}
  Let $p,q \in \ZZ$ and $n \in \NN^*$ be three relatively prime integers. We say that a regular elliptic element $[U] \in \pu21$ is of type $(\frac{p}{n},\frac{q}{n})$ if $[U]$ is conjugated in the ball model to:
  \[ \begin{bmatrix}
  e^{i \alpha} & 0 & 0 \\
  0 & e^{i \beta} & 0 \\
  0 & 0 & e^{i \gamma}
  \end{bmatrix}
  \]
  with $\alpha - \gamma = \frac{p}{n}$ and $\beta - \gamma = \frac{q}{n}$.
 \end{defn}

We can make two remarks about this definition: 
 \begin{rem}
  I $[U]$ is elliptic of type $(\frac{p}{n},\frac{q}{n})$, there is a one parameter subgroup $([U_s)_{s \in \RR}]$ such that $[U_1] = [U]$. A generic orbit of the subgroup is a torus knot of type $(p,q)$, turning $p$ times around a $\CC$-circle $C_1$ and $q$ times around a second $\CC$-circle $C_2$. The whole orbit is completed in a time $n$ of the flow, so the action of $[U]$ corresponds morally to $\frac{p}{n}$ turns around $C_1$ and $\frac{q}{n}$ turns around $C_2$. Remark also that if $p$ or $q$ equals $\pm 1$, then the torus knot is not knotted.
 \end{rem} 
 
 \begin{rem}
  Not every elliptic element is of some type $(\frac{p}{n}, \frac{q}{n})$. The elements of some type $(\frac{p}{n}, \frac{q}{n})$ are the ones for which the surgery theorem of \cite{acosta_spherical_2016} works, and for which a geometric structure is expected in the deformations that we consider further in this article. 
 \end{rem}

\subsection{Polarity and the box-product}

In order to have a better understanding of the space $\h2c$, we will sometimes use the language of polars and polarity. This language corresponds to a geometric point of view of the orthogonality of the Hermitian form $\Phi$.

\begin{defn}
 Given a point $[u] \in \mathbb{P}V$, let
  \begin{equation*}
  [u]^{\perp} = \mathbb{P} \left\lbrace v \in V\setminus \{0\} \mid \langle u, v\rangle = 0 \right\rbrace .
  \end{equation*}

  It is the projectivized of the orthogonal of $u$ for the Hermitian form $\Phi$. Hence, it is a complex line of $\mathbb{P}V$, called \emph{polar line} of $[u]$.
\end{defn}

 We state some results following immediately from linear algebra considerations and from the fact that the Hermitian form $\Phi$  is non degenerated.

\begin{notat}
 If $[u]$ and $[v]$ are distinct points of $\pv$, we denote by $l_{[u],[v]}$ the complex line passing by $[u]$ and $[v]$.
\end{notat}

\begin{defn}
 Given a complex line $l$ of $\pv$, there is a unique point $[v] \in \pv$ such that $l = [v]^{\perp}$. We say that $[v]$ is the \emph{pole} of $l$, and we denote it by $[v] = l^{\perp}$
\end{defn}

\begin{rem} If $[u],[v] \in \pv$, then :
 \begin{enumerate}
  \item $[u] \in [v]^{\perp} \iff [v] \in [u]^{\perp}$
  \item $[u] \in \dh2c \iff [u] \in [u]^{\perp}$
  \item $[u]^{\perp} \cap \h2c \neq \emptyset \iff [u] \in \pv^{+}$
\end{enumerate}
\end{rem}  
   
%   La notion suivante de triangle autopolaire, est reliée aux éléments réguliers de $\pu21$, et nous l'utiliserons à plusieurs reprises.
   
 \begin{defn}
  Let $[u],[v],[w] \in \pv$ be three non-aligned points. We say that they form an \emph{auto-polar triangle} if the poles of the lines $l_{[u],[v]}$, $l_{[v],[w]}$ and $l_{[w],[u]}$ are precisely the points $[u]$, $[v]$ and $[w]$. 
 \end{defn} 

We state some general remarks about the terms defined above.
  
\begin{rem} \label{rem_polaires}
\begin{enumerate}
Let $[U] \in \pu21$.
  \item The group $\pu21$ is the subgroup of $\mathrm{PGL}_3(\mathbb{C})$ stabilizing $\h2c$ (and hence also $\dh2c$ and $\mathbb{P}V^{+}$).
  \item A point $[Z] \in \cp2$ is fixed by $[U]$ if and only if $Z$ is an eigenvector of $U$.
  \item The elements of $\pu21$ preserve the polarity: if $[u]\in \cp2$, then $[U]([u]^{\perp})= ([U][u])^{\perp}$.
  \item If $l$ is a complex line of $\cp2$, then $l$ is stable by $[U]$ if and only if $[U](l^{\perp})=l^{\perp}$.
  \item If $[u],[v] \in \cp2$ are fixed by $[U]$, then the line $l_{[u],[v]}$ passing by $[u]$ and $[v]$ is stable by $[U]$ ; $l_{[u],[v]}^{\perp}$ is then fixed by $[U]$.
  \item If $[U] \in \pu21$ has exactly three non-collinear fixed points $[u],[v],[w] \in \cp2$, then they form an auto-polar triangle.
 \end{enumerate}
\end{rem}

 We can express the polarity in an algebraic language by using the hermitian cross product, that we define below. It is the complex version of the usual cross product on $\RR^3$. It is briefly described by Goldman in Chapter 2 of \cite{goldman}.
 
 % Commençons par faire une remarque, qui nous sera utile pour la définition.
 
 \begin{rem}
  Let $p,q \in \CC^3$. Let $\psi$ be the linear form
 \begin{equation*}
 \psi :
 \begin{array}{rcl}
  \CC^3 &\rightarrow& \CC \\
  r &\mapsto & \det(p,q,r)
 \end{array}.
\end{equation*}   
  Since the Hermitian form $\Phi$ is non-degenerated, there is a unique vector $s \in \CC^3$ such that $\psi(r) = \langle s, r \rangle$ for all $r \in \CC^3$.
 \end{rem}
 
% Avec cette remarque, nous définissons le produit vectoriel hermitien de la façon suivante:
 
 \begin{defn}
  Let $p,q \in \CC^3$. We define the \emph{Hermitian cross product of $p$ and $q$}, denoted by $p \boxtimes q$, as the unique vector $s \in \CC^3$ such that $\langle s, r \rangle = \det(p,q,r)$ for all $r \in \CC^3$.
 \end{defn} 
 
% La remarque suivante fait le lien entre la définition du produit vectoriel hermitien et la polarité.
 
 \begin{rem}
  Let $p,q \in \CC^3$. If $p$ and $q$ are collinear, then $p \boxtimes q =0$. If not, then $[p \boxtimes q] = l_{[p],[q]}^{\perp}$. Indeed, it is a nonzero vector such that $\langle p , p \boxtimes q \rangle = \langle q , p \boxtimes q \rangle = 0$.
 \end{rem}
 
 For the explicit computations that we make in Part \ref{part_eff_defor},
  we will need the expression of the Hermitian cross product with coordinates. We give this expression for the ball model and for the Siegel model in the two following lemmas, that we obtain immediately by checking the condition $\langle p \boxtimes q , r \rangle = \det(p,q,r)$ for $r$ in the canonical basis of $\CC^3$.
 
 \begin{lemme}
  In the ball model, we have:
  \begin{equation*}
   \begin{pmatrix}
   z_1 \\ z_2 \\ z_3
   \end{pmatrix}
   \boxtimes
   \begin{pmatrix}
   w_1 \\ w_2 \\ w_3
   \end{pmatrix}
   =
   \begin{pmatrix}
   \con{z_2w_3} - \con{z_3w_2} \\ \con{z_3w_1} - \con{z_1w_3} \\ \con{z_2w_1} - \con{z_1w_2}
   \end{pmatrix}
  \end{equation*}
 \end{lemme}
 
  \begin{lemme}
  In the Siegel model, we have:
  \begin{equation*}
   \begin{pmatrix}
   z_1 \\ z_2 \\ z_3
   \end{pmatrix}
   \boxtimes
   \begin{pmatrix}
   w_1 \\ w_2 \\ w_3
   \end{pmatrix}
   =
   \begin{pmatrix}
   \con{z_1w_2} - \con{z_2w_1} \\ \con{z_3w_1} - \con{z_1w_3} \\ \con{z_2w_3} - \con{z_3w_2}
   \end{pmatrix}
  \end{equation*}
 \end{lemme}

\subsection{Spherical CR structures and horotubes}
 We will consider Spherical CR structures on some manifolds in the second and third part of this article. We recall here some definitions and results about these structures as well as the definition of the horotubes, which are geometric objects that model cusps for the structures.

 \begin{defn}
A \CR structure on a manifold $M$ is a $(G,X)$-structure on $M$ for $G = \pu21$ and $X = \dh2c$. That is an atlas of $M$ with charts taking values in $\dh2c$ and transition maps given by elements of $\pu21$. 
 \end{defn}
 
 \begin{rem}
  A $(G,X)$-structure on a manifold $M$ defines a developing map $\Dev : \widetilde{M} \rightarrow X$ and a holonomy representation $\rho : \pi_1(M) \rightarrow G$ such that for all $x \in  \widetilde{M}$ and $g\in \pi_1(M)$ we have $\rho(g)\Dev(x) = \Dev(g \cdot c) $. Remark that the holonomy representation is defined up to conjugation an the developing map up to translation by an element of $G$.
 \end{rem}
 
 \begin{defn}
  Consider a $(G,X)$-structure on $M$ with holonomy $\rho$. Let $\Gamma = \Im(\rho)$. We say that the structure is \emph{complete} if $M \simeq \Gamma \backslash X$. We say that the structure is \emph{uniformizable} if $M \simeq \Gamma \backslash \Omega_\Gamma$, where $\Omega_\Gamma \in X$ is the set of discontinuity of $\Gamma$.
 \end{defn}

\begin{rem}
 The usual condition for a structure is to be complete, which is equivalent to be geodesically complete if $X$ has a complete Riemannian metric. However, when considering \CR structures, there are very few manifolds admitting complete structures since $\dh2c$ is compact. We will consider non-complete structures, and look for uniformizable ones, since they are still intrinsically related to the image of the holonomy representation.
\end{rem}

We will consider further in this paper \CR {} uniformizations on two particular manifolds: two uniformizable structures on the Whitehead link complement, constructed by Schwartz in \cite{schwartz} and by Parker and Will in \cite{parker_complex_2017a} respectively, and a uniformizable structure on the Figure eight knot complement, constructed by Deraux and Falbel in \cite{falbel}. 

 For these three structures, the image of a neighborhood of a cusp by the developing map is a horotube. We recall the definition of this object, which is crucial for attempting to perform \CR {} Dehn surgeries and construct structures on other manifolds, as made in \cite{schwartz} or in \cite{acosta_spherical_2016}.
 
 \begin{defn}
  Let $[P] \in \pu21$ be a parabolic element with fixed point
$[p] \in \dh2c$ . A \emph{$[P]$-horotube} is an open set $H$ of $\dh2c \setminus \{[p]\}$, invariant under
$[P]$ and such that the complement of $H/\langle [P] \rangle$ in ($\dh2c \setminus \{[p]\}) / \langle [P] \rangle$ is compact. We say that a $[P]$-horotube is \emph{nice} if it is invariant by a one-parameter parabolic subgroup of $\pu21$ containing $[P]$.
 \end{defn}

\section{The visual sphere of a point in $\cp2$}\label{sect_sphere_visuelle}
 In this section, we are going to define the visual sphere of a point in $\cp2$ and give coordinates for some charts of this object.
 We will use the visual sphere of a point of $\cp2$ in order to have a better understanding of bisectors and their topology, in Section \ref{sect_extors_biss_spinal}, and to parameter some intersections. We will also use this tool to control the intersections of the faces of the deformed Ford domain that we construct in Part \ref{part_eff_defor}.

\subsection{Definition}
\begin{defn}
Let $[p] \in \cp2$. We call \emph{visual sphere of} $[p]$ the set of complex lines of $\cp2$ passing by $[p]$. We will denote it by $L_{[p]}$. In this way:

\begin{equation*}
L_{[p]} = \left\{ l_{[p],[q]} \mid [q]\in \cp2 \setminus \{[p]\} \right\}.
\end{equation*}
\end{defn}

\begin{rem}
 The space $L_{[p]}$ is isomorphic to $\mathbb{CP}^1$. We can identify it in two other ways. We will often use the abusive language corresponding to the following identifications. On the one hand, the set of the lines passing by $[p]$ is the projectivized of the tangent space to $\cp2$ at $[p]$, hence
 \begin{equation*}
L_{[p]} = \mathbb{P}(T_{[p]} \cp2 ).
\end{equation*}
On the other hand, by considering the dual space, we also have:
\begin{equation*}
L_{[p]} = \left\{ [\varphi] \in \mathbb{P}((\mathbb{C}^{3})^*) \mid \varphi(p) = 0 \right\}.
\end{equation*}

At last, if we have a Hermitian product, $\mathbb{C}^3$ is canonically identified to its dual, and we have:
\begin{equation*}
L_{[p]} = [p]^{\perp}.
\end{equation*}
\end{rem}

\subsection{Coordinates for the visual sphere}
In this subsection, we are going to give coordinates for the visual sphere of a point $[p]$. These coordinates will be useful for making explicit computations in this space. The following proposition gives a way to construct a chart.

\begin{prop}
Let $\varphi_1,\varphi_2 \in (\mathbb{C}^3)^*$ be two independent linear forms and such that $\varphi_1(p) = \varphi_2(p) = 0$. Then the map

  \begin{equation*}
\begin{array}{llll}
f \colon 
  \begin{array}{lll}
L_{[p]} & \rightarrow & \mathbb{CP}^1 \\
l_{[p],[q]} & \mapsto & \frac{\varphi_1(q)}{\varphi_2(q)}
\end{array} 
\end{array}
 \end{equation*}
is well defined and an isomorphism.
\end{prop}

By translating this fact in terms of orthogonality, we obtain the following proposition:

\begin{prop}
Let $[p] \in \cp2$. Let $[p'],[p''] \in [p]^{\perp}$ be two distinct points.  Then the map

  \begin{equation*}
\begin{array}{llll}
f \colon 
  \begin{array}{lll}
L_{[p]} & \rightarrow & \mathbb{CP}^1 \\
l_{[p],[q]} & \mapsto & \frac{\langle p' , q \rangle}{\langle p'' , q \rangle}
\end{array} 
\end{array}
 \end{equation*}
is well defined and an isomorphism.
\end{prop}

\begin{notat}
 From now on, we will denote this map by $\psi_{p',p''} : L_{[p]} \rightarrow \mathbb{CP}^1$.
\end{notat}

 The following remark tells us that by choosing an auto-polar triangle as frame, the computations are easier in the corresponding chart.
\begin{rem}
 In the proposition above, if $[p], [p']$ and $[p'']$ form an auto-polar triangle, then $f([p'']) = 0$ and $f([p']) = \infty$.
\end{rem}

% Finalement, remarquons rapidement le lien entre ces cartes et les birapports de $\ccp1$:
\begin{rem}
 If $[q],[q'] \in \cp2 \setminus \{p\}$, then $\frac{f([q])}{f([q'])}$ is the cross ratio of $[p],[p'],[q]$ and $[q']$.
\end{rem}

\section{Extors, bisectors and spinal surfaces}\label{sect_extors_biss_spinal}

 In this section, we are going to study some objects appearing naturally when constructing Dirichlet or Ford domains in $\h2c$. These objects will be the surfaces equidistant to two points, that we call metric bisectors, and some natural generalizations of them, that we simply call bisectors. In order to study them, we will also study their analytic continuation to $\cp2$, called extors and their intersection with $\dh2c$, called spinal surfaces. In his book \cite{goldman}, Goldman dedicates chapter 5 to the topological study of metric bisectors, chapter 8 to extors, extending bisectors to $\cp2$, and chapter 9 to some intersections of bisectors.
  We will use significantly this study of bisectors and extors. However, we adopt a point of view closer to projective geometry.

\subsection{Definition}

 We begin by defining the objects that we will use to work, beginning by the metric bisectors, which are the equidistant surfaces of two points in $\h2c$.
 
\begin{defn}\label{defn_biss_metrique}
 Let $[p],[q] \in \h2c$ be two distinct points. The \emph{metric bisector} \footnote{In the literature, it is simply called \emph{bisector}. We will use this last term for a more general object, that we will define in Definition \ref{defn_bisector}.} of $[p]$ and $[q]$ is the set
 \[ \mathfrak{B} =
 \left\{ [z] \in \h2c \mid d([z],[p]) = d([z],[q]) \right\}. 
 \]
 If $p,q \in \mathbb{C}^3$ are lifts of $[p]$ and $[q]$ such that $\langle p, p \rangle = \langle q , q \rangle$, then the bisector can be written as
 \[ \mathfrak{B} =
 \left\{ [z] \in \h2c \mid |\langle z , p \rangle | = | \langle z , q \rangle| \right\}. 
 \]
 Its boundary at infinity is a \emph{spinal sphere}.
\end{defn}

 \begin{rem}
  In Chapter 5 of his book, Goldman shows that, topologically, a bisector is a three dimensional ball, and that a spinal sphere is a smooth sphere in $\dh2c$. They are analytic objects, but they are not totally geodesic, since there are no totally geodesic subspaces of $\h2c$ of dimension 3.
 \end{rem} 
 
 We define the extors below. They are objects of $\cp2$ extending the metric bisectors. We keep the terms used by Goldman in \cite{goldman} for this object.
 
\begin{defn}\label{defn_extor}
 Let $[f] \in \cp2$. Let $C$ be a real circle in $L_{[f]}$. The \emph{extor} from $[f]$ given by $C$ is the set
 \[ \mathfrak{E} =
 \left\{ [z] \in \cp2 \mid l_{[f],[z]} \in C \right\}.
 \]
 In this feature, $[f]$ is the \emph{focus} of $\mathfrak{E}$.
\end{defn}

We remark that all extors are projectively equivalent. The following remark gives an explicit link between extors and metric bisectors, and will motivate the study of extors in $\cp2$ and their intersections in order to understand the bisectors and their intersections.

\begin{rem}
 Every metric bisector extends to an extor. If $\mathfrak{B}$ is the metric bisector of $[p]$ and $[q]$, then it extends to an extor with focus $[p \boxtimes q]$, given by 
  \[ \mathfrak{E} =
 \left\{ [z] \in \cp2 \mid |\langle z , p \rangle | = | \langle z , q \rangle| \right\} 
 \]
 if $p$ and $q$ are lifts of $[p]$ and $[q]$ such that $\langle p , p \rangle = \langle q ,q \rangle$.
 The corresponding circle $C \subset L_{[p \boxtimes q]} \simeq [p \boxtimes q]^{\perp}$ is given by 
 \begin{equation*}
 C = \left\{ [p-\alpha q] \mid \alpha \in S^1 \right\}.
\end{equation*}
\end{rem}

\begin{rem}
 The extors are precisely the equidistant surfaces of two points of $\cp2$ when it is endowed with the Fubini-Study metric. A detailed proof can be found in Chapter 8 of \cite{goldman}.
\end{rem}

We will consider, when deforming a Ford domain in Part \ref{part_eff_defor}, some objects defined in the same way, but with fewer restrictions on the points $[p],[q] \in \cp2$. Hence we define the following generalization of the notion of metric bisector, that we study below.

\begin{defn}\label{defn_bisector} Let $p,q \in \mathbb{C}^3 \setminus \{0\}$. We define:
\begin{itemize}
 \item the \emph{extor} of $p$ and $q$ as
 \[\mathfrak{E}(p,q) = \{[z] \in \cp2 \mid |\langle z , p \rangle| = |\langle z , q \rangle| \}.\] 
 \item the  \emph{bisector} of $p$ and $q$ as its intersection with $\h2c$: 
 \[\mathfrak{B}(p,q) = \{[z] \in \h2c \mid |\langle z , p \rangle| = |\langle z , q \rangle| \}.\] 
 \item the  \emph{spinal surface} of $p$ and $q$ as the boundary at infinity of the bisector:
 \[\mathfrak{S}(p,q) = \{[z] \in \dh2c \mid |\langle z , p \rangle| = |\langle z , q \rangle| \}.\] 
\end{itemize}
\end{defn}

 We will limit ourselves to the case where the points $p$ and $q$ defining an extor, a bisector or a spinal sphere have the same norm. It is always the case when they are in the same orbit for a subgroup of $\su21$.
 We can recover the complex lines of the extor $\mathfrak{E}(p,q)$ in the following way:

\begin{prop}\label{prop_param_un_extor}
 Let $p,q \in \mathbb{C}^3 \setminus \{0\}$. The extor $\mathfrak{E}(p,q)$ can be written as a union of complex lines in the following way:
\[
\mathfrak{E}(p,q) = \bigcup_{\alpha \in S^1} [q - \alpha p]^{\perp}.
\]
\end{prop}
\begin{proof}
 A point $[z] \in \cp2$ belongs to $\mathfrak{E}(p,q)$ if and only if $|\langle z , p \rangle| = |\langle z , q \rangle|$. This happens if and only if there exists $\alpha \in S^1$ such that $\alpha \langle z , p \rangle = \langle z , q \rangle$, that is such that $\langle z , q-\alpha p \rangle = 0$. Hence, $[z] \in \mathfrak{E}(p,q)$ if and only if there exists $\alpha \in S^1$ such that $[z] \in [q - \alpha p]^{\perp}$.
\end{proof}

\begin{rem}
 By Proposition \ref{prop_param_un_extor}, the extor of $p$ and $q$ is an extor in terms of Definition \ref{defn_extor}.
 It is clear, by Definition \ref{defn_biss_metrique}, that a metric bisector is a bisector and that a spinal sphere is a spinal surface.
\end{rem}

\begin{rem}
 If we want to define these objects for $[p],[q] \in \cp2$, different choices of lifts will lead to different objects. From now on, we will only consider the case when the lifts satisfy $\langle p , p \rangle =  \langle q , q \rangle$. 

\begin{itemize}
  \item If $\langle p , p \rangle =  \langle q , q \rangle < 0$, then $\mathfrak{B}(p,q)$ and $\mathfrak{S}(p,q)$ are the metric bisector and the spinal sphere of $p,q$ in $\dh2c$ as in Definition \ref{defn_biss_metrique}.
 
 \item If $\langle p , p \rangle =  \langle q , q \rangle = 0$ and there is a preferred element $G \in \su21$ such that $[G][p] = [q]$, we will choose lifts $p$ and $q$ such that $Gp=q$. In this case, $\mathfrak{B}(p,q)$ and $\mathfrak{S}(p,q)$, are a metric bisector and a spinal sphere.
 
 \item If $\langle p , p \rangle =  \langle q , q \rangle > 0$, then $\mathfrak{B}(p,q)$ and $\mathfrak{S}(p,q)$, are sometimes a metric bisector and a spinal sphere. We prove this fact in Proposition \ref{lox_spinal_spheres}.
\end{itemize}
\end{rem}

\begin{prop}\label{prop_extor_equid}
 Let $\mathfrak{E}$ be an extor with focus $[f]$ and let $p \in f^{\perp} \sm0$ such that $[p]^\perp \nsubseteq \mathfrak{E}$. Then there is a unique $q \in \mathbb{C}^3 \setminus \{0\}$, up to multiplication by a unitary complex number, such that $\mathfrak{E} = \mathfrak{E}(p,q)$.
\end{prop} 
\begin{cor}
 Every extor $\mathfrak{E}$ is of the form $\mathfrak{E}(p,q)$.
\end{cor}
In order to prove this proposition, we need the following elementary lemma:
\begin{lemme}\label{lemme_cercle_point_cp1}
 Let $p \in \CC^2 \sm0$ and $C \subset \ccp1$ a circle not containing $[p]$. Then, there is a unique $q \in \CC^2 \sm0$, up to multiplication by a unitary complex number, such that $C = \{[p-\alpha q] \mid \alpha \in S^1\}$.
\end{lemme}
\begin{proof}
Since $[p] \notin C$, we can complete $p$ into a basis $(p,q)$ of $\CC^2$ in such a way that, in the chart $q^{*}$ of $\ccp1$, the circle $C$ is the unit circle. This vector $q$ is unique up to multiplication by a unitary complex number; a change of basis induces a non-trivial similarity of the chart. The circle can be written as $ \{ q^*(p- \alpha q) \mid \alpha \in S^1\}$. We deduce that $C = \{[p-\alpha q] \mid \alpha \in S^1\}$.
\end{proof}

\begin{proof}(of Proposition \ref{prop_extor_equid})

We know that $\mathfrak{E}$ can be written as $\bigcup_{\alpha \in S^1}l_\alpha$, where the complex lines $l_\alpha$ form a circle in $L_{[f]}$. We identify $L_{[f]}$ with $[f]^{\perp} \simeq \ccp1$. Hence, we have a circle $C$ defining $\mathfrak{E}$ and a point $[p]$ that does not belong to $C$. By Lemma \ref{lemme_cercle_point_cp1}, there exists $q \in f^\perp \subset \CC^3$, unique up to multiplication by a unitary complex number, such that $C = \{[p-\alpha q] \mid \alpha \in S^1\}$. We deduce that $\mathfrak{E} = \bigcup_{\alpha \in S^1}l_\alpha = \bigcup_{\alpha \in S^1} [p- \alpha q]^{\perp}$.
A point $[z]$ of $\cp2$ belongs to $\bigcup_{\alpha \in S^1}[p- \alpha q]^{\perp}$ if and only if there exists $\alpha \in S^1$ such that $\alpha \langle z , p \rangle = \langle z , q \rangle$, i.e. if $|\langle z , p \rangle| =  |\langle z , q \rangle|$. We deduce that $\mathfrak{E} = \mathfrak{E}(p,q)$.
\end{proof}

\subsection{Topology of bisectors and spinal surfaces.}
 We are going to study in detail the topology of the objects that we defined above, and we will see that there are three possibilities, depending on the relative position of certain points and $\dh2c$. We begin by defining the complex spine and the real spine of a bisector or an extor, which are a complex and a real line of $\cp2$ and that will help us to understand extors and bisectors.

\begin{defn} Let $\mathfrak{E}$ be an extor with focus $[f]$ and given by the circle $C \subset
 L_{[f]}$. The \emph{complex spine} $\Sigma$ of $\mathfrak{E}$ is the complex line $[f]^{\perp}$. By identifying $L_{[f]}$ with $[f]^{\perp}$, we define the \emph{real spine} of $\mathfrak{E}$ as the real circle $\sigma \subset \Sigma$ corresponding to $C$.
\end{defn}

We make some remarks about this definition.

\begin{rem}\label{rem_spine_reel_extor}
 If $[f] \notin \dh2c$, then its real spine $\sigma$ is the set $\mathfrak{E} \cap \Sigma$.
\end{rem}

\begin{rem}
In the case of metric bisectors, as described by Goldman in \cite{goldman}, the complex and the real spine are the intersections with $\h2c$ of the ones we defined above.
\end{rem}

\begin{rem}
 An extor $\mathfrak{E}$ is determined by its real spine $\sigma$. Indeed, there exists a unique complex line $\Sigma$ containing it: it must be the complex spine. Hence, the focus of $\mathfrak{E}$ is $\Sigma^{\perp}$ and the circle determining $\mathfrak{E}$ is given by $\sigma$. 
\end{rem}

\begin{rem}
 Let $p,q \in \mathbb{C}^3 \sm0$ be two distinct points. The complex spine of the extor $\mathfrak{E}(p,q)$ is the complex line $l_{[p],[q]}$. 
\end{rem}

 In the case that we consider, the real spine of an extor cannot be any real circle of $\cp2$. The following lemma gives a necessary condition for a real circle to be the real spine of one of the extors that we consider.
 \begin{lemme}\label{lemme_echine_orthogonale}
  Let $p,q \in \CC^3 \sm0$ be two distinct points such that $\langle p,p \rangle = \langle q ,q \rangle$. Let $\Sigma$ be the complex spine of $\mathfrak{E}(p,q)$ and $\sigma$ its real spine. Then
 \begin{itemize}
  \item If $\Sigma$ intersects $\h2c$, then $\sigma$ is a circle orthogonal to $\dh2c$ in $\cp2$.
  \item If $\Sigma$ is tangent to $\dh2c$, then $\sigma$ is a circle containing $[p \boxtimes q]$
\end{itemize}   
 \end{lemme}
\begin{proof}
 We begin by the first case. If $\alpha \in \CC$ has modulus 1, we know that 
\begin{eqnarray*} 
 |\langle p , p + \alpha q \rangle| 
 &=& |\langle p ,p \rangle + \alpha\langle p ,q \rangle| \\
 &=&
 |\langle q ,q \rangle + \con{\alpha}\langle q ,p \rangle| \\
  &=&
  |\langle q , p + \alpha q \rangle| 
\end{eqnarray*} 
  We can hence parametrise the real spine by $\sigma = \Sigma \cap \mathfrak{E}(p,q) = \{ [p + \alpha q] \mid \alpha \in S^1 \}$. In the chart $\lambda \mapsto [p + \lambda q]$ of $\Sigma$, the points of $\dh2c$ are given by the equation
  \begin{equation*}
   (|\lambda|^2 +1) \langle p ,p \rangle + 2 \Re (\lambda \langle p,q \rangle) = 0
  \end{equation*}
  It is the equation of a circle (or a line passing by $0$ if $\langle p ,p \rangle = 0$) which is orthogonal to the unit circle, since if $\lambda$ is a solution of the equation, then $\con{\frac{1}{\lambda}}$ is also a solution.
  
  For the second point, it is enough to check that if $\Sigma$ is tangent to $\dh2c$, then at least a point of $l_{[p],[q]}$ different from $[p \boxtimes q]$ belongs to the extor $\mathfrak{E}(p,q)$. Since $l_{[p],[q]}$ is tangent to $\dh2c$, the restriction of the Hermitian form to $l_{[p],[q]}$ is degenerated. Its determinant in the basis $(p,q)$ is equal to $\langle p ,p \rangle \langle q,q \rangle - \langle p, q \rangle \langle q , p \rangle$ ; we deduce that $| \langle p,p \rangle| = |\langle p , q \rangle|$ and that $[p] \in \mathfrak{E}(p,q)$.
\end{proof}

\subsubsection{Two decompositions of extors}
 Following the description given by Goldman in Chapters 5 and 8 of his book \cite{goldman}, we give here two decompositions of extors that will be useful later. There are the slice decomposition, in complex lines, and the meridional decomposition, in real planes.
 
\begin{prop} \emph{(Slice decomposition)}
 Let $\mathfrak{E}$ be an extor with focus $[f]$. Then the complex lines contained in $\mathfrak{E}$ passing by $[f]$ form a foliation of  $\mathfrak{E} \setminus \{[f]\}$ and are the only complex lines contained in $\mathfrak{E}$. If, moreover, $[f]\notin \dh2c$ and $\mathfrak{E}$ admits $\Sigma$ as complex spine and $\sigma$ as real spine, they are precisely the complex lines orthogonal to $\Sigma$ at the points of $\sigma$.
\end{prop}

\begin{proof}
 Let $l$ be a complex line contained in $\mathfrak{E}$. Let $[p]\in l$ be a point different from $[f]$. We know that the lines $l$ and $l_{[f],[p]}$ are contained in $\mathfrak{E}$. But an extor is a smooth sub-manifold of dimension 3 of $\cp2$ besides its focus: the tangent space at $[p]$ to $\mathfrak{E}$ is hence of real dimension 3. It contains a maximal holomorphic subspace of complex dimension 1, which must be at the same time the tangent space of $l$ and of $l_{[f],[p]}$ at $[p]$. We deduce that $l = l_{[f],[p]}$, which proves the first assertion. 
 
 By Remark \ref{rem_spine_reel_extor}, we know that the lines contained in  $\mathfrak{E}$ are precisely the lines passing by $[f]$ and a point of $\sigma$. Since the lines passing by $[f] = \Sigma^{\perp}$ are precisely  the orthogonal lines to $\Sigma$, this shows the second assertion.
\end{proof}

\begin{defn}
 Such a complex line is called a \emph{slice} of $\mathfrak{E}$. This decomposition is called \emph{the slice decomposition of the extor}. We can also consider it on the corresponding bisector.
\end{defn}

 The other decomposition given by Goldman is the meridional decomposition, given in the following proposition. For a complete proof, see Section 8.2.3 of \cite{goldman}, or Theorem 5.1.10 of \cite{goldman}.

\begin{prop} \emph{(Meridional decomposition)}
 Let $\sigma$ be a real circle in $\cp2$. Then the union of the real planes of $\cp2$ containing $\sigma$ form a singular foliation of the extor of real spine $\sigma$.
\end{prop}

 \begin{defn}
  Such a real plane is called a \emph{meridian} of $\mathfrak{E}$; the associated decomposition is the \emph{meridional decomposition} of $\mathfrak{E}$, that we can also consider on the corresponding bisector.
 \end{defn}

 We are now going to describe the topology of bisectors and spinal surfaces, but only in the case of a bisector of the form $\mathfrak{B}(p,q)$ with $p$ and $q$ with the same norm. A general description is not more difficult, but we would need to consider some extra cases depending on the relative position of the real spine and $\dh2c$.
 
\subsubsection{Metric bisectors and spinal spheres.}
We begin by describing the metric bisectors, which are the usual bisectors studied in detail by Goldman in \cite{goldman}. We characterize them by their real spine in the following proposition.
 
\begin{prop}\label{prop_carac_biss_metr_echine}
 A bisector is a metric bisector if and only if its focus belongs to $\cp2 \setminus \con{\h2c} $ and its real spine is orthogonal to $\dh2c$.
\end{prop}
\begin{proof}
 Consider a metric bisector $\mathfrak{B}(p_0 , q_0)$, where $p_0 , q_0 \in \CC^3$ satisfy $\langle p_0,p_0 \rangle = \langle q_0,q_0 \rangle < 0$.
 By Lemma \ref{lemme_echine_orthogonale}, we know that the real spine of $\mathfrak{B}(p_0 , q_0)$ is a circle orthogonal to $\dh2c$, and that its focus $[f_0] = [p_0 \boxtimes q_0]$ belongs to $\cp2 \setminus \con{\h2c}$. This shows the first side of the equivalence. Consider now a bisector $\mathfrak{B}$ with focus $[f] \in \cp2 \setminus \con{\h2c}$ and whose real spine $\sigma$ is orthogonal to $\dh2c$. Since $\pu21$ acts transitively on $\cp2 \setminus \con{\h2c}$, we can suppose that $[f] = [f_0] = [p_0 \boxtimes q_0]$.

  Furthermore, remark that the stabilizer of $[f]$ is isomorphic to $\mathrm{PU}(1,1)$, and acts $2$-transitively on the circle $\dh2c \cap [p]^{\perp}$.
 Hence it also acts transitively on the circles orthogonal to $\dh2c \cap [p]^{\perp}$. Hence there exists an element $[G] \in \pu21$ such that $[G]\sigma$ is the real spine of $\mathfrak{B}(p_0,q_0)$; we deduce that $[G]\mathfrak{B} = \mathfrak{B}(p_0,q_0)$ and that $\mathfrak{B} = \mathfrak{B}(G^{-1}p_0 , G^{-1}q_0)$ is a metric bisector.
\end{proof}

 With the preceding proof, we can make the following remark:
\begin{rem}\label{rem_carac_biss_metr_sommets}
 A metric bisector is determined by the two points of the intersection of its real spine with $\dh2c$. These two points are called the \emph{vertices of the bisector} by Goldman in \cite{goldman}. Furthermore, since $\pu21$ acts 2-transitively on $\dh2c$, it acts transitively on the metric bisectors.
\end{rem}

% At last, we can show that a bisector is homeomorphic to a 3-dimensional ball, and that its boundary at infinity is a smooth sphere in $\dh2c$.
 
\begin{prop}
 A metric bisector is homeomorphic to a 3-dimensional ball. A spinal sphere is a smooth sphere $\dh2c$.
\end{prop}
\begin{proof}
Let $\mathfrak{B}$ be a metric bisector with focus $[f]$ and real spine $\sigma$.
By Proposition \ref{prop_carac_biss_metr_echine}, we know that $\sigma \cap \h2c$ is an open interval. Furthermore, we have
 $\mathfrak{B} = \bigcup_{[s] \in \sigma \cap \h2c} l_{[f],[s]} \cap \h2c$,
  which is the product of an open interval and a disk, and hence is homeomorphic to a 3-dimensional ball. By Remark \ref{rem_carac_biss_metr_sommets}, we know that $\pu21$ acts transitively on the metric bisectors. We can hence study the boundary at infinity of a particular metric bisector in order to complete the proof. Consider the bisector with vertices 
\begin{equation*}
\begin{array}{rcl}
 \begin{bmatrix}
 1 \\ 0 \\0
 \end{bmatrix} 
 &\text{ and }&
 \begin{bmatrix}
 0\\0\\1
 \end{bmatrix}
 \end{array}
\end{equation*}

in the Siegel model. The corresponding spinal surface is then given by:
\begin{equation*}
\mathfrak{S} = 
 \left\{ 
\begin{bmatrix}
-\frac{1}{2}|z|^2 \\z \\1
\end{bmatrix} 
\mid z \in \CC
  \right\}
  \cup
  \left\{
  \begin{bmatrix}
  1\\0\\0
  \end{bmatrix}
  \right\},
\end{equation*}
which is a smooth sphere in $\dh2c$.
\end{proof}

\subsubsection{Fans}
 We are going to describe now other types of bisectors, that are not necessarily metric bisectors but that will be useful to construct fundamental domains on $\h2c$ for the actions of some subgroups of $\pu21$. We will see the fans an the Clifford cones. We begin by defining and describing a fan.

\begin{defn}
 We call \emph{fan} a bisector whose focus $[f]$ belongs to $\dh2c$ and that does not contain $[f]^{\perp}$ as a slice. We call \emph{spinal fan} its boundary at infinity.
\end{defn}

\begin{prop}
 A fan is homeomorphic to a 3-dimensional ball. A spinal fan is a smooth sphere with a singular point in $\dh2c$.
\end{prop}
\begin{proof}
 Let $\mathfrak{E}$ be an extor with focus $[f] \in \dh2c$.
 We work on the Siegel model and suppose, without lost of generality, that $[f] = \begin{bmatrix}
 1 \\ 0 \\0
 \end{bmatrix} $. Without lost of generality, suppose that the $\RR$-plane of points with real coordinates is a meridian of $\mathfrak{E}$. The slices of the bisectors are hence of the form $T_r$ for $r \in \RR$, where
 \begin{equation*}
 T_r = \{[f]\} \cup \left\{\begin{bmatrix}
 z \\ r \\1
 \end{bmatrix} \mid z \in \mathbb{C}\right\}.
 \end{equation*}
 Hence, the bisector $\mathfrak{E} \cap \h2c$ is diffeomorphic to the set $\{(z,r) \in \CC \times \RR \mid 2\Re (z) < r^2\}$, which diffeomorphic to a  3-dimensional ball.
 Its boundary at infinity is the following set, which is a smooth sphere besides the point $[f]$, where it has a singularity:
 \begin{equation*}
 \left\{\begin{bmatrix}
 1 \\ 0 \\0
 \end{bmatrix}\right\} \cup \left\{\begin{bmatrix}
 -\frac{1}{2}(r^2 + it) \\ r \\ 1
 \end{bmatrix} \mid (r,t) \in \RR ^2 \right\}.
 \end{equation*}
\end{proof}

\subsubsection{Clifford cones and tori}
 For the last type of bisectors, which are Clifford cones, we use partially the terms given by Goldman in his book \cite{goldman}. Nevertheless, it seems preferable to us to call Clifford torus the boundary at infinity of a Clifford cone, even if Goldman keeps this name for the intersections of extors.

\begin{defn}
 We call \emph{Clifford cone} a bisector whose focus belongs to $\h2c$. We call \emph{Clifford torus} its boundary at infinity.
\end{defn}

\begin{prop}
 A Clifford cone is homeomorphic to $(S^1 \times D^2)/(S^1 \times \{0\})$. A Clifford torus is a smooth torus in $\dh2c$.
\end{prop}
\begin{proof}
Let $\mathfrak{E}$ be an extor with focus $[f] \in \h2c$. Every complex line of $\cp2$ passing by $[f]$ intersects $\h2c$ in a complex geodesic, which is homeomorphic to a disk $D^{2}$, and has a $\mathbb{C}$-circle as boundary at infinity. By the slice decomposition of $\mathfrak{E}$, we know that $(\mathfrak{E} \setminus \{[f]\}) \cap \h2c$ is homeomorphic to $S^{1} \times (D^{2}\setminus \{[f]\})$, and that its boundary at infinity is homeomorphic to $S^1 \times S^1$.
\end{proof}

Putting together those definitions and considering the relative position of two points $[p]$ and $[q]$ and of the pole $[p\boxtimes q]$ of the line $l_{[p],[q]}$, we obtain the following proposition:
\begin{prop}\label{lox_spinal_spheres}
 Let $p,q \in \mathbb{C}^3 \sm0$ be non collinear and such that $\langle p , p \rangle =  \langle q , q \rangle $. Let $r = \langle p , p \rangle \langle q , q \rangle -  \langle p , q \rangle \langle q , p \rangle$. We suppose that $\mathfrak{B}(p,q) \neq \emptyset$.
 \begin{itemize}
  \item If $r<0$, then $\mathfrak{B}(p,q)$ is a bisector and $\mathfrak{S}(p,q)$ is a spinal sphere.
  \item If $r=0$, then $\mathfrak{B}(p,q)$ is a fan and $\mathfrak{S}(p,q)$ is a spinal fan.
  \item If $r>0$, then $\mathfrak{B}(p,q)$ is a Clifford cone and $\mathfrak{S}(p,q)$ is a Clifford torus.
\end{itemize}  
\end{prop}

\begin{proof}
  Denote by $\Phi '$ the restriction of the Hermitian form to $\mathrm{Vect}(p,q)$. In the basis $(p,q)$, its determinant equals $r$. 
 
 If $r<0$, the Hermitian form $\Phi '$ has signature $(1,1)$. The focus of $\mathfrak{B}(p,q)$, $[p \boxtimes q]$, belongs hence to $\cp2 \setminus \con{\h2c}$. The intersection of an extor with focus in $\cp2 \setminus \con{\h2c}$ with $\h2c$ is a metric bisector.
 
   If $r = 0$, the Hermitian form $\Phi '$ is degenerated. The focus of $\mathfrak{B}(p,q)$, $[p \boxtimes q]$, belongs hence to $\dh2c$. The intersection of an extor with focus in  $\dh2c$ with $\h2c$ is a fan.
   
   If $r>0$, the Hermitian form $\Phi '$ has signature $(2,0)$. The focus of $\mathfrak{B}(p,q)$, $[p \boxtimes q]$, belongs hence to $\h2c$. The intersection of an extor with focus in $\h2c$ with $\h2c$ is a cone over its boundary at infinity, which is a Clifford torus.
\end{proof}

\subsection{From the visual sphere}
 We are going to state two facts about bisectors and some particular visual spheres. We take the following notation for the natural projection on a visual sphere:
 \begin{notat}
  Let $[p]\in \cp2$.We denote by 
    \begin{equation*}
\begin{array}{rcl}
\pi_{[p]} & \colon &
  \begin{array}{rcl}
    \cp2 \setminus \{ [p] \} & \rightarrow & L_{[p]} \\
    \left[ q \right] & \mapsto & l_{[p],[q]}
  \end{array} 
\end{array}
 \end{equation*}
the natural projection on the visual sphere $L_{[p]}$.
 \end{notat}

  The first remark follows from the definition of an extor by a focus and a circle in the visual sphere:

\begin{rem}
 Let $p,q \in \mathbb{C}^3 \setminus \{0\}$ be two distinct points such that $\langle p,p \rangle = \langle q ,q \rangle $. Then $\pi_{[p \boxtimes q]} (\mathfrak{E}(p,q))$ is a circle in $L_{[p \boxtimes q]}$. Furthermore, $\pi_{[p \boxtimes q]} (\mathfrak{B}(p,q))$ is an arc of a circle or a circle in $L_{[p \boxtimes q]}$, depending on whether $[p\boxtimes q] \in  \con{\h2c}$ or not; the set $\pi_{[p \boxtimes q]} (\mathfrak{S}(p,q))$ is its closure. 
 \end{rem}

 The following proposition describes the projection on the visual sphere $L_{[p]}$ of a bisector defined as $\mathfrak{B}(p,q)$ or of its corresponding spinal surface.

 \begin{prop}\label{prop_proj_biss_sphere_visuelle}
 Let $p,q \in \mathbb{C}^3 \sm0$ be non-collinear and such that $\langle p , p \rangle =  \langle q , q \rangle $. Let $r = \langle p , p \rangle \langle q , q \rangle -  \langle p , q \rangle \langle q , p \rangle$.
 \begin{itemize}
  \item If $r \neq 0$, then $\pi_{[p]}(\mathfrak{S}(p,q))$ is a closed disk with centres $l_{[p],[q]}$ and $l_{[p],[p \boxtimes q]}$.
  \item If $r=0$, then $\pi_{[p]}(\mathfrak{S}(p,q))$ is a closed disk whose boundary contains $l_{[p],[q]}$.
\end{itemize}
 In the two cases $\pi_{[p]} (\mathfrak{B}(p,q))$ is the interior of $\pi_{[p]} (\mathfrak{S}(p,q))$. 
\end{prop}
\begin{proof}
 First, notice that a complex line that intersects $\mathfrak{B}(p,q)$ also intersects $\mathfrak{S}(p,q)$, and that the complex lines that intersect $\mathfrak{S}(p,q)$ also intersect $\mathfrak{B}(p,q)$ unless they are tangent to $\mathfrak{S}(p,q)$. This will show the last point.

 Suppose at first that $r \neq 0$. In this case, $(p,q,p \boxtimes q)$ is a basis of $\mathbb{C}^3$. The stabilizer of $([p],[q])$ in $\pu21$ is then isomorphic to $S^1$ and its elements can be written in the basis $(p,q,p \boxtimes q)$ in the form
\begin{equation*}
\begin{bmatrix}
 e^{i\theta} & 0& 0 \\
 0 & e^{i\theta} & 0 \\
 0 & 0 & e^{-2i\theta}
\end{bmatrix}. 
 \end{equation*}
 These elements stabilize the bisector $\mathfrak{B}(p,q)$ and, since they fix $[p]$, act in a natural way on $L_{[p]}$. In the chart $\psi_{q , p\boxtimes q}$, the action is given by a rotation centered at $0$.
 In the chart sending $l_{[p],[q]}$ to $0$ and $l_{[p],[p \boxtimes q]}$ to $\infty$, the projection $\pi_{[p]}(\mathfrak{S}(p,q))$ is a compact set, connected, invariant by the rotations of $\mathbb{C} \cup \{\infty\}$. It only remains to check that exactly one of the points $0$ and $\infty$ belongs to the image of $\pi_{[p]}$. This corresponds to check that only one line among $l_{[p],[q]}$ and $l_{[p],[p \boxtimes q]}$ intersects $\mathfrak{B}(p,q)$.
 
 \paragraph{First case: $r<0$}
 Let $\alpha \in \mathbb{C}$ be of modulus $1$ such that $\langle p, \alpha q \rangle \in \mathbb{R}^-$. On the one hand, we know that $\langle p+\alpha q, p+\alpha q \rangle = 2(\langle p,p \rangle + \langle p, \alpha q \rangle ) < 0$ since $r<0 $. On the other hand, $|\langle p, p+\alpha q \rangle| = |\langle p, p \rangle -|\langle p, q \rangle|| = |\langle q, q \rangle -|\langle p, q \rangle|| = |\langle q, p+\alpha q \rangle |$. Hence we know that $l_{[p],[q]}$ intersects $\mathfrak{B}(p,q)$.
 
 Let us show that $l_{[p],[p \boxtimes q]}$ does not intersect $\mathfrak{S}(p,q)$. We know that $[p\boxtimes q] \notin \con{\h2c}$. Let $\lambda \in \mathbb{C}$. We have:
 \begin{eqnarray*}
  |\langle p + \lambda p \boxtimes q , p \rangle | &=&|\langle p , p \rangle | \\
  |\langle p + \lambda p \boxtimes q , q \rangle | &=&|\langle p , q \rangle | 
 \end{eqnarray*}
 Since $r<0$, the two quantities are different, hence $l_{[p],[p \boxtimes q]}$ does not intersect $\mathfrak{S}(p,q)$.
 
 \paragraph{Second case: $r > 0$}
 In this case, we know that $l_{[p],[q]}$ does not intersect $\con{\h2c}$, and hence does not intersect $\mathfrak{S}(p,q)$. Furthermore, $[p \boxtimes q] \in \h2c$ and $\langle p , p\boxtimes q \rangle = \langle p , p\boxtimes q \rangle = 0$, hence $l_{[p],[p \boxtimes q]}$ intersects $\mathfrak{B}(p,q)$.
 
 \paragraph{Third case: $r=0$}
 Suppose now that $r=0$. In this case, $[p\boxtimes q] \in \dh2c$ and $[p], [q]$ and $[p\boxtimes q]$ are collinear. The fixing subgroup of $l_{[p],[q]}$ is then a vertical unipotent subgroup isomorphic to $\mathbb{R}$, that acts on $L_{[p]}$ as a translation. Furthermore, since $l_{[p],[q]}$ is tangent to $\dh2c$ at $[p \boxtimes q]$, $\pi_{[p]}(\mathfrak{B}(p,q)) \setminus \{l_{[p][q]}\}$ is connected and invariant by one translation direction. %Comme $\langle p ,p \rangle > 0$, il existe une droite complexe passant par $[p]$ et qui n'intersecte pas $\h2c \cup \dh2c$. L'image de $\mathfrak{B}(p,q)$ n 'est donc pas $L_{[p]}$ tout entier. 
 
 Consider coordinates in the Siegel model. By multiplying by an element of $\su21$, we can assume that 
 
\begin{eqnarray*} 
 p = \begin{pmatrix}
 0 \\ 1\\0
\end{pmatrix}
&\text{ and }&
p \boxtimes q = \begin{pmatrix}
 1 \\ 0\\0
\end{pmatrix}.
\end{eqnarray*}
 
 In this case, $q$ can be written in the following way, where $\theta \in \RR$:
\begin{equation*} 
 q = \begin{pmatrix}
 -1 \\ e^{i\theta}\\0
\end{pmatrix}
.
\end{equation*}
 We choose $\frac{z_1}{z_3}$ as coordinate in $\mathcal{L}_{[p]}$ for the complex line passing by $[p]$ and $\begin{pmatrix}
 z_1 \\ 1 \\ z_3
\end{pmatrix}$. With these coordinates, the fixing subgroup of $l_{[p],[q]}$ acts on $\mathcal{L}_{[p]}$ by translations of the form $z \mapsto z + it$, where $t\in \RR$. In order to determine the image of $\mathfrak{B}(p,q)$ by $\pi_{[p]}$, it is enough to determine the set of $s \in \RR$ such that a point of the form $[w_s] = \begin{bmatrix}
 sz_3 \\ 1 \\ z_3
\end{bmatrix}$ belongs to $\con{\mathfrak{B}}(p,q)$.
We compute then: $\langle p , w_s \rangle = 1$ and $\langle q , w_s \rangle = e^{i\theta} - z_3$. In order to have $[w_s] \in \mathfrak{B}(p,q)$, it is then necessary that $z_3$ could be written in the form $z_3(\phi) = e^{i\theta} + e^{i\phi}$, where $\phi \in \RR$.

 But $\langle w_s, w_s \rangle = 1 + 2s|z_3|^2$, and $0\leq |z_3|^2 \leq 4$, where the two equalities hold when $\phi = \pi+\theta$ and $\phi = \theta$ respectively. We deduce that there exists $\phi \in \RR$ such that $ 1 + 2s|z_3(\phi)|^2 \leq 0$ if and only if $s \leq -\frac{1}{8}$.
 The image of $\con{\mathfrak{B}}(p,q)$ by $\pi_{[p]}$ is hence equal to
 \[
 \left\{ s+it \mid (s,t)\in \RR^2 , s \leq -\frac{1}{8} \right\} \cup \{ \infty \}.
 \]
 It is hence a closed disk whose boundary contains $\infty$, which is the coordinate of $l_{[p],[q]}$.
\end{proof}

\subsection{Real visual diameter of a metric bisector}
 We are going to consider here the real visual diameter of a bisector of the form $\mathfrak{B}(p,q)$ seen from $[p]$. In order to control the intersections of certain bisectors it will be useful to understand this angular diameter. It will be of a crucial for the construction of Dirichlet domains in Part \ref{part_eff_defor}.

 \begin{prop}\label{prop_diam_ang_reel_bisector}
  Let $[p] , [q] \in \h2c$. Let $\theta_{\max}$ be the real angular diameter of $\mathfrak{B}(p,q)$ seen from $[p]$.  Then 
\begin{equation*} 
 \cos\left(\frac{\theta_{\max}}{2}\right) = \tanh \left(\frac{d([p],[q] )}{2}\right).
 \end{equation*}
 \end{prop}
 
  \begin{proof}
   We work in the ball model, and let $r = \frac{1}{2} d([p],[q])$. After translating by an element of $\su21$, we can suppose that
   \begin{equation*}
 \begin{array}{rclcrclrccl}   
    [p] &=& \begin{bmatrix}
   0 \\ 0 \\ 1
\end{bmatrix} 
&\text{ and }&
[q] &=& \begin{bmatrix}
   \sinh(r) \\ 0 \\ \cosh(r)
\end{bmatrix}. 
&\text{ In this case, } &
p \boxtimes q &=& \begin{pmatrix}
   0 \\ \sinh(r) \\ 0 \\
\end{pmatrix}.
\end{array}
  \end{equation*}
   
   Let $g_0$ be the real geodesic passing by $[p]$ and $[q]$. We search for a real geodesic $g$ passing by $[p]$ and a point $[z]$ of $\con{\mathfrak{B}}(p,q)$ forming a maximal angle with $g_0$. Hence we can suppose that $[z] \in \mathfrak{S}(p,q)$.
  The points of $\con{\mathfrak{B}}(p,q)$ can be written in the form $[p + e^{i\phi} q + \lambda p \boxtimes q]$, where $\phi \in \RR$ and $\lambda \in \CC$. The points with this form belong to $\dh2c$ if and only if
  $|\lambda|^2 = \frac{2}{\sinh^2(r)}(1 + \cos(\phi)\cosh(r))$.
  In particular, $\cos(\phi) \geq \frac{-1}{\cosh(r)}$. Hence the points $[z] \in \mathfrak{S}(p,q)$ can be written in the form:
  \begin{equation*}
   [z] = \begin{bmatrix}
   e^{i\phi}\sinh(r) \\
   e^{i\psi} \sqrt{2 + 2\cos(\phi)\cosh(r)} \\
   1 + e^{i\phi} \cosh(r)
   \end{bmatrix}
  \end{equation*}
   where $\psi \in \RR$.  
  We want to compute now the tangent vector to the real geodesic passing by $[p]$ and $[z]$. In order to do it, we parametrize the geodesic. Normalize first $p$ and $q$, to have $\langle p , q \rangle$ and $\langle z, p \rangle \in \RR$.
  From now on, we choose as lifts of $[p]$ and $[q]$ the vectors 
  \begin{equation*}
  \begin{array}{rclcrcl}
    p &=& \begin{pmatrix}
   0 \\ 0 \\ 1 + e^{i \phi}\cosh(r)
\end{pmatrix} 
&\text{ and }&
q &=& (1 + e^{i \phi}\cosh(r)) \begin{pmatrix}
   \sinh(r) \\ 0 \\ \cosh(r)
\end{pmatrix}.
\end{array}
  \end{equation*}
  
  With this normalization, the unit tangent vector to $g_0$ at $[p]$ is equal to:
  
\begin{equation*}  
    u_0 = \frac{1}{\sqrt{1  + 2\cos(\phi)\cosh(r)+ \cosh^2(r))}} \begin{pmatrix}
  1 + e^{i \phi}\cosh(r) \\ 0 \\ 0
\end{pmatrix} 
\end{equation*}
  
  Notice that, if $t \in \RR$, we have  $\langle p + tz , p + tz \rangle = -(2t+1)(1  + 2\cos(\theta)\cosh(r)+ \cosh^2(r))$.
   For $t\in \RR$, let  $v_t = \frac{1}{\sqrt{- \langle p + tz , p + tz \rangle}} (p + tz)$. It is a parametrization of the real geodesic, with normalized vectors of norm $-1$.
   We compute 
\begin{equation*}
u_1 = 
\frac{\partial v_t}{\partial t} \restriction_{t=0}
 = 
 \frac{1}{\sqrt{1  + 2\cos(\phi)\cosh(r)+ \cosh^2(r))}}
 \begin{pmatrix}
   \sinh(r)e^{i\phi} \\
    \sqrt{2(1 + \cos(\phi)\cosh(r))}e^{i \psi} \\ 
    0
\end{pmatrix} 
\end{equation*}

In this case, we have $\langle u_1 , u_1 \rangle = 1$.
If $\frac{\theta}{2}$ is the angle between $u_0$ and $u_1$, it satisfies
$\cos(\frac{\theta}{2}) = \Re(\langle u_0 , u_1 \rangle)$.
Now,
\begin{equation*}
\Re(\langle u_0 , u_1 \rangle)
=
\Re \left(\frac{1 + e^{i\phi}\cosh(r)}{|1 + e^{i\phi}\cosh(r)|^2} e^{-i\phi} \sinh(r) \right) \\
=
 \frac{\sinh(r)(\cosh(r) + \cos(\phi))}{1  + 2\cos(\phi)\cosh(r)+ \cosh^2(r)}
\end{equation*}

 This real part is minimal if $\cos(\phi)$ is minimal. Since $\cos(\phi) \geq \frac{-1}{\cosh(r)}$, we deduce that
 \begin{equation*}
\Re(\langle u_0 , u_1 \rangle)
\geq
 \frac{\sinh(r)(\cosh(r) - \frac{1}{\cosh(r)})}{ \cosh^2(r) -1}
 =
 \tanh(r)
\end{equation*}
 
 Remark that this bound is reached when $[z]$ is one of the two ends of the real spine of $\mathfrak{B}(p,q)$. The maximal angle between two real geodesics passing by $[p]$ and by points of $\con{\mathfrak{B}}(p,q)$ is hence the angle $\theta_{\max}$ between the geodesics passing by the ends of the real spine. Since $g_0$ is the bisector of these two real geodesics, we have
\begin{equation*}    
    \cos\left(\frac{\theta_{\max}}{2}\right) = \tanh \left(\frac{d([p],[q] )}{2}\right).
    \end{equation*}
  \end{proof}
 
 In a more precise way, we will use the following corollary, giving an upper bound for the angular diameter that is easy to compute. 
 
 \begin{cor}\label{cor_diam_ang_reel_bisector}
  Let $p , q \in \CC \sm0$ be two distinct points satisfying $\langle p , p \rangle = \langle q ,q \rangle < 0$. If $\frac{\langle p, q \rangle \langle q ,p \rangle}{\langle p , p \rangle \langle q , q \rangle} > 4$, then the angular diameter of  $\mathfrak{B}(p,q)$ seen from $[p]$ is $< \frac{\pi}{3}$.
 \end{cor}
 
 \begin{proof}
  If $r = \frac{d([p],[q])}{2}$, we have $\frac{\langle p, q \rangle \langle q ,p \rangle}{\langle p , p \rangle \langle q , q \rangle} = \cosh^2(r)$. If $\cosh^2(r) > 4$, then $\tanh^2(r) = 1 - \frac{1}{\cosh^2(r)} > \frac{3}{4}$. Denoting by $\theta_{\max}$ the real angular diameter of $\mathfrak{B}(p,q)$ seen from $[p]$, we have $\cos(\frac{\theta_{\max}}{2}) > \frac{\sqrt{3}}{2}$, and hence $\theta_{\max} < \frac{\pi}{6}$.
 \end{proof}

\subsection{Pairs of extors}
 We are going to consider some intersections of bisectors in order to study some Ford domains and their deformations. We consider here, in the same way that Goldman does in Chapter 8 of his book \cite{goldman}, the intersections of the corresponding extors before the intersections of bisectors. We begin by classifying the pairs of extors.

 \begin{defn}
 Let $\mathfrak{E}_1$ and $\mathfrak{E}_2$ be two extors with respective foci $[f_1]$ and $[f_2]$. We say that the pair $(\mathfrak{E}_1,\mathfrak{E}_2)$ is:
 \begin{itemize}
  \item \emph{Confolcal} if $[f_1] = [f_2]$.
  \item \emph{Balanced} if $[f_1] \neq [f_2]$ and $l_{[f_1],[f_2]} \subset \mathfrak{E}_1 \cap \mathfrak{E}_2 $.
  \item \emph{Semi-balanced} if $[f_1] \neq [f_2]$ and $l_{[f_1],[f_2]}$ is contained in exactly one of the two extors.
  \item \emph{Unbalanced} if $[f_1] \neq [f_2]$ and $l_{[f_1],[f_2]}$ is not contained in any of the two extors.
 \end{itemize}
 \end{defn}
 
 \begin{defn}
 We say that a pair of extors $(\mathfrak{E}_1,\mathfrak{E}_2)$ is \emph{coequidistant} if there exist distinct points $p,q,r \in \mathbb{C}^3 \setminus \{0\}$ such that $\mathfrak{E}_1 = \mathfrak{E}(p,q)$ and $\mathfrak{E}_2 = \mathfrak{E}(p,r)$.
 \end{defn}

 \begin{rem}
  A coequidistant pair of extors is either confocal or unbalanced.
 \end{rem}
 \begin{proof}
 Let $p,q,r \in \mathbb{C}^3 \setminus \{0\}$ be three distinct points. Let $\mathfrak{E}_1 = \mathfrak{E}(p,q)$ and $\mathfrak{E}_2 = \mathfrak{E}(p,r)$. If $[p],[q]$ and $[r]$ are collinear, then $\mathfrak{E}_1$ and $\mathfrak{E}_2$ have the same complex spine and hence the same focus. If not, let $[f_1] = [p \boxtimes q]$ and $[f_2]=[p \boxtimes r]$ be the foci of $\mathfrak{E}_1$ and $\mathfrak{E}_2$ respectively. On the one hand; it is trivial that $\langle f_1 , p \rangle = 0$, and on the other hand $\langle f_1 , r \rangle = \langle p \boxtimes q , r \rangle \neq 0$ since $p,q$ and $r$ are not collinear. We deduce that $[f_1] \notin \mathfrak{E}_2$. In the same way, $[f_2] \notin \mathfrak{E}_1$, which concludes the proof.
 \end{proof}

 When considering Ford domains and their deformations, we will only work with coequidistant bisectors with respect to normalized lifts.

 \begin{rem}
  An unbalanced pair of bisectors is coequidistant.
 \end{rem}
 \begin{proof}
  Let $(\mathfrak{E}_1,\mathfrak{E}_2)$ an unbalanced pair of extors. Let $[f_1]$ be the focus of $\mathfrak{E}_1$ and $[f_2]$ be the focus of $\mathfrak{E}_2$. Since $[f_1]$ and $[f_2]$ are distinct, then $[f_1]^{\perp} \cap [f_2]^{\perp}$ is reduced to a point in $\cp2$. Let $[p]$ be this point and $p \in \mathbb{C}^{3} \setminus \{0\}$ a lift. By Proposition \ref{prop_extor_equid}, there exist $q,r \in \mathbb{C}^{3} \setminus \{0\}$ such that $\mathfrak{E}_1 = \mathfrak{E}(p,q)$ and $\mathfrak{E}_2 = \mathfrak{E}(p,r)$.
 \end{proof}

 \begin{prop}
  Let $(\mathfrak{E}_1,\mathfrak{E}_2)$ be a confocal pair of distinct extors with focus $[f]$. Then $\mathfrak{E}_1 \cap \mathfrak{E}_2$ is either $\{[f]\}$, or a complex line passing by  $[f]$, or two complex lines passing by $[f]$. 
 \end{prop} 
 \begin{proof}
  The extors $\mathfrak{E}_1$ and $\mathfrak{E}_2$ are given by two real circles $C_1$ and $C_2$ of $L_{[f]}$. The intersection is hence reduced to $\{[f]\}$ if $C_1 \cap C_2 = \emptyset$, it is a complex line if $C_1$ and $C_2$ are tangent, and it is formed by two complex lines if $C_1$ and $C_2$ intersect at two points.
\end{proof}  

 The two following results describe the intersection of a balanced or semi-balanced pair of extors. See Chapter 8 of \cite{goldman} for detailed proofs. 
 
  \begin{thm}(Theorem 8.3.1 of \cite{goldman})\label{thm_inter_extors_equilibres}
  Let $(\mathfrak{E}_1, \mathfrak{E}_2)$ be a balanced pair of extors. Then there exist a complex line $l$ and a real plane $P$ of $\cp2$ such that $\mathfrak{E}_1 \cap \mathfrak{E}_2 =  P \cup l$.
 \end{thm}
 
 \begin{prop}(Section 8.3.3 of \cite{goldman})
  Let $(\mathfrak{E}_1, \mathfrak{E}_2)$ be a semi-balanced pair of extors, where $\mathfrak{E}_1$ contains $[f_2]$ and $\mathfrak{E}_2$ does not contain $[f_1]$. Then $\mathfrak{E}_1 \cap \mathfrak{E}_2 \setminus \{[f_2]\}$ is a real cylinder in $\cp2$, that can be compactified by adding $[f_2]$.
 \end{prop}

 The following proposition describes the intersection in $\cp2$ of two extors of an unbalanced pair. Goldman calls this intersection a "Clifford torus", but we chose to keep this term for the boundary at infinity of a Clifford cone.
 
 \begin{prop}\label{prop_param_inter_extors_deseq}
 Let $(\mathfrak{E}_1,\mathfrak{E}_2)$ be an unbalanced pair of extors. Then $\mathfrak{E}_1 \cap \mathfrak{E}_2$ is a torus in $\cp2$. If $\mathfrak{E}_1 = \mathfrak{E}(p,q)$ and $\mathfrak{E}_2 = \mathfrak{E}(p,r)$, then the intersection is parametrized by $[(q-\alpha p)\boxtimes (r- \beta p)]$ where $(\alpha,\beta) \in S^1 \times S^1$.
 \end{prop} 
 \begin{proof}
  Let $[f_1]$ and $[f_2]$ be the respective foci of $\mathfrak{E}_1$ and $\mathfrak{E}_2$.
  We know that $\mathfrak{E}_1$ can be written as a union of complex lines passing by $[f_1]$ and parametrized by $S^1$. We can therefore write:
  \begin{equation*}
   \mathfrak{E}_1 = \bigcup_{\alpha \in S^1} l_\alpha \text{ and, in the same way, } \mathfrak{E}_2 = \bigcup_{\beta \in S^1} l'_\beta.
  \end{equation*}
Since the pair of extors is unbalanced, we know that each $l_\alpha$ intersects each $l_{\beta}'$ exactly at one point. Hence, we have:
 \begin{equation*}
  \mathfrak{E}_1 \cap \mathfrak{E}_2 = \bigcup_{\alpha \in S^1}\bigcup_{\beta \in S^1} l_{\alpha} \cap l'_\beta ,
 \end{equation*}
 which is a torus. If $\mathfrak{E}_1 = \mathfrak{E}(p,q)$ and $\mathfrak{E}_2 = \mathfrak{E}(p,r)$, then, by Proposition \ref{prop_param_un_extor}, the complex lines $l_{\alpha}$ can be written as $[q-\alpha p]^{\perp}$ and the lines $l'_{\beta}$ as $[r-\beta p]^{\perp}$. In this case, $l_{\alpha} \cap l'_{\beta} = [q-\alpha p]^{\perp} \cap [r-\beta p]^{\perp} = \{[(q-\alpha p)\boxtimes (r- \beta p)]\}$.
 \end{proof}
 
\subsection{Pairs of coequidistant bisectors} \label{subsect_pairs_coeq_biss}

 From now on, we will consider unbalanced pairs of bisectors defined by normalized lifts, and we will be interested by their intersections.

 \begin{lemme}[Theorem 9.1.2 of \cite{goldman}]
  Let $\mathfrak{E}_1$ and $\mathfrak{E}_2$ be two extors with foci $[f_1]$ and $[f_2]$ respectively. Assume that they intersect at a point $[x_0] \neq [f_1], [f_2]$. Then, either the intersection $\mathfrak{E}_1 \cap \mathfrak{E}_2$ is transverse at $[x_0]$, or $\mathfrak{E}_1$ and $\mathfrak{E}_2$ have a common slice passing by $[x_0]$.
 \end{lemme}
 \begin{proof}
  The extor $\mathfrak{E}_1$ is a smooth real sub-manifold of $\cp2$ besides its focus. Its tangent space at $[x_0]$ is a real space of dimension 3. Hence it admits a maximal holomorphic subspace of complex dimension 1. This subspace is the tangent space to the line $l_{[f_1],[x_0]}$. It is the same for $\mathfrak{E_2}$. We deduce that either the intersection $\mathfrak{E}_1 \cap \mathfrak{E}_2$ is transverse at $[x_0]$, or $l_{[f_1],[x_0]} = l_{[f_2],[x_0]}$ and $\mathfrak{E}_1$ and $\mathfrak{E}_2$ have a common slice passing by $[x_0]$.
 \end{proof}

 \subsubsection{Goldman intersections}
In Chapter 9 of \cite{goldman}, Goldman considers pairs of bisectors coequidistant from points of $\h2c$ or $\dh2c$, and shows that their intersection is connected and a topological disk. We recall here some of the results, obtained by a study of tangencies of spinal spheres.

\begin{prop}[Lemma 9.1.5 of \cite{goldman}]
 Let $\mathfrak{B}_1$ and $\mathfrak{B}_2$ be two metric bisectors, with boundaries at infinity $\mathfrak{S}_1$ and $\mathfrak{S}_2$ respectively. Then, each connected component of $\mathfrak{S}_1 \cap \mathfrak{S}_2$ is a point or a circle, and each connected component of $\mathfrak{B}_1 \cap \mathfrak{B}_2$ is a disk.
\end{prop}

 We give a particular name to those disks, that will often appear in the constructions of Ford or Dirichlet domains.
 
\begin{defn}
 We call a \emph{Giraud disk} such a disk in the intersection of bisectors. We call a \emph{Giraud circle} its boundary at infinity.
\end{defn}

 At last, the following theorem ensures us that the intersection of two coequidistant metric bisectors is either a point or a Giraud disk, and that the intersection of the corresponding spinal spheres is either a point or a Giraud circle.
 
\begin{thm}[Theorem 9.2.6 of \cite{goldman}]
 Let $p,q,r \in \CC \sm0$ such that $\langle p,p \rangle = \langle q,q \rangle = \langle r,r \rangle \leq 0$. Then $\mathfrak{S}(p,q) \cap \mathfrak{S}(p,r)$ is connected.
\end{thm}

 \subsubsection{Other intersections}
We will need to consider more general intersections in order to deform a Ford domain. We describe an explicit example that will be useful later. It is a very symmetrical case, where the bisectors are equidistant from points in the same real plane, and have an order 3 symmetry. We will need the following lemma for a technical point about a sign in the proposition that we show below.

 \begin{lemme}\label{lemme_signe_prod_herm_sym}
  Let $p,q,r \in \CC^3 \sm0$. Assume that they belong to the same $\RR$-plane and that there exists $S \in \su21$ of order 3 such that $Sp = q$ and $Sq = r$. Then  $\langle p \boxtimes q, q \boxtimes r \rangle = \langle q \boxtimes r, r \boxtimes p \rangle = \langle r \boxtimes p, p \boxtimes q \rangle \in \RR^-$.
 \end{lemme}

\begin{proof}
 Since $p,q,r$ belong to the same $\RR$-plane, the points $p \boxtimes q$, $q \boxtimes r$ and $r \boxtimes p$ are also in the same $\RR$-plane. By the symmetry of order 3, we know that $\langle p \boxtimes q , p \boxtimes q \rangle = \langle q \boxtimes r , q \boxtimes r \rangle = \langle r \boxtimes p , r \boxtimes p \rangle$. Denote by $l\in \RR$ this quantity. On the other hand, we also know that $\langle p \boxtimes q, q \boxtimes r \rangle = \langle q \boxtimes r, r \boxtimes p \rangle = \langle r \boxtimes p, p \boxtimes q \rangle \in \RR$. Denote by $k$ this quantity.
 
 Consider the generic case, where $k \neq 0$ and $(p \boxtimes q , q \boxtimes r , r \boxtimes p)$ is a basis of $\CC^3$; the result will follow in the general case by density. In this basis, the matrix of the Hermitian form is:
\[
\begin{pmatrix}
l & k & k \\
k & l & k \\
k & k & l
\end{pmatrix}.
\] 
 It admits a double eigenvalue equal to $l-k$ and a simple eigenvalue equal to $l + 2k$. Since the Hermitian form has signature $(2,1)$, we deduce that $l-k > 0$ and $l+2k < 0$, which implies $k<0$.
\end{proof}

 \begin{prop}\label{prop_inter_biss_sym}
  Let $p,q,r \in \CC^3 \sm0$. Assume that they belong to the same $\RR$-plane and that there exists $S \in \su21$ of order 3 such that $Sp = q$ and $Sq = r$. We know that $p \boxtimes q$, $q \boxtimes r$ and $r \boxtimes p$ have the same norm and belong to the same $\RR$-plane.
 By Lemma \ref{lemme_signe_prod_herm_sym}, we know that $\langle p \boxtimes q, q \boxtimes r \rangle = \langle q \boxtimes r, r \boxtimes p \rangle = \langle r \boxtimes p, p \boxtimes q \rangle \in \RR^-$.
 Let $u = \frac{\langle p \boxtimes q, p \boxtimes q \rangle}{\langle p \boxtimes q, q \boxtimes r \rangle}$.
 
Then
 \begin{itemize}
  \item If $u  < \frac{2}{3}$, then $\mathfrak{B}(p,q) \cap \mathfrak{B}(p,r)$ is a disk, and its boundary at infinity is a smooth circle.
  \item If $u  = \frac{2}{3}$, then $\mathfrak{B}(p,q) \cap \mathfrak{B}(p,r)$ is a disk, and its boundary at infinity consists of three $\CC$-circles.
  \item If $u  > \frac{2}{3}$, then $\mathfrak{B}(p,q) \cap \mathfrak{B}(p,r)$ is a torus minus two disks, and its boundary at infinity consists of two smooth circles.
\end{itemize}  
 \end{prop}

 \begin{proof}
  The intersection of the extors $\mathfrak{E}(p,q)$ and $\mathfrak{E}(p,r)$ is the torus parametrized by:
 
  \begin{equation*}
  \{ (q+e^{i\theta}p)\boxtimes(r + e^{i \phi}p) \mid (\theta, \phi) \in [\pi, \pi]^2 \}
\end{equation*}   
  The points of the intersection $\mathfrak{B}(p,q) \cap \mathfrak{B}(p,r)$ are exactly the points of this torus with negative norm.
  We compute the norm of $(q + e^{i\theta}p) \boxtimes (r + e^{i \phi}p)$:
  
 \begin{align*}
   \langle (q + e^{i\theta}p) \boxtimes (r + e^{i \phi}p),& (q + e^{i\theta}p) \boxtimes (r + e^{i \phi}p) \rangle &
   \\
  &= \langle q \boxtimes r + e^{i\theta} p \boxtimes r + e^{i\phi} q \boxtimes p , q \boxtimes r + e^{i\theta} p \boxtimes r + e^{i\phi} q \boxtimes p \rangle
   \\
 & = \langle  q \boxtimes r , q \boxtimes r \rangle + \langle  p \boxtimes r , p \boxtimes r \rangle + \langle  q \boxtimes p , q \boxtimes p \rangle
  \\
  &+ 2 \Re (
   e^{i \theta}\langle  q \boxtimes r , p \boxtimes r \rangle
  + e^{i \phi}\langle  q \boxtimes r , q \boxtimes p \rangle
  + e^{i (\phi - \theta)}\langle  p \boxtimes r , q \boxtimes p \rangle
  )
  \\
  &= 3 \langle  q \boxtimes r , q \boxtimes r \rangle
   + 2\langle  q \boxtimes r , p \boxtimes r \rangle (\cos(\theta) + \cos(\phi) + \cos(\phi - \theta))
    \\
    &= 2\langle  q \boxtimes r , p \boxtimes r \rangle
     (\frac{3}{2}u
     +
     \cos(\theta) + \cos(\phi) + \cos(\phi - \theta))
\end{align*}

Since $\langle  q \boxtimes r , p \boxtimes r \rangle < 0$, the sign of the expression above is the opposite of the sign of $\frac{3}{2}u + \cos(\theta) + \cos(\phi) + \cos(\phi - \theta)$.
 The level sets of the function $(\theta,\phi) \mapsto \cos(\theta) + \cos(\phi) + \cos(\phi - \theta)$ are traced in Figure \ref{fig_lignes_niveau_biss}. The set describing $\mathfrak{B}(p,q) \cap \mathfrak{B}(p,r)$ is hence given by the level sets of level $\geq -\frac{3}{2}u$, which are
\begin{itemize}
 \item a disk with smooth boundary if $u < \frac{2}{3}$
 \item a disk with boundary the circles of equations $\theta = 0$ , $\phi = 0$ and $\phi - \theta = \pi \mod 2\pi$ if $u = \frac{2}{3}$
 \item The torus minus two disks if $u > \frac{2}{3}$.
\end{itemize} 
  \end{proof}

\begin{figure}[htbp]
 \center
 \includegraphics[width = 6.5cm]{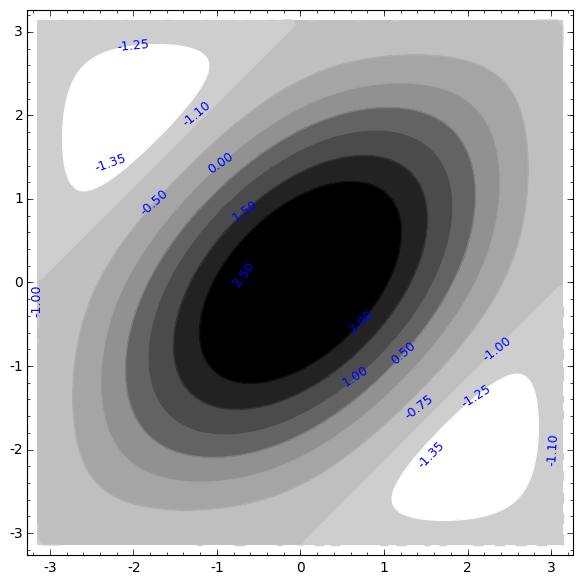}
 \caption{Level sets of the function $(\theta,\phi) \mapsto \cos(\theta) + \cos(\phi) + \cos(\phi - \theta)$ for $(\theta, \phi) \in [-\pi,\pi]^2$.} \label{fig_lignes_niveau_biss}
\end{figure}  
  
\newpage

\newpage
\part{Surgeries on the Whitehead link complement} \label{part_surgeries_wlc}
\section{Surgeries on the Whitehead link complement}\label{sect_ch_wlc}
 In this section, we will consider some \CR {} structures on the Whitehead link complement, and the Dehn surgeries of this link admitting a \CR {} structure. We will use the results of R. Schwartz, that can be found in his book \cite{schwartz} and of Parker and Will, given in the article \cite{parker_complex_2017a}.
 
\subsection{The Whitehead link complement} \label{subsect_description_wlc}
 The Whitehead link is the link given by the projection of Figure \ref{dessin_entrelacs_whitehead}. It has two components and a minimal crossing number of 5. Each component in an unknotted circle.
 
 \begin{figure}[H]
 \begin{center}
   \includegraphics[width=4cm]{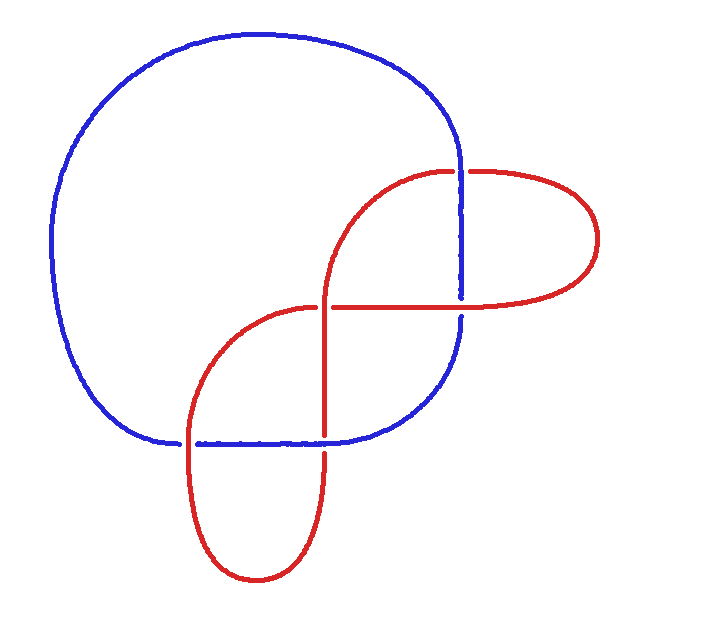}
   \caption{The Whitehead link (SnapPy)} \label{dessin_entrelacs_whitehead}
 \end{center}
 \end{figure}

\begin{rem}
 If we denote by $W$ the Whitehead link and $W'$ link obtained by exchanging its two components, then $W$ and $W'$ are isotopic. In other words, the two components play the same role. This fact will be reflected on the Parker-Will \CR {} structure, that we will see below.
  \end{rem} 

 %  \begin{notat}
 We will denote by $WLC$ the Whitehead link complement in $S^3$. The complement of a tubular neighbourhood of the link in $S^3$ is a compact manifold with two torus boundaries that we denote by $T_1$ and $T_2$. Its interior is homeomorphic to $WLC$ ; we will identify $\con{WLC}$ with $WLC \cup T_1 \cup T_2$.
% \end{notat}
 %
 The fundamental group of $WLC$ is given by the following presentation: 
 
 \[\pi_1(WLC) = \langle u,v \mid [u,v][u,v^{-1}][u^{-1},v^{-1}][u^{-1},v] \rangle \]
 
Choosing as generators $s,t$ satisfying $u=st$ and $v=tst$, we obtain a new presentation: 
 
\[\pi_1(WLC) = \langle s,t \mid t s^{-1} t^{-3} s^{-2} t^{-1} s t^3 s^2 \rangle \] 
 
 \begin{rem}
  This presentation is the one given by SnapPy, with $t=a$ and $s^{-1}=b$.
 \end{rem}

 In this presentation, the couples longitude-meridian of the peripheral subgroups corresponding to $T_1$ and $T_2$ are given by: 
 
 \[ l_1 = t^{-2} s^{-1}ts^2 t^{-1} s^{-1} \hspace{2cm} m_1 = t^{-2}s^{-1} \]
 \[ l_2 = ststs^{-1}t^3s^{-1}t \hspace{2cm} m_2 = st \]
 
 \begin{rem}
  With this marking, the laces $m_i$ correspond to the actual meridians of the components of $WLC$ in $S^3$ and the longitudes are trivial in homology.
 \end{rem} 
 
 Notice, as Parker and Will do in \cite{parker_complex_2017a}, that by imposing $s^3 = t^3 = 1$, the group we obtain is the free product $\z3z3$ ; $\pi_1(WLC)$ admits a surjection onto this group.

\subsection{Deformation spaces}
 We are going to consider representations of $\pi_1(WLC)$ that factor through the quotient $\z3z3$, up to conjugacy.
 In \cite{acosta_character_2016}, we showed that the character variety $\mathcal{X}_{\sl3c}(\z3z3)$ has 16 irreducible components: 15 isolated points and an irreducible component $X_0$.
   In their article \cite{guilloux_will_2016}, Guilloux and Will show that the component $X_0$ is also an irreducible component of the character variety $\mathcal{X}_{\sl3c}(\pi_1(WLC))$. We will limit ourselves to this component $X_0$, and to its intersection with the character variety
$\mathcal{X}_{\su21}(\z3z3)$. This gives us a whole component of deformations of representations of $\pi_1(WLC)$ with values in $\su21$, considered up to conjugacy.

 This space can be parametrized by traces. 
More precisely, if $\z3z3 = \langle s,t \mid s^3 , t^3 \rangle$, the traces of $s,s^{-1},t,t^{-1},st,(st)^{-1},s^{-1}t,st^{-1}$ and of the commutator $[s,t]$ determine, up to conjugacy, an irreducible representation of $\z3z3$ into $\sl3c$. 
In Section 4 of his article \cite{will_generateurs}, Will considers the restriction to $\su21$ or $\mathrm{SU}(3)$. Notice that if $U \in \su21$, then $\mathrm{tr}(U^{-1}) = \con{\mathrm{tr}(U)}$; hence we will only consider the traces of $s,t,st,s^{-1}t$ and the commutator $[s,t]$.
  Furthermore, in the component $X_0$, the images of the elements $s$ and $t$ are regular elliptic of order 3, and hence have trace $0$. For $\rho \in \mathrm{Hom}(\z3z3, \su21)$, denote by $z_\rho = \mathrm{tr}(\rho(st))$, $w_\rho = \mathrm{tr}(\rho(st^{-1}))$ and $x_\rho = \mathrm{tr}(\rho([s,t]))$.
 As detailed in \cite{acosta_character_2016}, 
   the union of the two character varieties $\mathcal{X}_{\su21}(\z3z3)$ and $\mathcal{X}_{\mathrm{SU}(3)}(\z3z3)$ in $X_0$ is described with this coordinates by: 
\begin{equation*}
\{ (z,w,x) \in \mathbb{C}^3 \mid x + \con{x} = Q(z,w) , x\con{x} = P(z,w) \}
 \end{equation*}
where
\begin{eqnarray*}
 Q(z,w) & = & |z|^2 + |w|^2 - 3 \\
 \text{ and }
 P(z,w) & = & 2\Re(z^3) + 2\Re(w^3) + |z|^2|w|^2 - 6|z|^2 -6|w|^2 + 9.
\end{eqnarray*}

\subsection{Parker-Will representations}

 In their article \cite{parker_complex_2017a}, Parker and Will construct a two parameter family of representations $\rho$ with values in $\su21$, in such a way that $\rho(s)$ and $\rho(t)$ are regular elliptic of order $3$ and $\rho(st)$ and $\rho(ts)$ are unipotent. In terms of traces, they parameter the slice of $X_0 \cap \mathcal{X}_{\su21}(\z3z3)$ for which one of the coordinates equals $3$, i.e. a subset of
 \begin{equation*}
 \{ (3,w,x) \in \mathbb{C}^3 \mid , x + \con{x} = Q(3,w) , x\con{x} = P(3,w) \}.
\end{equation*}

 In the Siegel model, this family is explicitly parametrized by $(\alpha_1 , \alpha_2) \in ]-\frac{\pi}{2},\frac{\pi}{2}[^2$, in the following way:
 \begin{equation*}
\rho(s) = e^{-i\alpha_1/3} \left( \begin{matrix}
e^{i \alpha_1} & x_1e^{i\alpha_1 - i\alpha_2} & -1 \\
-x_1e^{i\alpha_2} & -e^{i\alpha_1} & 0 \\
-1 & 0 & 0
\end{matrix} \right),
%\hspace{0.2cm}
\rho(t) = e^{i\alpha_1/3} \left( \begin{matrix}
0 & 0 & -1 \\
0 & -e^{-i \alpha_1} & -x_1e^{-i\alpha_1 - i\alpha_2} \\
-1 & x_1e^{i\alpha_2} & e^{-i\alpha_1}
\end{matrix} \right)
 \end{equation*}
 
 where $x_1 = \sqrt{2\cos(\alpha_1)}$. Since $\rho(s)^3 = \rho(t)^3 = \mathrm{Id}$, we have:
\begin{equation*}  
   \rho(l_1) = \rho(ts^{-1}ts^{-1}ts^{-1}) \hspace{2cm} \rho(m_1) = \rho(ts^{-1}) 
   \end{equation*}
   \begin{equation*}
 \rho(l_2) = \rho(ststst) \hspace{2cm} \rho(m_2) = \rho(st) 
 \end{equation*}
 
 Hence $\rho(l_1) = \rho(m_1)^3$ and $\rho(l_2) = \rho(m_2)^3$ in the whole family of representations. Furthermore, by construction, $\rho(m_2)$ is unipotent for all the representations parametrized by Parker and Will. 
For the other peripheral representation, the type of  $\rho(m_1) = \rho(ts^{-1})$ is given by the curve of Figure \ref{peche1}. If $(\alpha_1,\alpha_2)$ is on the curve, then $\rho(ts^{-1})$ is parabolic, if it is inside, $\rho(ts^{-1})$ is loxodromic, and if it is outside, elliptic. We denote these regions by $\mathcal{L}$ and $\mathcal{E}$ respectively. By setting $\alpha_2^{\lim} = \arccos(\sqrt{\frac{3}{8}})$, the curve has two singular points at $(0, \pm \alpha_2^{\lim})$, for which $\rho(ts^{-1})$ is unipotent.
  Moreover, Parker and Will define the region $\mathcal{Z}$ by the equation 
 \begin{equation*}
 D(4\cos^2(\alpha_1) , 4\cos^2(\alpha_2)) > 0,
\end{equation*} 
  where $D(x,y) = x^3y^3 - 9x^2y^2-27xy^2+81xy-27x-27$. They show, using the Poincaré polyhedron theorem, that the image of $\rho$ is discrete and faithful in the interior of the region $\mathcal{Z}$.
 
 \begin{figure}[htbp]
\center
\begin{subfigure}{0.4\textwidth}
 \includegraphics[width=6cm]{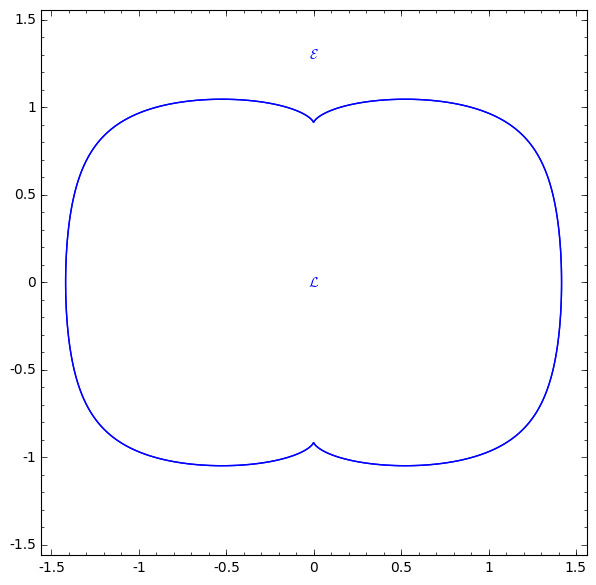}
   \caption{Curve of unipotents}\label{peche1}
 \end{subfigure} \hspace{1cm}
 \begin{subfigure}{0.4\textwidth}
 \includegraphics[width=6cm]{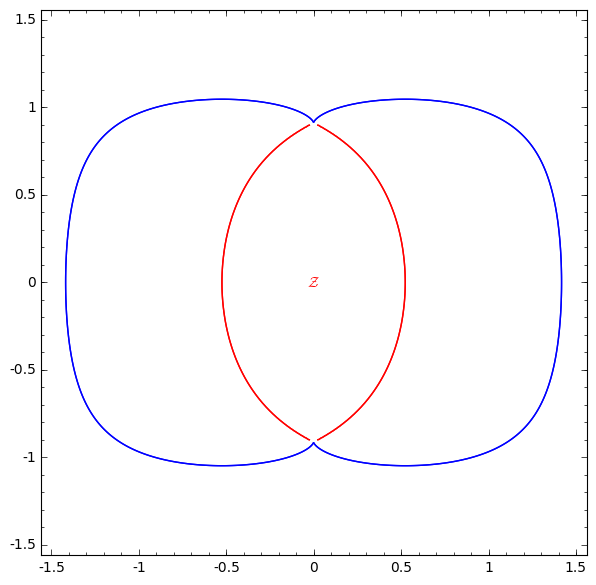}
   \caption{Region $\mathcal{Z}$}\label{peche2}
 \end{subfigure} 
\caption{Curve of unipotents and the region $\mathcal{Z}$ in the Parker-Will slice }
\end{figure}

 \begin{rem}
 Considering coordinates $(z,w,x)$ for the character variety $\mathcal{X}_{\su21}(\z3z3)$ described above, the projection on $w$ of the slice $z = 3$ is a double cover besides the red curve in Figure \ref{fig_peche_traces}, for which the fibres are singletons.

With the parametrization of Parker-Will, the image of the map $\rho \mapsto z_\rho$ is the union of the three lobes of Figure \ref{fig_peche_traces}. The type of $\rho(s^{-1}t)$ is then determined by the sign of $f(z_\rho)$, where $f(z) = |z|^4 - 8 \Re (z^3 ) + 18 |z|^2 - 27$ is the function defined by Goldman in \cite[Theorem 6.2.4]{goldman}. The regions $\mathcal{E}$ and $\mathcal{L}$ are then separated by the blue curve of Figure \ref{fig_peche_traces}, with equation $f(z) = 0$.
\end{rem}

   \begin{figure}[H] 
 \begin{center}
   \includegraphics[width=4cm]{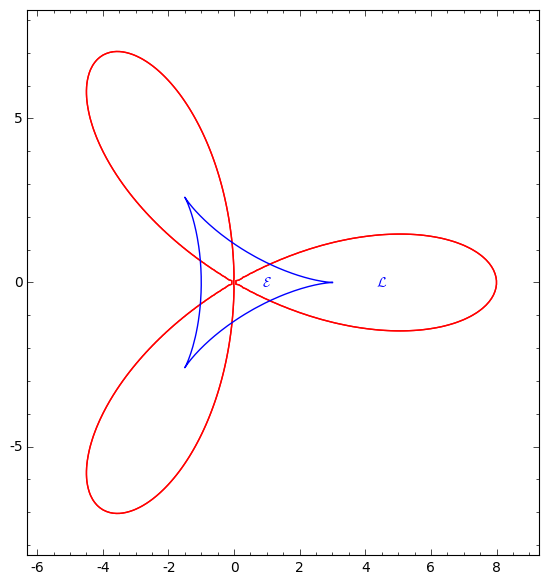}
   \includegraphics[height=4cm]{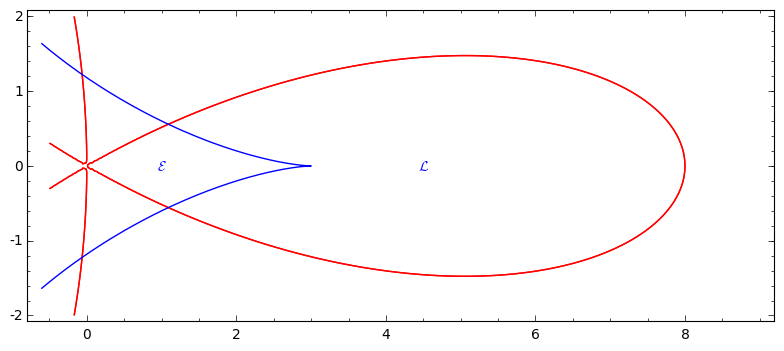}
   \caption{The Parker-Will slice in trace coordinates.}\label{fig_peche_traces}
 \end{center}
 \end{figure}

\subsection{Spherical CR structures with parabolic peripheral holonomy}
 We are going to consider \CR {} structures on the Whitehead link complement. Whenever a structure has a parabolic peripheral holonomy, we are going to deform it in order to obtain \CR {} structures either on Dehn surgeries of $WLC$ or on manifolds obtained similarly, by gluing a torus knot complement in a lens space. At first, we will apply the surgery theorem of \cite{acosta_spherical_2016}, and then give explicit bounds for deformations and wonder about the uniformizability of the structures obtained.
 
 We will then consider in detail the structure constructed by Parker and Will in \cite{parker_complex_2017a} which admits as holonomy representation the representation $\rho$ with coordinates $(0, \alpha_2^{\lim})$. We will also briefly cite the structure constructed by Schwartz in his book \cite{schwartz}, corresponding to the point $(\alpha_1^{\lim},0)$, where $\alpha_1^{\lim} = \arccos(\frac{\sqrt{3}}{4})$.

 \subsubsection{The Parker-Will structure}

 In \cite{parker_complex_2017a}, Parker and Will study the groups inside the region $\mathcal{Z}$. Using the Poincaré polyhedron theorem, they show that in this case they are faithful and discrete representations of $\z3z3$, and that they are holonomy representations for open manifolds of dimension 4 with a $( \pu21 , \h2c )$-structure. In Section 6 of the article, they study the group of parameter $(0,\alpha_2^{\lim})$ and show that it is the holonomy representation of a \CR {} uniformization of the Whitehead link complement. In this case, they compute the images of $s$ and $t$:
\[
\rho(s) = 
\left[
\begin{matrix}
 1 & \frac{\sqrt{3}}{2}   - i \frac{\sqrt{5}}{2}  & -1\\
- \frac{\sqrt{3}}{2} - i\frac{\sqrt{5}}{2} & -1 & 0    \\
-1 & 0 & 0
\end{matrix}
\right]
, \hspace{1cm}
\rho(t) = 
\left[
\begin{matrix}
0 & 0 & -1 \\
0 & -1 & - \frac{\sqrt{3}}{2} + i\frac{\sqrt{5}}{2}    \\
-1 & \frac{\sqrt{3}}{2}   + i \frac{\sqrt{5}}{2}  & 1
\end{matrix}
\right]
\] 
 
 They construct then a Ford domain invariant by the holonomy $\rho(m_2) = \rho(st)$ of the second cusp. This domain is a horotube (see Figure 11 and proposition 6.8 of \cite{parker_complex_2017a}).
 
 \begin{rem}
 We remark that, for the holonomy representation of the uniformization, the traces of the images of $s,t,st,s^{-1}t$ and $[s,t]$ are respectively $(0,0,3,3,x_0)$, where $x_0 = \frac{15}{2}- \frac{3}{2}i\sqrt{15}$. 
 \end{rem}

\begin{prop}\label{prop_involution_WLC}
 Consider $WLC$ endowed with the uniformizable \CR {} structure of Parker-Will. Then there is an anti-holomorphic involution $\iota$ of $WLC$ exchanging the two cusps.
\end{prop}

\begin{proof}
Consider the automorphism $\varphi$ of $\z3z3$ given by 
 \[\varphi(s) = s^{-1} , \varphi(t) = t \]
 and the representation $\rho ' = \rho \circ \varphi$. Since $[s^{-1},t]$ is conjugate to $[s,t]^{-1}$, the character $\chi_\rho$ has coordinates $(0,0,3,3,x_0)$ and $\chi_{\rho '}$ has coordinates $(0,0,3,3,\con{x_0})$. By the Lubotzky-Magid theorem on characters of semi-simple representations (Theorem 1.28 of \cite{lubotzky_varieties_1985}), the representations $\rho$ and $\con{\rho '}$ are conjugate in $\slnc$. Since they are irreducible and with values in $\su21$, they are conjugate in $\su21$.
  Hence there exist anti-holomorphic involutions $\iota$ and $\eta$ of $\cp2$ and $\su21$ respectively such that:
 \begin{enumerate}
 \item $\iota$ stabilizes $\h2c$ and $\dh2c$
 \item For all $U \in \su21$ and $[Z] \in \cp2$ we have $[\eta(U)] [\iota(Z)]  = \iota([U][Z])$.
 \item $\eta(\rho(s)) = \rho'(s) = \rho(s)^{-1}$
 \item $\eta(\rho(t)) = \rho'(t) = \rho(t)$
\end{enumerate}   
   
 Since $\Im(\rho) = \Im(\rho ') = \Gamma$, the domain of discontinuity $\Omega_\Gamma$ is stabilized by $\iota$. Hence, the involution $\iota$ induces an anti-holomorphic involution of  $WLC = \Gamma \backslash \Omega_\Gamma$. 
 The holonomy representation of the structure given by $\iota(\Gamma \backslash \Omega_\Gamma)$ is then $\rho'$. Thus, $\rho'(st) = \rho(s^{-1}t) = \rho(t^{-1}) \rho(ts^{-1}) \rho(t)$ and $\rho'(s^{-1}t) = \rho(st)$. 
 Hence, the involution $\eta$ exchanges the peripheral holonomies of the two cusps: we deduce that $\iota$ exchanges the two cusps of $WLC$.
\end{proof}

 In particular, we deduce that there exists a neighborhood of the fist cusp whose image by the developing map is a horotube invariant by $\rho(m_1)  = \rho(t s^{-1})$. It is, indeed, the image by $\iota$ of a neighborhood of the second cusp, whose image by the developing map is a horotube invariant by $\rho(m_2)  = \eta(\rho(s)^{-1}\rho(m_1) \rho(s))$.

 \subsubsection{The Schwartz structure}\label{par_structure_schwartz}
 In his book \cite{schwartz}, Schwartz had already studied the groups corresponding to the real axis of the Parker-Will parametrization, constructing them as subgroups of triangle groups. He shows, in particular, that the representations are discrete in the segment $[-\alpha_1^{\lim},\alpha_1^{\lim}] \times \{0\}$, and that the representation with coordinates $(\alpha_1^{\lim},0)$ is the holonomy representation of a \CR {} uniformization of the Whitehead link complement. Furthermore, its results can be reformulated to establish that the image of a neighbourhood of each cusp by the developing map is a horotube. Schwartz also describes the two peripheral holonomies: the first one is horizontal unipotent and the second is generated by an ellipto-parabolic element $[P]$. By explicitly computing $[P]$ from the data in Chapter 4 of \cite{schwartz}, we obtain that the matrix $P$ is conjugate in $\su21$ to
   \[e^{i\theta}\begin{pmatrix}
1 & 0 & -\frac{i}{2} \\
0 & e^{-3i\theta} & 0 \\
0 & 0 &  1 \\
\end{pmatrix} \] 

where $\theta = \frac{1}{3} \arccos(-\frac{7}{8})$. In the trace coordinates of $\mathcal{X}_{\su21}(\z3z3)$, this representation is at the point $(3,w_{\mathrm{sch}})$, where $w_{\mathrm{sch}} = 2e^{i\theta} + e^{-2i\theta} \thickapprox 1,09062813494126 + 0,557252430478823i$. It is the intersection point of the red and blue curves of the Parker-Will slice, in Figure \ref{fig_peche_traces}.

\subsection{Spherical CR surgeries}\label{sect_ch_cr_wlc}
 We will apply the \CR {} surgery theorem of \cite{acosta_spherical_2016} to the Parker-Will uniformization of the Whitehead link complement.
 We keep the notation of \cite{acosta_spherical_2016} for the boundary thickening in order to state the result on a simpler way. We denote by $(\Dev_0 , \rho_0)$ the \CR {} structure on $WLC$ given by the Parker-Will uniformization. We will use the abusive notation of identifying $\rho : \mathbb{Z}/3\mathbb{Z} * \mathbb{Z}/3\mathbb{Z} \rightarrow \su21$ with representations $\rho : \pi_1(WLC) \rightarrow \pu21$.

\subsubsection{Applying the surgery theorem} 

We have seen below that the hypothesis of the surgery theorem of \cite{acosta_spherical_2016} are satisfied:
 the images of the two peripheral holonomies are generated by the unipotent elements $\rho(m_1) = \rho(st)$ and $\rho(m_1) = \rho(ts^{-1})$, and there exists $s \in [0,1[$ such that $\Dev_0(\widetilde{T_1}_{[s,1[})$ and $\Dev_0(\widetilde{T_2}_{[s,1[})$ are horotubes invariant under $\rho_0(m_1)$ and $\rho_0(m_2)$ respectively.

 \begin{rem}
  For the representations of $\pi_1(WLC)$ coming from representations of $\z3z3$, the relations $\rho(l_1) = \rho(m_1)^3$ and $\rho(l_2) = \rho(m_2)^3$ hold. Hence, the space parametrized by Parker and Will is contained in $\mathcal{R}_1(\pi_1(WLC) , \pu21)$. These relations are rigid: they are satisfied in the whole component of the $\slnc$-character variety of $WLC$. This fact is showed, with other techniques, by Guilloux and Will in \cite{guilloux_will_2016}.
 \end{rem} 

 Applying the surgery theorem of \cite{acosta_spherical_2016}, we obtain:
 
 \begin{prop}\label{prop_ouvert_ch_wlc}
  There exists an open neighborhood $\Omega$ of $\rho_0$ in $\mathcal{R}_1(\pi_1(WLC), \pu21)$ such that, for all $\rho \in \Omega$, there exist a \CR {} structure $(\Dev_\rho , \rho)$ on $WLC$ and for $i=1,2$:
  
  \begin{enumerate}
  \item If $\rho(m_i)$ is loxodromic, then the structure $(\Dev_\rho,\rho)$ extends to a structure on the Dehn surgery of $WLC$ of type $(-1,3)$ on $T_i$.
  \item If $\rho(m_i)$ is elliptic of type $(\frac{p}{n},\frac{1}{n})$, then the structure $(\Dev_\rho,\rho)$ extends to a structure on the Dehn surgery of $WLC$ of type $(-p,n+3p)$ on $T_i$.
  \item If $\rho(m_i)$ is elliptic of type $(\frac{p}{n},\frac{q}{n})$, then the structure $(\Dev_\rho,\rho)$ extends to a structure on the gluing of $WLC$ with the manifold $V(p,q,n)$ along $T_i$.
 \end{enumerate}   
 \end{prop}
 
 \begin{rem}\label{rem_marquage_wlc}
 The marking $(l_0,m_0)$ of the surgery theorem of \cite{acosta_spherical_2016} does not correspond to the marking $(l_1,m_1)$ that we consider here.
 % nous considérons depuis la sous-section \ref{subsect_description_wlc}.
  We know that $3m_1 - l_1 \in \ker (\rho)$ and that the image of each peripheral holonomy of $T_1$ is generated by $\rho(m_1)$. Hence the conclusions of the surgery theorem apply to the marking $l_0 = m_1$ and $m_0 = 3m_1 - l_1$. In particular, $nl_0 + pm_0 = nm_1 + p(3m_1-l_1) = -3 l_1 + (n+3p)m_1$. The same relation holds for the peripheral holonomy of $T_2$.
\end{rem}  

 In particular, in the region parametrized by Parker and Will, the peripheral holonomy of $T_2$ is always unipotent, and in a neighborhood of $\rho_0$, there exist open sets for which the peripheral holonomy of $T_1$ is loxodromic and elliptic respectively.
 Hence we obtain:
 
 \begin{cor}
  There are infinitely many \CR {} structures on the Dehn surgery of $WLC$ of type $(-1,3)$ on $T_1$.
 \end{cor}
 
 \begin{rem}
 Using SnapPy, we know that the Dehn surgery of $WLC$ of type $(-1,3)$ on $T_1$ is a manifold with a torus boundary and fundamental group $\langle a,b \mid a^3b^3 \rangle$ which is not hyperbolic. The Dehn surgery on $T_2$ of type $(0,1)$ of this last manifold is the lens space $L(3,1)$.
 \end{rem}

\subsubsection{Expected Dehn surgeries}\label{sect_ch_attendues}
 In order to explicit the third point of Proposition \ref{prop_ouvert_ch_wlc} and determine the Dehn surgeries on $T_1$ admitting \CR {} structures extending $(\Dev_\rho, \rho)$, we need to identify the type of the elliptic element which generates the peripheral holonomy through the deformation. 
 Outside from the curve of unipotents of the region parametrized by Parker and Will, $\rho(m_1)$ is elliptic. 
  At first, consider the point $\rho_1$ with coordinates $(\alpha_1,\alpha_2) = (0, \frac{2\pi}{3})$. Denoting by $\omega = e^{\frac{2i\pi}{3}}$, the element $\rho_1(m_1)$ is conjugate to

\[ \rho_1(s^{-1}t) = 
\left(\begin{array}{ccc}
1 & -\sqrt{2}\omega & -1 \\

-\sqrt{2}\omega & 1+2\omega^2 & - \sqrt{2} \\

-1 & - \sqrt{2} & -2\omega^2
\end{array}\right)
\]

The eigenvalues of this matrix are $-\omega^2 $, $-\omega $ and $1$, of respective eigenvectors 

\[
V_1 = \begin{pmatrix}
1 \\
-\sqrt{2} \omega  \\
 \omega^2
\end{pmatrix} ,
\hspace{0.5cm}
V_2 = \begin{pmatrix}
1 \\
0 \\
 -\omega
\end{pmatrix} ,
\hspace{0.5cm}
V_3 = \begin{pmatrix}
\sqrt{2} \\
- \omega  \\
 \sqrt{2}\omega^2
\end{pmatrix} 
.\]
Moreover, this vectors have norms $\Phi(V_1) = \Phi(V_2) = 1$ and $\Phi(V_3) = -1$. Hence the fixed point of $\rho(s^{-1}t)$ in $\h2c$ is $[V_3]$. After a quick computation, we deduce that $\rho_1(m_1)$ is of type $(\frac{1}{9}, \frac{-1}{9})$. If a regular elliptic element admits as eigenvalues of its positive eigenvectors $e^{2i\pi\alpha}$ and $e^{2i\pi\beta}$, then $2\alpha-\beta$ and $2\beta - \alpha$ are nonzero. If, in addition, the element is of type $(\frac{p}{n}, \frac{q}{n})$, up to exchanging $\alpha$ and $\beta$, we have $2\alpha - \beta = \frac{p}{n}$ and $2\beta - \alpha = \frac{q}{n}$. Since $\mathcal{E}$ is connected and the eigenvectors and eigenvalues of a matrix are continuous, if $\rho \in \mathcal{E}$ has type $(\frac{p}{n},\frac{q}{n})$ with $p \geq q$, then $p>0>q$.

Since $\mathrm{tr}(m_1)$ is a local parameter of the region parametrized by Parker and Will, there exists an open neighborhood of $3 \in \mathbb{C}$ of traces reached by $\mathrm{tr}(m_1)$.
  But if $\rho(m_1)$ is in $\mathcal{E}$ and its trace is of the form $e^{2i\pi\frac{2p-1}{n}} + e^{2i\pi\frac{2-p}{n}} + e^{2i\pi\frac{-p-1}{n}}$ with $1>0>p>-n$, then $\rho(m_1)$ is elliptic of type $(\frac{p}{n},\frac{1}{n})$. By exchanging $p$ into $-p$ to have a more clear statement, we obtain the following proposition:
  
\begin{prop}\label{prop_ch_dehn_attendues_wlc}
 There exists $\delta > 0$ such that, for all relatively prime integers $p,n$ satisfying $0<p<n$ and $\frac{p}{n} < \delta$, there exists a deformation of the structure $(\Dev_0,\rho_0)$ which extends into a \CR {} structure on the Dehn surgery of $WLC$ of type $(p,n-3p)$ on $T_1$.
\end{prop}

\begin{rem}
 If the open set for which Proposition \ref{prop_ch_dehn_attendues_wlc} holds is large enough, we will recognize some \CR {} structures on manifolds that had already been studied. By checking with SnapPy, we can notice that for the parameters $(p,n) = (1,4)$, the obtained surgery is the Figure eight knot complement, and the corresponding representation is the holonomy representation of the Deraux-Falbel structure constructed in \cite{falbel}. 
 Moreover, when $(p,n) = (1,5)$, the obtained Dehn surgery is the manifold \texttt{m009} of the census of Falbel, Koseleff and Rouillier of \cite{falbel_kosleff_rouillier}, and the representation is the one studied by Deraux in \cite{deraux_spherical_2015}, where he shows that it gives a \CR {} uniformization of the manifold \texttt{m009}.
 
 Hence we expect that these structures can be obtained as spherical CR Dehn surgeries of the Parker-Will structure on $WLC$. We will prove this fact in Part \ref{part_eff_defor}.
\end{rem}

\subsubsection{Surgeries on the Schwartz structure}
 We can also apply the \CR {} surgery theorem of \cite{acosta_spherical_2016} to the Schwartz uniformization  of $WLC$ that we briefly described in paragraph \ref{par_structure_schwartz}. 
 In the rest of this article we will not consider this case any more, and we will take the Parker-Will uniformiztion as a starting point.
 
 The Schwartz structure satisfies the hypotheses of the surgery theorem of \cite{acosta_spherical_2016} for the cusp of unipotent peripheral holonomy; we only need to describe a space of representations where the non-unipotent cusp has constant peripheral holonomy and the image of the holonomy of the unipotent cusp becomes elliptic or loxodromic. 
 This is equivalent to study the points of $\mathcal{X}_{\su21}(\z3z3)$ with coordinates $(z,w_{\mathrm{sch}})$, taking as starting point the coordinates $(3,w_{\mathrm{sch}})$ of the Schwartz representation describes in paragraph \ref{par_structure_schwartz}.

\begin{figure}[htb] 
 \begin{center}
   \includegraphics[width=5cm]{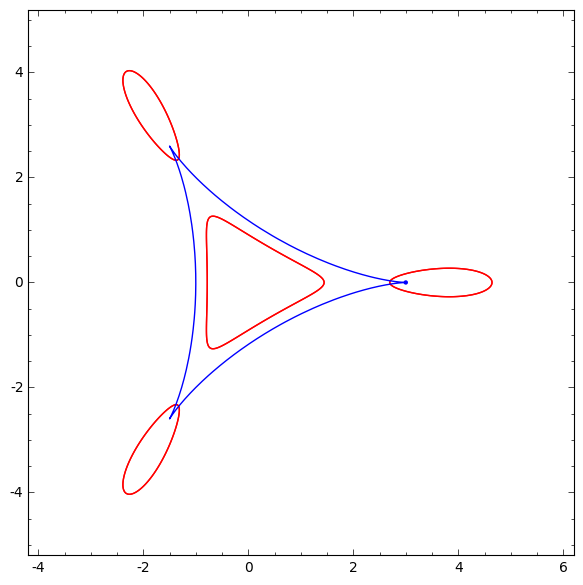}
   \includegraphics[height=3.5cm]{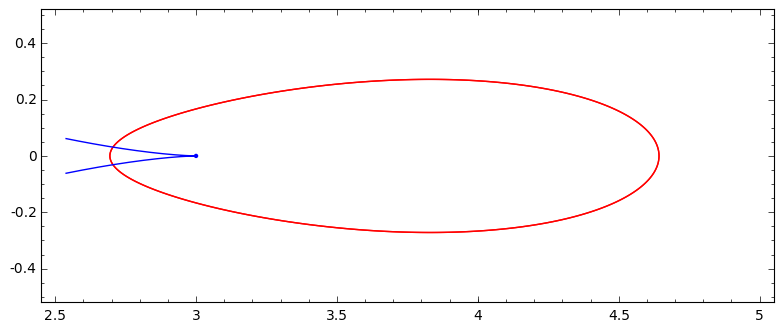}
   \caption{The Schwartz slice, $w = w_{\mathrm{sch}}$ of $\mathcal{X}_{\su21}(\z3z3)$.}\label{fig_tranche_schwartz}
 \end{center}
\end{figure} 

 These points correspond to the interior of the red lobes of Figure \ref{fig_tranche_schwartz}, in which there is also traced the curve of non-regular elements. We remark that in a neighbourhood of the point with coordinate $3$, there are representations with loxodromic peripheral holonomy as well as elliptic. In order to know the type of the elliptic when deforming the continuity argument in paragraph \ref{sect_ch_attendues} still holds. Hence, the elliptic elements that appear are of type $(\frac{p}{n},\frac{q}{n})$, where $p>0>q$. By applying the surgery theorem of \cite{acosta_spherical_2016}, and denoting by $\rho_{\mathrm{sch}}$ the holonomy representation of the Schwartz structure on $WLC$ we obtain:

 \begin{prop}
  There is an open neighborhood $\Omega$ of $\rho_{\mathrm{sch}}$ in $\mathcal{R}_1(\pi_1(WLC), \pu21)$ such that for all $\rho \in \Omega$, there exists a \CR {} structure $(\Dev_\rho , \rho)$ on $WLC$, close to the Schwartz uniformization, and such that:
  
  \begin{enumerate}
  \item If $\rho(m_1)$ is loxodromic, then the structure $(\Dev_\rho,\rho)$ extends to a structure on the Dehn surgery of $WLC$ of type $(-1,3)$ on $T_1$.
  \item If $\rho(m_1)$ is elliptic of type $(\frac{p}{n},\frac{1}{n})$, then the structure $(\Dev_\rho,\rho)$ extends to a structure on the Dehn surgery of $WLC$ of type  $(-p,n+3p)$ on $T_1$.
  \item If $\rho(m_1)$ is elliptic of type $(\frac{p}{n},\frac{q}{n})$, then the structure $(\Dev_\rho,\rho)$ extends to a structure on the gluing of $WLC$ with the manifold $V(p,q,n)$ along $T_1$.
 \end{enumerate}   
 \end{prop}
 
 More precisely, for the Dehn surgeries, we obtain the two following propositions:
 
\begin{prop}
 There is an open set of real dimension 2 parametrizing \CR {} structures on the Dehn surgery of $WLC$ of type $(-1,3)$ on $T_1$, which are obtained by deforming the Schwartz uniformization.
\end{prop} 
 
\begin{prop}
 There exists $\delta > 0$ such that for all relatively prime integers $p,n$ satisfying $0<p<n$ and $\frac{p}{n} < \delta$, there exists a deformation of the Schwartz uniformization which extends to a \CR {} structure on the Dehn surgery of $WLC$ of type $(p,n-3p)$ on $T_1$.
\end{prop}

\subsubsection{The Figure eight knot complement}
 At last, we make some remarks about the following observation, made by Parker and Will in \cite{parker_complex_2017a}:
 
 \begin{rem}
  At the point of coordinates $(\alpha_1,\alpha_2) = (0,\arctan(\sqrt{7}))$, by setting $G_1 = \rho(st)$, $G_2= \rho((tst)^3)$ and $G_3 = \rho(ts)$, we obtain the representation $\rho_2$ of the fundamental group of the Figure eight knot complement found by Falbel in \cite{falbel_spherical}. This representation is also the holonomy representation of the uniformization of Deraux-Falbel, constructed in \cite{falbel}.
 \end{rem}

 \begin{rem}
  At the point of coordinates $(\alpha_1,\alpha_2) = (0,\arctan(\sqrt{7}))$, the element $\rho(m_1) = \rho(ts^{-1})$ is elliptic of type $(\frac{-1}{4}, \frac{1}{4})$. If the open set $\Omega$ given by proposition \ref{prop_ouvert_ch_wlc} contains the point $(0,\arctan(\sqrt{7}))$, then the expected Dehn surgery is of type $(1,-1)$ on $T_1$. It is, indeed, the Figure eight knot complement.
 \end{rem}

 By noticing that $tst$ is conjugate to $st^{-1}$, we obtain the following proposition:
 
 \begin{prop}
  Let $\rho \in \mathrm{Hom}(\z3z3, \su21)$ be a representation. The following assertions are equivalent:
 \begin{enumerate}
  \item $\rho(tst)$ has order 4
  \item $\mathrm{tr}(\rho(st^{-1})) = 1$
  \item $\chi_\rho$ has coordinates $(0,0,z,1,x)$ in $\mathcal{X}_{\su21}(\z3z3)$ with $z,x \in \mathbb{C}$
  \item  By setting $G_1 = \rho(st)$, $G_2= \rho((tst)^3)$ and $G_3 = \rho(ts)$, we obtain a representation of $\pi_1(M_8)$.
\end{enumerate}   
 \end{prop}
 
\begin{rem}
 In this case, the character $\chi_\rho$ has trace coordinates of the form $(0,0,z_\rho , 1 , x_\rho)$. The map $[\rho] \mapsto z_\rho$ is a double cover over its image, besides the boundary curve where the fibres are singletons. 
 We have then the parametrization of the component of the character variety of the Figure eight knot given by  Falbel, Guilloux, Koseleff, Rouillier and Thistlethwaite in \cite{character_sl3c},
 studied in \cite{acosta_character_2016}.
\end{rem}

\newpage
\part{Effective deformation of a Ford domain} \label{part_eff_defor}

\section{Statements and strategy of proof}\label{sect_strategie_de_preuve}
 We will give an explicit bound, at least in one direction, to Theorem \ref{prop_ouvert_ch_wlc}, by deforming the Ford domain of the Parker-Will uniformization. We are going to consider the representations parametrized by Parker and Will in \cite{parker_complex_2017a} with parameter $(0, \alpha_2)$, where $\alpha_2 \in ]-\frac{\pi}{2}, \frac{\pi}{2}[$. We take as starting point the point with parameter $(0, \alpha_2^{\lim})$, corresponding to the holonomy representation of the uniformization.

  \subsection{Spherical CR structures: statements}
\begin{notat}
If $\rho$ is a representation with parameter $(0, \alpha_2)$ in the Parker-Will parametrization,  we denote by $\Gamma(\alpha_2) = \Im (\rho)$ its image. In order to avoid heavy notation, we will often use $\Gamma$ instead of $\Gamma(\alpha_2)$ if there is not ambiguity for the parameter.
 For $\alpha_2^{\lim}$, we denote $\rho_\infty = \rho(\alpha_2^{\lim})$ and $\Gamma_\infty = \Gamma(\alpha_2^{\lim})$ 
  We finally denote by $\rho_n$, for $n\geq 4$,  the representation $\rho(\alpha_2)$ such that $8 \cos^2(\alpha_2) = 2\cos(\frac{2\pi}{n}) + 1$.
\end{notat}

   We will show the two following theorems:
  
\begin{thm}\label{thm_ch_dehn_eff_ell}
 Let $n\geq 4$. Let $\rho_n$ be the representation with parameter $(0, \alpha_2)$ such that $8 \cos^2(\alpha_2) = 2\cos(\frac{2\pi}{n}) + 1$ in the Parker-Will parametrization. Then, $\rho_n$ is the holonomy representation of a \CR {} structure on the Dehn surgery of the Whitehead link complement on $T_1$ of type $(1,n-3)$ (i.e. of slope $\frac{1}{n-3}$). 
\end{thm}

\begin{thm}\label{thm_ch_dehn_eff_lox}
 Let $\alpha_2 \in ]0 , \alpha_2^{\lim}[$. Let $\rho$ be the representation with parameter $(0, \alpha_2)$ in the Parker-Will parametrization. Then $\rho$ is the holonomy representation of a \CR {} structure on the Dehn surgery of the Whitehead link complement on $T_1$ of type $(1,-3)$ (i.e. of slope $-\frac{1}{3}$).
\end{thm}  
  
  \begin{rem}
 We will give a complete proof of Theorem \ref{thm_ch_dehn_eff_ell} only for  $n \geq 9$. The techniques used for the proof do not let us treat the last five cases. However, the global combinatorics of bisectors that would be needed to conclude for $n \geq 4$ are shown by Parker, Wang and Xie in \cite{parker_wang_xie}, but with other techniques, using in particular a parametrization of a family of triangle groups. The result on the global combinatorics corresponds to the statement of Theorem 4.3 in \cite{parker_wang_xie}; the link between their notation and ours is given by $U = I_1I_2$, $S^{-1}=I_1I_3$ and $T = I_3I_2$.
 \end{rem}

\subsection{Strategy of proof}
 By studying the image $\Gamma_\infty$ of the holonomy representation $\rho_\infty : \pi_1(WLC) \rightarrow \pu21$, Parker and Will construct a Ford domain for the set of left cosets $[A] \backslash \Gamma_{\infty}$ for some unipotent element $[A]$, which generates the image of one of the peripheral holonomy. This structure has a symmetry exchanging the two cusps, seen in Proposition \ref{prop_involution_WLC}; we denote by $U$ the image of $A$ by the corresponding involution $\eta$.

 We will consider some deformations of the holonomy representation $\rho$ given by Parker and Will and deform the Ford domain for $[U] \backslash \Gamma$ into a domain centred at a fixed point of $[U]$ and invariant by $[U]$, with face identifications given by elements of $\Gamma$ and with the same local combinatorics as the Parker-Will domain. If $[U]$ is elliptic, it will be a Dirichlet domain, like the one given by Deraux and Falbel in \cite{falbel} for the Figure eight knot complement. If $[U]$ is loxodromic, the domain will be centred \emph{outside from} $\h2c$.
 
 Considering separately a neighbourhood of the cusp and the rest of the structure, we identify the obtained structures as Dehn surgeries, like in the surgery theorem of \cite{acosta_spherical_2016}.

 In order to deform the Ford domain of Parker and Will,
we will deform their construction by defining a one parameter family of bisectors $(\mathcal{J}_k^{\pm})_{k \in \ZZ} \subset \con{\h2c}$ invariant by the action of $[U]$.
 We will define the 3-faces of our domain $\mathcal{F}_k^{\pm} \subset \mathcal{J}_k^{\pm}$ by cutting off a part of the bisector $\mathcal{J}_{k}^+$ by the bisectors $\mathcal{J}_{k}^{-}$ and $\mathcal{J}_{k+1}^{-}$ to define the face $\mathcal{F}_{k}^+$ and by cutting off a part of the bisector  $\mathcal{J}_{k}^-$ by the bisectors $\mathcal{J}_{k}^{+}$ and $\mathcal{J}_{k-1}^{+}$ to define the face $\mathcal{F}_{k}^-$. See Figure \ref{fig_3faces_1} in order to have in mind the shape of these 3-faces. We will give a more precise definition in section \ref{sect_notat_combi_initiale}, as well as the notation that we will use from that point on.

We will need to check three conditions to establish our results:
  a condition on the topology of faces, that we will denote by (TF),
  a condition on the local combinatorics of intersections, that we will denote by (LC) and
  a condition on the global combinatorics of intersections, that we will denote by  (GC).
More precisely, we can state them in the following way:
\begin{notat}
 We call conditions (TF), (LC) and (GC) the following conditions:

\begin{itemize}
 \item[(TF)] The intersections of the form $\mathcal{F}_k^+ \cap \mathcal{F}_k^-$ and $\mathcal{F}_k^+ \cap \mathcal{F}_{k-1}^-$ are bi-tangent Giraud disks. In particular, $\partial_\infty \mathcal{F}_k^{+}$ is bounded by two bi-tangent circles, defining a <<bigon>> and a <<quadrilateral>>. The same holds for intersections of the form $\mathcal{F}_k^- \cap \mathcal{F}_k^+$ and $\mathcal{F}_k^- \cap \mathcal{F}_{k+1}^-$.
 \item[(LC)] The intersections of the form $\mathcal{F}_k^+ \cap \mathcal{F}_{k+1}^+$ and $\mathcal{F}_k^- \cap \mathcal{F}_{k+1}^-$ contain exactly two points; the ones of the form $\mathcal{F}_k^+ \cap \mathcal{F}_{k+1}^-$ contain exactly one point; all these points are in $\dh2c$.
 \item[(GC)] The face $\mathcal{F}_k^{\pm}$ intersects $\mathcal{F}_{k+l}^{\pm}$ if and only if $l \in \{-1,0,1\}$. The face $\mathcal{F}_k^{+}$ intersects $\mathcal{F}_{k+l}^-$ if and only if $l \in \{-1,0,1\}$.
 The indexes are modulo $n$ whenever $[U]$ is elliptic of order $n$. 
\end{itemize} 

 \end{notat}

\begin{figure}[htbp]
\center
 \includegraphics[width = 4cm]{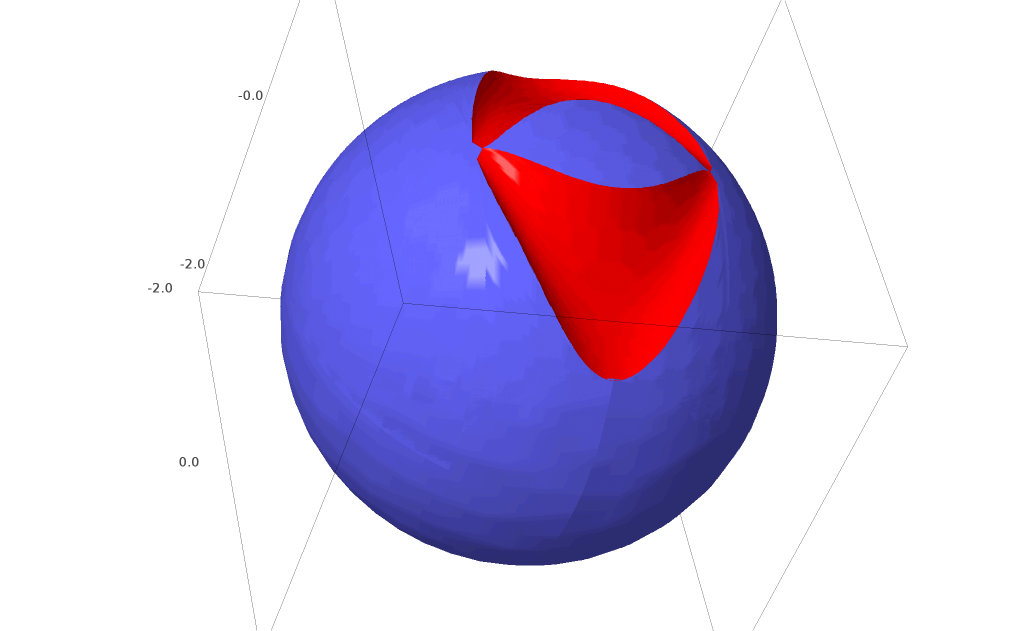}
 \includegraphics[width = 4cm]{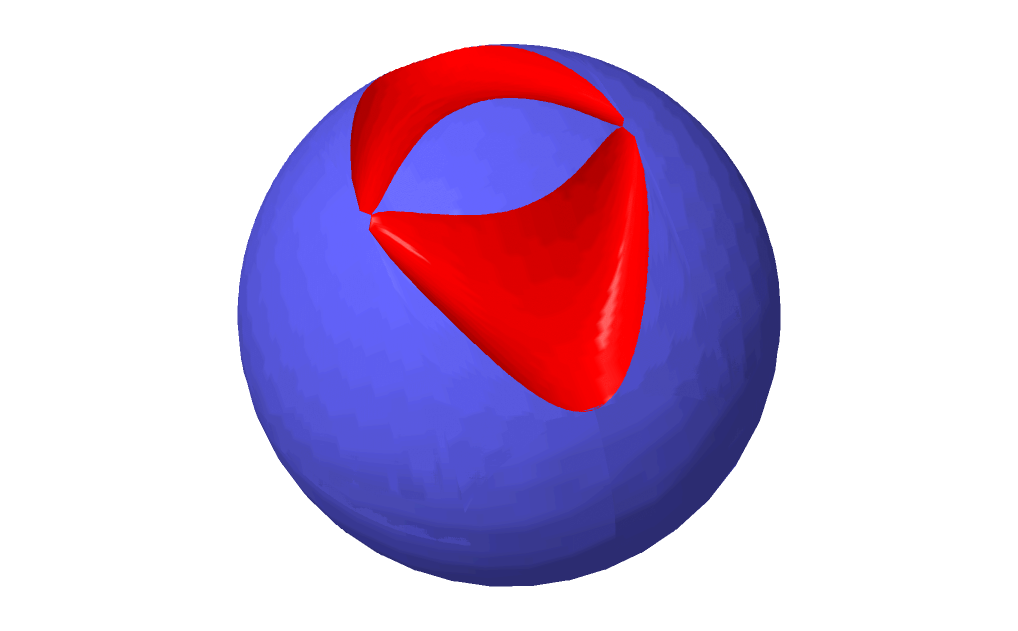}
 \includegraphics[width = 4cm]{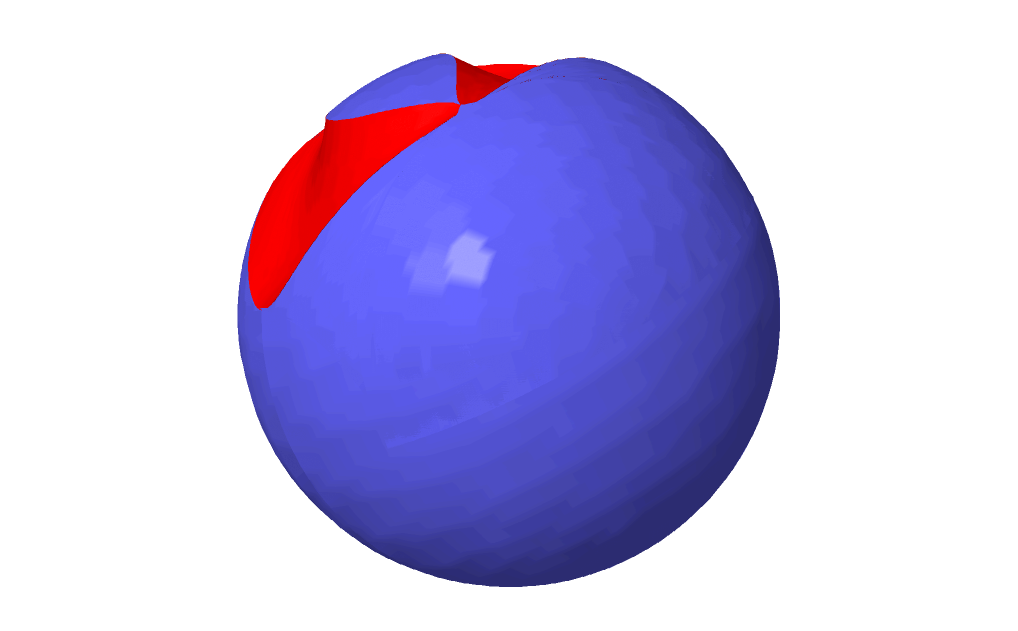}
 \caption{Simplified picture of a face $\mathcal{F}_k^{\pm}$. The blue region is in $\dh2c$; the two red regions are Giraud disks in $\h2c$. The blue region consists of a bigon $B_k^{\pm}$ and a quadrilateral $Q_k^{\pm}$.}\label{fig_3faces_1}
\end{figure}

 If these three conditions are satisfied, then the faces $\mathcal{F}_k^{\pm}$ are well defined and border a domain in $\h2c$, which has the same side pairing as the one given by Parker and Will in \cite{parker_complex_2017a}. The same holds for the boundary at infinity of the domain, which is bordered by the bigons and quadrilaterals of $\partial_\infty \mathcal{F}_k^{\pm}$. A simplifyed picture of the topology of the faces is given in Figure \ref{fig_3faces_1}. We will come back later to the details of the picture below.
 If the three conditions are satisfied, then the domain defined in $\dh2c$ together with the side pairings determines a \CR {} structure on $WLC$, which extends to the surgery expected by Theorem \ref{prop_ouvert_ch_wlc}, more precisely:
 
 \begin{enumerate}
  \item If $[U]$ is loxodromic, then the structure $(\Dev_\rho,\rho)$ extends to a structure on the Dehn surgery on $WLC$ of type $(-1,3)$ on $T_1$.
  \item If $[U]$ is elliptic of type $(\frac{p}{n},\frac{1}{n})$, then the structure $(\Dev_\rho,\rho)$ extends to a structure on the Dehn surgery on $WLC$ of type $(-p,n+3p)$ on $T_1$.
  \item If $[U]$ is elliptic of type $(\frac{p}{n},\frac{q}{n})$, then the structure $(\Dev_\rho,\rho)$  extends to a structure on the gluing of $WLC$ with the manifold $V(p,q,n)$ along $T_1$.
 \end{enumerate}

We are going to show conditions (TF), (LC) and (GC) in the particular case of deformations with parameter $(0,\alpha_2)$ for $\alpha_2 \in ]0,\frac{\pi}{2} [$ in the Parker-Will parametrization. This will give Theorems \ref{thm_ch_dehn_eff_lox} and \ref{thm_ch_dehn_eff_ell}.
 We will begin by setting notation and recalling the initial combinatorics in Section \ref{sect_notat_combi_initiale}. Then, assuming conditions (TF), (LC) and (GC), we will prove the statements on \CR {} structures in section \ref{sect_fin_preuve}. Thereafter, we will check the three conditions, which is mostly technical work. We will show the condition (TF) of the topology of the faces in Section \ref{sect_topo_faces}, then the condition (LC) of local combinatorics in Section  \ref{sect_combi_locale} and finally we will consider the condition (GC) of global combinatorics in two steps, in Section \ref{sect_combi_globale} for a global strategy of proof, and in Subsections \ref{subsect_combi_lox} for the case where $[U]$ is loxodromic and \ref{subsect_combi_ell} for the case where $[U]$ is elliptic.
 
 \subsection{Results involving the Poincaré polyhedron theorem}
 
 We can wonder if the spherical CR structures given by Theorems
 \ref{thm_ch_dehn_eff_ell} and \ref{thm_ch_dehn_eff_lox} are uniformizable. In order to prove such a result, we will need to apply a Poincaré polyhedron theorem in $\h2c$, as stated for example in \cite{parker_complex_2017a}. A complete proof of this theorem will appear in the book of Parker \cite{parker_complex_Toappear}.
  By applying the Poincaré polyhedron theorem in $\h2c$ from Theorem  \ref{thm_ch_dehn_eff_ell}, we obtain the following theorem:
  
  \begin{thm}\label{thm_unif_ch_dehn_ell}
 Let $n \geq 4$. Then the Dehn surgery of the Whitehead link complement on $T_1$ of slope $\frac{1}{n-3}$ admits a \CR {} uniformization, given by the group $\Gamma_n$.
\end{thm}

Once again, our proof only holds for $n \geq 9$; and the use of the Poincaré polyhedron theorem is essentially the same as the one done by Parker, Wang and Xie in \cite{parker_wang_xie}. 
Considering the results on the combinatorics of the intersections shown in \cite{parker_wang_xie}, we can complete the proof for the five last cases.
For $\alpha_2 \in [ \frac{\pi}{6} , \alpha_2^{\lim} [$, the Poincaré polyhedron theorem can still be applied, and apart from the condition of being a \emph{polyhedron} (where all faces are homeomorphic to balls), we check the hypothesis for $\alpha_2 \in ]0 , \frac{\pi}{6} [$. See Lemma \ref{lemme_nature_bissecteurs} for more details. This allows us to conjecture the following result:

\begin{conj}\label{conj_unif_ch_dehn_lox}
 Let $\alpha_2 \in ]0 , \alpha_2^{\lim}[$. Then the group $\Gamma(\alpha_2)$ gives a \CR {} uniformization on the Dehn surgery of the Whitehead link complement on $T_1$ of slope $-\frac{1}{3}$.
\end{conj}

\section{Notation and initial combinatorics}\label{sect_notat_combi_initiale}

 With the notation and the tools related to bisectors, extors and the intersections that we studied in Subsection \ref{subsect_pairs_coeq_biss}, we will describe a deformation of the Ford domain constructed by Parker and Will in \cite{parker_complex_2017a}.
 Let us begin by recalling the combinatorics of the domain and fix a notation for some remarkable elements of the group as well as some points of $\cp2$.

\subsection{Notation - remarkable points} \label{subsect_notat_points}

 We set at first notation for the elements of the group and some remarkable points that we will use below.

\begin{notat}[Elements] In the family of representations  $\rho$ of $\z3z3 = \langle s, t \rangle$ with values in $\su21 $ parametrized by Parker and Will, we denote by $S = \rho(s)$ and $T = \rho(t)$. They are two regular elliptic elements of order 3.

 In the same way as Parker and Will, we denote by $A = ST$ and $B = TS$. These two elements, conjugate by $S$, are unipotent in the family of representations, parametrized by $(\alpha_1,\alpha_2) \in ]\frac{-\pi}{2} , \frac{\pi}{2} [^2$.
 At last, we denote by $U = S^{-1} T$, $V = SUS^{-1} = T S^{-1}$ and $W = SVS^{-1} = STS$.
\end{notat}

\begin{notat}[Fixed points] If $[G] \in \pu21$ is unipotent or regular, we will denote by $[p_G] \in \cp2$ a particular fixed point. 

\begin{itemize}
 \item If $[G]$ is parabolic, denote by $[p_G]$ its unique fixed point in $\dh2c$.
 \item If $[G]$ is regular elliptic, denote by $[p_G]$ its unique fixed point in $\h2c$.
 \item If $[G]$ is loxodromic, denote by $[p_G]$ its unique fixed point in $\cp2 \setminus \con{\h2c}$.
\end{itemize}
\end{notat}

\begin{rem}
 When $[G]$ is unipotent or regular, the map $[G] \mapsto [p_G]$ is continuous.
\end{rem}

In this way, in the region parametrized by Parker and Will, we have:

\[[p_A] = \begin{bmatrix} 1 \\ 0 \\ 0 \end{bmatrix}
\text{ and }
[p_B] = \begin{bmatrix} 0 \\ 0 \\ 1 \end{bmatrix}. \]

We will limit ourselves, to the deformations of the Parker-Will structure with coordinate $\alpha_1 = 0$. In this case, we have:

\[S = \begin{pmatrix}
1 &  \sqrt{2} e^{- i \alpha_2} & -1 \\
-\sqrt{2}e^{i\alpha_2} & -1 & 0 \\
-1 & 0 & 0
\end{pmatrix}  
\text{, }
T = \begin{pmatrix}
0 &  0 & -1 \\
0 & -1 & -\sqrt{2} e^{- i \alpha_2} \\
-1 & \sqrt{2}e^{i\alpha_2} & 1
\end{pmatrix} 
\]
\[\text{ and }U = \begin{pmatrix}
1 & - \sqrt{2} e^{i \alpha_2} & -1 \\
-\sqrt{2}e^{i\alpha_2} & 1 + 2 e^{2 i \alpha_2} & 2 \sqrt{2} \cos(\alpha_2) \\
-1 & 2 \sqrt{2} \cos(\alpha_2) & 2+2 e^{-2i \alpha_2}
\end{pmatrix}. \] 

Then 
$[p_U] = \begin{bmatrix} 1 \\ -\frac{\sqrt{2}}{2} e^{i \alpha_2} \\ e^{2 i \alpha_2} \end{bmatrix}$
,
$[p_V] = [S p_U] = \begin{bmatrix} -e^{2 i \alpha_2} \\ -\frac{\sqrt{2}}{2} e^{i \alpha_2} \\ -1 \end{bmatrix}$
 and 
$[p_W] = [S p_V] = \begin{bmatrix} -e^{2 i \alpha_2} \\ \sqrt{2}e^{3 i \alpha_2} + \frac{\sqrt{2}}{2} e^{i \alpha_2} \\ e^{2 i \alpha_2} \end{bmatrix}$.

\begin{notat}
If $U$ is not unipotent, it has three different eigenvalues. We will denote by $p'_U$ and $p''_U$ two eigenvectors associated to the eigenvalues different from 1. Denoting by $\delta$ a square root of $(8\cos^2(\alpha_2) - 3)(8\cos^2(\alpha_2) + 1)$, we have:
\[
[p'_U] = \begin{bmatrix}
2(2e^{2i\alpha_2}+1) \\
 -\sqrt{2}e^{i\alpha_2}(2e^{2i\alpha_2}+1 + \delta) \\
 -(8\cos^2(\alpha_2)+1) -\delta
\end{bmatrix}
\text{ and }
[p''_U] = \begin{bmatrix}
2(2e^{2i\alpha_2}+1) \\
 -\sqrt{2}e^{i\alpha_2}(2e^{2i\alpha_2}+1 - \delta) \\
 -(8\cos^2(\alpha_2)+1) + \delta
\end{bmatrix}.
\]

\end{notat}

\begin{rem}\label{rem_alpha2_ell}
 If $\alpha_2 > \alpha_2^{\lim}$, then $U$ is a regular elliptic element with eigenvalues $1, e^{i\beta}$ and $e^{-i\beta}$ for some $\beta \in ]0,\frac{\pi}{2}[$. The respective eigenvectors are then $p_U, p'_U$ and $p''_U$.
  In this case, $\mathrm{tr}(U) = 2\cos(\beta) +1 $ and $(8\cos^2(\alpha_2) - 3)(8\cos^2(\alpha_2) + 1) = (\mathrm{tr}(U)-3)(\mathrm{tr}(U)+1) = -4\sin^2(\beta)$. We will take $\delta = 2i\sin(\beta)$. 
\end{rem}

\begin{rem}\label{rem_alpha2_lox}
 If $\alpha_2 < \alpha_2^{\lim}$, then $U$ is a loxodromic element with eigenvalues $1, e^{l}$ and $e^{-l}$ for some $l \in \mathbb{R}^+$. The respective eigenvectors are then $p_U, p'_U$ and $p''_U$.
  In this case, $\mathrm{tr}(U) = 2\cosh(l) +1 $ and $(8\cos^2(\alpha_2) - 3)(8\cos^2(\alpha_2) + 1) = (\mathrm{tr}(U)-3)(\mathrm{tr}(U)+1) = 4\sinh^2(l)$. We will take then $\delta = 2\sinh(l)$.
\end{rem}

\subsection{Combinatorics of the Parker-Will uniformization}\label{subsect_combi_parker_complex_2017a}
\begin{notat}
 Following the article of Parker and Will \cite{parker_complex_2017a}, let
 $\mathcal{I}_0^{+} = \con{\mathfrak{B}}(p_A, S^{-1} p_A)$, and, for $k \in \mathbb{Z}$, let $\mathcal{I}_k^{+} = A^k \mathcal{I}_0^{+} = \con{\mathfrak{B}}(p_A, A^k S^{-1} p_A)$.
  Similarly, we denote by
  $\mathcal{I}_0^{-} = \con{\mathfrak{B}}(p_A, S p_A) $ and $\mathcal{I}_k^{-}= A^k \mathcal{I}_0^{-} = \con{\mathfrak{B}}(p_A, A^k S p_A)$.
  These sets are closed bisectors in $\con{\h2c}$.
\end{notat}

Parker and Will show that, for the representation with parameter $(0, \alpha_2^{\lim})$, the bisectors $\mathcal{I}_k^{\pm}$ bound an infinite polyhedron in $\h2c$, locally finite and invariant by $[A]$, which is endowed with a side pairing. Its boundary at infinity is the region of $\dh2c$ containing $[p_A]$ and limited by the spinal spheres $\partial_\infty \mathcal{I}_k^{+}$ and $\partial_\infty \mathcal{I}_k^{-}$. Parker and Will show that this region is a Ford domain invariant by $[A]$ for the \CR {} uniformization of the Whitehead link complement.  

The domain in $\dh2c$ has four classes of faces: quadrilaterals $Q_k^{+}$ and $Q_k^{-}$, contained in $\partial_\infty \mathcal{I}_k^{+}$ and $\partial_\infty \mathcal{I}_k^{-}$ respectively, and bigons $B_k^{+}$ and $B_k^{-}$, contained in $\partial_\infty \mathcal{I}_k^{+}$ and $\partial_\infty \mathcal{I}_k^{-}$ respectively. We can see the incidences, combinatorially, in Figure \ref{combi_ford_wlc_avant}.

\begin{figure}[ht]
 \center
 \includegraphics[width=8cm]{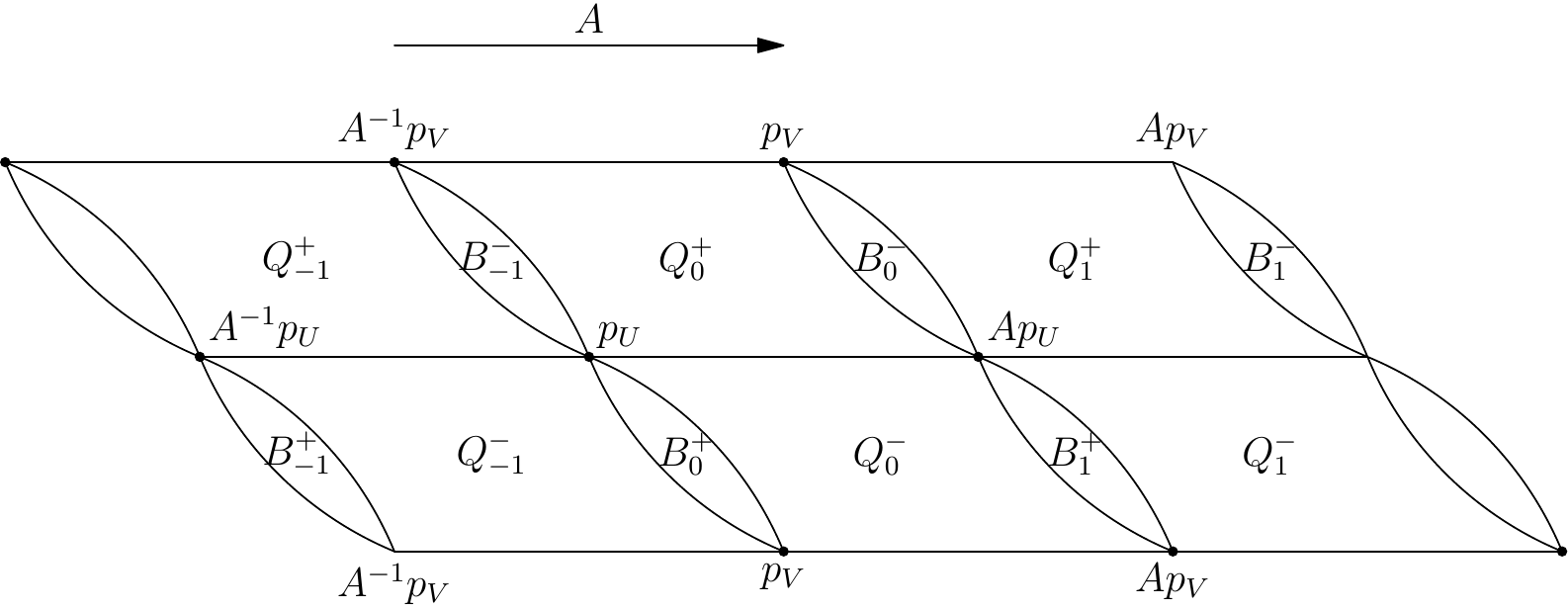}
 \caption{Combinatorics of the boundary at infinity of the Ford domain for $\Gamma_{\infty}$ construced by Parker and Will. The two horizontal boundaries are identifyed by a vertical translation.} \label{combi_ford_wlc_avant}
\end{figure}

In this way, the quadrilateral $Q_0^{+}$ has vertices $[p_U]$, $[A^{-1}p_V]$, $[p_V]$ and $[A p_U]$
 and the bigon $B_0^{+}$ has vertices $[p_U]$ and $[p_V]$. The sides of $Q_0^{+}$ and $B_0^{+}$ are arcs of the Giraud circles $\partial_\infty \mathcal{I}_0^{+} \cap \partial_\infty\mathcal{I}_{-1}^{+}$ and $\partial_\infty\mathcal{I}_0^{+} \cap \partial_\infty\mathcal{I}_{0}^{-}$, which contain  $\{[A^{-1}p_V] , [p_U] , [p_V] \}$ and $\{[A p_V] , [p_U] , [p_V] \}$ respectively.

\begin{rem} 
 This configuration gives indeed a bigon and a quadrilateral since the two Giraud circles are tangent at $[p_U]$ and $[p_V]$. The reader can find a detailed proof in \cite{parker_complex_2017a}, and we will show this fact again when deforming the domains.
\end{rem}

However, we will not use this domain for the deformation, but its image by the involution $\iota$ defined in Proposition \ref{prop_involution_WLC}. Recall that $\iota$ is compatible with the involution $\eta : \Im(\rho) \rightarrow \Im(\rho)$ given by $\eta(T) = T$ and $\eta(S) = S^{-1}$. Hence, $\eta(A) = U$ and $\eta(B) = V$. 

\begin{notat} We will denote by $\mathcal{J}_0^{\pm} = \iota(\mathcal{I}_0^{\pm})$. Hence,
 $\mathcal{J}_k^{+} = U^k \con{\mathfrak{B}}(p_U,p_V)$  and  $\mathcal{J}_k^{-} = U^k \con{\mathfrak{B}}(p_U,p_W)$.
 Furthermore, we will use some abusive language and still denote by $Q^{\pm}_k$ and $B^{\pm}_k$ the images of the quadrilaterals and the bigons by the involution $\iota$.
\end{notat}

The incidences and the combinatorics of the boundary of the new Ford domain in $\dh2c$ that we consider are given in Figure \ref{combi_ford_wlc_apres}. The quadrilateral $Q_0^{+}$ has vertices $[p_A]$, $[U^{-1}p_B]$, $[p_B]$ and $[U p_A]$
 and the bigon $B_0^{+}$ has vertices $[p_A]$ and $[p_B]$. The sides of $Q_0^{+}$ and $B_0^{+}$ are arcs of the Giraud circles $\partial_\infty\mathcal{J}_0^{+} \cap \partial_\infty\mathcal{J}_{-1}^{+}$ and $\partial_\infty\mathcal{J}_0^{+} \cap \partial_\infty\mathcal{J}_{0}^{-}$, which contain $\{[U^{-1}p_B] , [p_A] , [p_B] \}$ and $\{[U p_A] , [p_A] , [p_B] \}$ respectively. Furthermore, we denote the faces of the domain contained in $\h2c$ in the following way:

\begin{figure}[ht]
 \center
 \includegraphics[width=8cm]{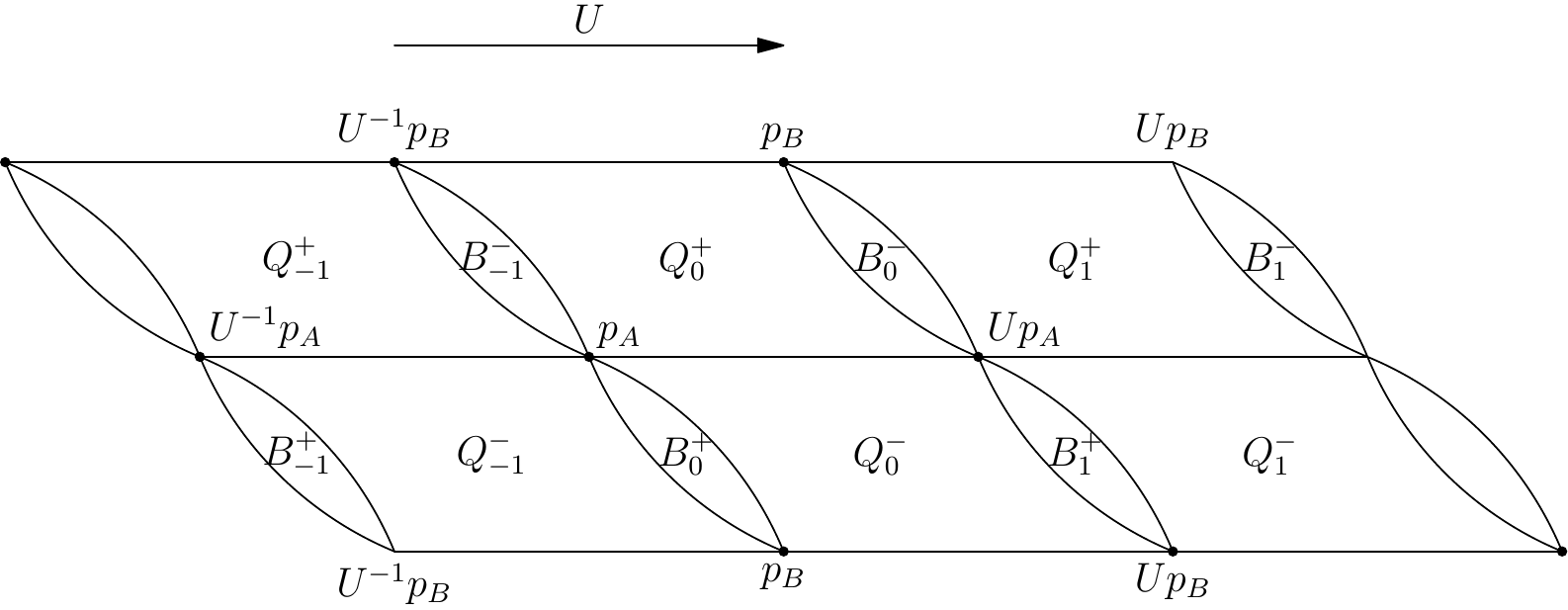}
 \caption{Combinatorics of the boundary at infinity of the Ford domain after the involution. The two horizontal boundaries are identifyed by a vertical translation.}\label{combi_ford_wlc_apres}
\end{figure}
 
\begin{notat}
 For $k \in \ZZ$, we define the 3-face $\mathcal{F}_k^+$ as
\begin{equation*}
 \mathcal{F}_k^+ = \left\{ [z] \in \mathcal{J}_k^+ \mid |\langle z, p_U \rangle| \leq \min (|\langle z, U^k p_W \rangle| , |\langle z, U^{k-1}p_W \rangle| )  \right\}.
\end{equation*} 
  Its boundary in $\mathcal{J}_k^+$ is then given by $\mathcal{J}_k^+ \cap (\mathcal{J}_k^- \cup \mathcal{J}_{k-1}^-)$. We also define the face $\mathcal{F}_k^-$ as
\begin{equation*}
 \mathcal{F}_k^- = \left\{ [z] \in \mathcal{J}_k^- \mid |\langle z, p_U \rangle| \leq \min (|\langle z, U^k p_V \rangle| , |\langle z, U^{k+1}p_V \rangle| )  \right\}.
\end{equation*} 
  Its boundary in $\mathcal{J}_k^-$ is then given by $\mathcal{J}_k^- \cap (\mathcal{J}_k^+ \cup \mathcal{J}_{k+1}^+)$.
\end{notat}

\begin{rem}
 The boundary of the face $\mathcal{F}_k^+$ is, a priori, the union of the two Giraud disks $\mathcal{J}_k^+ \cap \mathcal{J}_k^- $ and $\mathcal{J}_k^+ \cap  \mathcal{J}_{k-1}^-$. We are going to show, in Section \ref{sect_topo_faces}, that it is the case, and that these two disks are bi-tangent  at infinity during all the deformation.
\end{rem}

 We are going to deform the Ford domain of Parker-Will into a domain with the same local combinatorics. It will be either a Dirichlet domain, or a Ford domain centred \emph{outside} from $\h2c$ depending on whether $[U]$ is elliptic or loxodromic. From now on, all points and elements of the representation of $\z3z3$ with values in $\su21$ depend on the parameter $\alpha_2$, that will vary.
 
 Recall that the representations with parameter $\alpha_2 \in [-\alpha_2^{\lim} , \alpha_2^{\lim} ]$ have been studied by Parker and Will in \cite{parker_complex_2017a} by considering a Ford domain in $\h2c$ centred at the point $[p_A]$. They show, using the Poincaré polyhedron theorem, that the representation of $\z3z3$ is then discrete and faithful, and that the quotient of $\h2c$ by its image is a manifold. Furthermore, when $\alpha_2 = \alpha_2^{\lim}$, they show that the manifold at infinity is the Whitehead link complement, as described above.
 
 On the other hand, Parker, Wang and Xie study the representations with parameter $\alpha_2 \in ]\alpha_2^{\lim}, \frac{\pi}{2} [$ for which $[U]$ is of finite order $\geq 4$. They obtain these groups as the index 2 subgroup of some $(3,3,n)$ triangle group generated by involutions $I_1 , I_2 , I_3$. With our notation, we have
 $S = I_3 I_1$,
 $T = I_3 I_2$
 and $U = I_1 I_2$.
  They construct a Dirichlet domain in $\h2c$ for these groups, and show, using the Poincaré polyhedron theorem, that they are discrete, and that the quotients of $\h2c$ by these groups are manifolds. Whenever $n\geq 9$ , these Dirichlet domains are the same that the ones that we will construct below, but we will study the combinatorics of the intersections with different techniques, hoping for them to be also useful for representations not coming from triangle groups. We will discuss this with more details in Subsection  \ref{subsect_combi_ell}.

\section{Effective deformation: Proof} \label{sect_fin_preuve}
 We will assume in this section that the conditions (TF) on the topology of the faces, (LC) of local combinatorics and (GC) of global combinatorics are satisfied, and we will show Theorems \ref{thm_ch_dehn_eff_ell} and \ref{thm_ch_dehn_eff_lox}. Then, considering the Poincaré polyhedron theorem, we will discuss Theorem \ref{thm_unif_ch_dehn_ell} and Conjecture \ref{conj_unif_ch_dehn_lox}. We will show thereafter the conditions in Sections \ref{sect_topo_faces}, \ref{sect_combi_locale} and \ref{sect_combi_globale} respectively.
 From now on, we will assume that the combinatorics of the incidence of the faces $\mathcal{F}_k^{\pm}$, as well as their boundaries at infinity, is the one expected for $\alpha_2 \in ]0 , \alpha_2^{\lim} ]$
  when $[U]$ is loxodromic or unipotent, and for $\alpha_2 \in ]\alpha_{2}^{\lim} , \frac{\pi}{2}[$ when $[U]$ is elliptic of finite order $\geq 9$.

\begin{notat}
 We say that a parameter $\alpha_2$ is \emph{admissible} if $\alpha_2 \in ]0 , \alpha_2^{\lim}]$ or $\alpha_2 \in ]\alpha_2^{\lim} , \frac{\pi}{2}[$ and $[U]$ is elliptic of type $(\frac{1}{n},\frac{-1}{n})$ with $n \geq 9$.
\end{notat}

\begin{rem}
The parameters for which $[U]$ is of order $4,5,6,7$ or $8$ are also suitable for the following, but our approach using visual spheres to control global intersections does not allow us to conclude for these parameters. 
 However, the result is still true; a proof can be found in  \cite{parker_wang_xie}, where Parker, Wang and Xie show, using other techniques and parametrizing some triangle groups, that the global combinatorics of the bisectors $\mathcal{J}_k^\pm$ is the expected one when $[U]$ is elliptic of type $(\frac{1}{n},\frac{-1}{n})$ with $n \geq 4$.
\end{rem}

We begin by fixing some notation and recalling the uniformization result of Parker and Will, which gives a \CR {} structure on $WLC$ for the parameter $\alpha_2 = \alpha_2^{\lim}$.

\begin{notat}
 Denote by $D_0 \subset \h2c$ the Ford domain given by Parker and Will in \cite{parker_complex_2017a}. It corresponds to the parameter $\alpha_2 = \alpha_2^{\lim}$. Its boundary is given by the faces $\mathcal{F}_k^\pm$. Let $\partial_\infty D_0 \subset \dh2c$ be its boundary at infinity. The boundary of this last set in $\dh2c$ consists of the bigons $B_k^\pm$ and the quadrilaterals $Q_k^\pm$.
\end{notat}

  In \cite{parker_complex_2017a}, Parker and Will show that the side pairing of the faces $\mathcal{F}_k^\pm$ of $D_0$ are given by the group $\Gamma_\infty$ modulo the action of $[U]$: using the Poincaré polyhedron theorem, they show that in that case, the group $\Gamma_\infty$ is discrete and that the manifold at infinity, which is homeomorphic to $WLC$, is uniformizable.
 Hence, we know that $\Gamma_\infty \backslash \partial_\infty D_0$ is homeomorphic to the Whitehead link complement. We have seen, in Section \ref{sect_ch_cr_wlc}, that the image by the developing map of a neighbourhood of the cusp corresponding to $T_1$ is a horotube for the action of $[U]$. If we only consider a thickening of the quadrilaterals and the bigons before taking the quotient of $\partial_\infty D_0$ by $\Gamma_\infty$, we obtain the structure on $WLC$ besides a neighbourhood of the cusp number 1.

 Using the conditions (LC) and (GC) about the combinatorics, we are going to show the following lemma, which ensures that the domain $D_0$ can be deformed into a domain in $\h2c$ with boundary the faces $\mathcal{F}_k^\pm$.

\begin{lemme}
 If $\alpha_2$ is an admissible parameter, then the faces $\mathcal{F}_k^\pm$ border a domain $D(\alpha_2)$ in $\h2c$, obtained as a deformation of $D_0$.
\end{lemme}
\begin{proof}
 Let 
\[ 
 D(\alpha_2) = \left\{ [z] \in \h2c
  \mid
   \forall k \in \ZZ : 
   |\langle z, p_U \rangle|
    \leq 
    \min (|\langle z, U^k p_V  \rangle| , |\langle z, U^k p_W  \rangle|)
    \right\} .
\]
 
 It is the set of points of $\h2c$ which are "nearest" from $[p_U]$ than to the orbits by $[U]$ of $[p_V]$ and $[p_W]$. When $\alpha_2 = \alpha_2^{\lim}$, this domain is exactly the Ford domain $D_0$ of Parker and Will. It is then bordered by the faces $\mathcal{F}_k^{\pm}$ which are contained in the bisectors $\mathcal{J}_k^{\pm}$ for $k \in \ZZ$.
  If $\alpha_2 \in ]0, \frac{\pi}{2}[$ is an admissible parameter, the element $[U]$ generates a discrete subgroup and the boundary of the domain $D(\alpha_2)$ is contained in the bisectors $\mathcal{J}_k^{\pm}$. By the global combinatorics condition (GC), the bisectors only intersect their neighbours in the local combinatorics, and they determine the faces $\mathcal{F}_k^{\pm}$. Hence, these faces form the boundary of a domain $D(\alpha_2)$, which is obtained as a deformation of $D_0$.
\end{proof}

 The same holds for the boundary at infinity: the quadrilaterals and the bigons border a domain of $\dh2c$ which is obtained by deforming $\partial_\infty D_0$.

\begin{lemme}
 If $\alpha_2$ is an admissible parameter, the bigons and quadrilaterals $B_k^\pm$ and $Q_k^\pm$ border a domain in $\dh2c$, obtained as a deformation of $\partial_\infty D'_0$. It is the boundary at infinity of $D(\alpha_2)$; we will denote it by $\partial_\infty D(\alpha_2)$.
\end{lemme}

 We are now going to study the topology of the manifold $\Gamma(\alpha_2) \backslash \partial_\infty D(\alpha_2)$ for the admissible parameters, and we are going to show that it corresponds to the one expected by the surgery theorem of \cite{acosta_spherical_2016}, that we identified in Section \ref{sect_ch_attendues}. In order to do it, we are going to cut the domain $\partial_\infty D(\alpha_2)$ into two pieces: the first "near" the faces, that will give the structure besides the cusp associated to $T_1$, and the second "far" from the faces, that will give the solid torus glued to $T_1$ in order to obtain a Dehn surgery. 

 \begin{notat}
  Let $\mathcal{V}_0$ be a thickening of $\bigcup_{k \in \ZZ} (Q_k^+ \cup B_k^+ \cup Q_k^- \cup B_k^-)$ in $\partial_\infty D_0$, and $N_0$ its complement. 
 \end{notat}

\begin{lemme}\label{lemme_struct_defor_loin_pointe}
 If $\alpha_2$ is an admissible parameter, then $\mathcal{V}_0$ deforms into a thickening  $\mathcal{V}(\alpha_2)$ of $\bigcup_{k \in \ZZ} (Q_k^+ \cup B_k^+ \cup Q_k^- \cup B_k^-)$ in $\partial_\infty D(\alpha_2)$.
  The quotient $\Gamma(\alpha_2) \backslash \mathcal{V}(\alpha_2)$ is homeomorphic to $WLC$ minus the cusp corresponding to $T_1$.
\end{lemme}

\begin{proof}
 By the conditions on local and global combinatorics (LC) and (GC), the bigons and the quadrilaterals $B_k^\pm$ and $Q_k^\pm$ intersect with the same local combinatorics as for the parameter $\alpha_2^{\lim}$, and form a surface in $\dh2c$. Hence, we can consider a thickening $\mathcal{V}(\alpha_2)$ of this surface, which is a deformation of $\mathcal{V}_0$. A fundamental domain in $\mathcal{V}(\alpha_2)$ for the action of $\langle [U] \rangle$ is then given by a thickening of $B_0^+ \cup Q_0^+ \cup B_0^- \cup Q_0^+$. Since the side pairings are given by the same elements of the group, the topology of the quotient $\Gamma(\alpha_2) \backslash \mathcal{V}(\alpha_2)$ is the same as the one of $\Gamma_\infty \backslash \mathcal{V}_0$, which is precisely the Parker-Will structure on $WLC$ minus the cusp corresponding to $T_1$.
\end{proof}

 We will focus now on the other part of the structure, which allows to identify the manifold on which the \CR {} structure on $WLC$ minus a cusp can be extended by Lemma \ref{lemme_struct_defor_loin_pointe}. It will always be the expected Dehn surgery, as described in Section \ref{sect_ch_attendues}.

\begin{notat}
 Let $N(\alpha_2) = \partial_\infty D(\alpha_2) \setminus \mathcal{V}(\alpha_2)$.
\end{notat}

 The two following propositions complete the proof of Theorems \ref{thm_ch_dehn_eff_ell} and \ref{thm_ch_dehn_eff_lox}, by identifying the topology of the manifold $\langle [U] \rangle \backslash N(\alpha_2)$, glued to the \CR {} structure on $WLC$ given by $\langle [U] \rangle \backslash D(\alpha_2)$

\begin{prop}
 If $\alpha_2 \in ]0,\alpha_2^{\lim}[$, then the quotient $\langle [U] \rangle \backslash N(\alpha_2)$ is a solid torus, in which the curve $l_1^{-1}m_1^3$ is homotopically trivial.
\end{prop}

\begin{proof}
 The domain $\partial_\infty D_0$ is a horotube, bordered by the quadrilaterals $Q_k^{\pm}$ and the bigons $B_k^{\pm}$ when $\alpha_2 = \alpha_2^{\lim}$. If $\alpha_2 < \alpha_2^{\lim}$, the element $[U]$ is loxodromic. In this case, $N(\alpha_2) $ becomes homeomorphic to a cylinder in which the curve $l_1^{-1}m_1^3$ is homotopically trivial.
  This verification is analogous to the one made in the proof of the surgery theorem of \cite{acosta_spherical_2016}.
  
If $\alpha_2 < \frac{\pi}{6}$, then the cylinder becomes a surface of infinite genus, but, by the continuity of the deformation, the same curve is homotopically trivial, since it stays in a fixed compact set.
\end{proof}

\begin{prop}
 Let $\alpha_2 \in  ]\alpha_2^{\lim},\frac{\pi}{2}[$ be an admissible parameter. Then $[U]$ is elliptic of order $n \geq 9$. The quotient $\langle [U] \rangle \backslash N(\alpha_2)$ is then a solid torus, in which the curve $l_1m_1^{n-3}$ is homotopically trivial.
\end{prop}
\begin{proof}
 The domain $\partial_\infty D_0$ is a horotube, bordered by the quadrilaterals $Q_k^{\pm}$ and the bigons $B_k^{\pm}$ when $\alpha_2 = \alpha_2^{\lim}$. If $\alpha_2 > \alpha_2^{\lim}$ is an admissible parameter, the element $[U]$ is elliptic and $\partial_\infty D(\alpha_2) $ is a torus invariant by $[U]$. Since $[U] = \rho(m_1)$ has type $(\frac{1}{n} , \frac{-1}{n})$, this torus is not knotted, and $N(\alpha_2)$ is a solid torus invariant by $[U]$ in which the curve $l_1m_1^{n-3}$ is homotopically trivial. Indeed, with the notation of Remark \ref{rem_marquage_wlc}, the curves $l_0^n = m_1^n$ and $m_0 = l_1^{-1}m_1^{3}$ are homotopic in $N(\alpha_2)$ to one of the two $\CC$-circles invariant by $[U]$: the curve $l_1m_1^{n-3}$ is homotopically trivial in $N(\alpha_2)$. 
  We deduce that the quotient $\langle [U] \rangle \backslash N(\alpha_2)$ is also a solid torus in which the curve $l_1m_1^{n-3}$ is homotopically trivial.
This verification is analogous to the one made in the proof of the surgery theorem of \cite{acosta_spherical_2016}.
\end{proof}

 The two last propositions complete the proof of Theorems  \ref{thm_ch_dehn_eff_ell} and \ref{thm_ch_dehn_eff_lox}. It remains to discuss the use of the Poincaré polyhedron theorem for $\h2c$ to obtain Theorem \ref{thm_unif_ch_dehn_ell} and a part of Conjecture \ref{conj_unif_ch_dehn_lox}.
 If $\alpha_2 \in ] \frac{\pi}{6} , \frac{\pi}{2} [$ is an admissible parameter, the domain $D(\alpha_2)$ is a polyhedron invariant by the action of $[U]$, and has a side pairing given by $S$ and its conjugates by $[U]$. The hypotheses of the Poincaré polyhedron theorem for $\h2c$, as stated in \cite{parker_complex_2017a}, are satisfied. A complete proof of the theorem will appear in the book of Parker \cite{parker_complex_Toappear}.
  By applying this theorem to the domain $D(\alpha_2)$ and the side pairing between $\mathcal{F}_k^\pm$ in $\h2c$, we deduce that for the admissible parameters $\alpha_2$, the group $\Gamma(\alpha_2)$ is discrete, and that the structure on the manifold at infinity $\Gamma(\alpha_2) \backslash D(\alpha_2)$ is uniformizable. 
 This shows Theorem \ref{thm_unif_ch_dehn_ell} and a part of Conjecture \ref{conj_unif_ch_dehn_lox}.
 
  When $\alpha_2 \in ] 0, \frac{\pi}{6} [$, the domain $D(\alpha_2)$ has the same side pairings, but the faces $\mathcal{F}_k^{\pm}$ are no longer homeomorphic to a three dimensional ball: so we are no longer in the conditions to apply the Poincaré polyhedron theorem for $\h2c$. However, we can expect that a similar statement can be applied, for faces that are not balls, but we won't go in that direction, and limit ourselves to state Conjecture \ref{conj_unif_ch_dehn_lox}.

 In order to complete the proofs above, it remains to show the conditions on the topology of the faces (TF), local combinatorics (LC) and global combinatorics (GC). The rest of the article will consist of a rather technical proof of these conditions. Indeed, it is not trivial to reduce the global considerations to a finite number of verifications, and we will need to use the techniques involving visual spheres described in Part \ref{part_geom_background}.
\section{Topology of faces during the deformation (TF)}\label{sect_topo_faces}
 In this section, we are going to show the condition (TF) on the topology of the faces and we will give almost all the tools for showing the local combinatorics condition (LC) given in the strategy of proof of Section \ref{sect_strategie_de_preuve}.
First of all, remark that there is a change in the topology of bisectors at the point $\alpha_2 = \frac{\pi}{6}$. 

\begin{lemme}\label{lemme_nature_bissecteurs}
 If $\alpha_2 \in ]\frac{\pi}{6} , \frac{\pi}{2}[$, then $\mathcal{J}_k^{+}$ and $\mathcal{J}_k^{-}$ are metric bisectors. If $\alpha_2 = \frac{\pi}{6}$, then $\mathcal{J}_k^{+}$ and $\mathcal{J}_k^{-}$ are fans. If $\alpha_2 \in ]0, \frac{\pi}{6}[$, then $\mathcal{J}_k^{+}$ and $\mathcal{J}_k^{-}$ are Clifford cones.
\end{lemme}
\begin{proof}
By Proposition \ref{lox_spinal_spheres}, we only need to consider the signature of the restriction of the Hermitian form to the plane generated by $p_U$ and $p_V$. We have:
\begin{itemize}
 \item $\langle p_U , p_U \rangle = \langle p_V , p_V \rangle = 4 \cos^2(\alpha_2) - \frac{3}{2}$
 \item $\langle p_U , p_V \rangle = -\frac{3}{2}$
\end{itemize}
In the basis $(p_U,p_V)$, the determinant of the Hermitian form equals 
\[\left(4 \cos^2(\alpha_2) - \frac{3}{2}\right)^2 - \left(\frac{3}{2}\right)^2 = 4 \cos^2(\alpha_2)(4 \cos^2(\alpha_2) - 3).\]

Hence it is positive if $\alpha_2 \in [0, \frac{\pi}{6}[$, zero if $\alpha_2 =\frac{\pi}{6}$ and negative if $\alpha_2 \in ] \frac{\pi}{6} , \frac{\pi}{2} [$.
\end{proof}

\subsection{Incidence of points and bisectors}

We begin by checking the incidence of the points and the bisectors: the points will be the vertices of the faces which lie on spinal surfaces.

\begin{lemme}\label{lemme_incidences_pa}
 The point $[p_A]$ is contained in the bisectors $\mathcal{J}_0^{+}$,  $\mathcal{J}_0^{-}$, $\mathcal{J}_{-1}^{+}$ and $\mathcal{J}_{-1}^{-}$.
\end{lemme}

\begin{proof}
 It is enough to compute the Hermitian products of $p_A$ with $p_U, p_V, p_W , U^{-1}p_V$ and $U^{-1}p_W$ and check that they have the same modulus.
 We compute: $\langle p_A, p_U \rangle = e^{2i\alpha_2}$, $\langle p_A, p_V \rangle = -1$ and $\langle p_A, p_W \rangle = e^{2i\alpha_2}$.
  Furthermore, $\langle p_A, U^{-1}p_V \rangle = \langle U p_A, p_V \rangle = e^{2i\alpha_2}$ and $\langle p_A, U^{-1}p_W \rangle = \langle U p_A, p_W \rangle = -1$. All these products have modulus 1, which completes the proof.
\end{proof}

\begin{lemme}\label{lemme_incidence_pb}
 The point $[p_B]$ is contained in the bisectors $\mathcal{J}_0^{+}$, $\mathcal{J}_0^{-}$,$\mathcal{J}_{1}^{+}$ and $\mathcal{J}_{-1}^{-}$.
\end{lemme}

\begin{proof}
 It is enough to compute the Hermitian products of $p_B$ with $p_U, p_V, p_W , Up_V$ and $U^{-1}p_W$ and check that they have the same modulus.
 We compute: $\langle p_B, p_U \rangle = 1$, $\langle p_B, p_V \rangle = -e^{-2i\alpha_2}$ and $\langle p_B, p_W \rangle = 1$. Furthermore, $\langle p_B, U p_V \rangle = 1$ and $\langle p_B, U^{-1}p_W \rangle = \langle U p_B, p_W \rangle = 1$. All these products have modulus 1, which completes the proof.
\end{proof}

Considering the translation by $U^k$, we obtain the following corollary:

\begin{cor}\label{cor_incidences} For $k\in \mathbb{Z}$ we have the following incidences:
\begin{itemize}
 \item The bisector $\mathcal{J}_k^{+}$ contains the points $U^{k}p_A$, $U^{k+1}p_A$, $U^{k}p_B$ and $U^{k-1}p_B$.
  \item The bisector $\mathcal{J}_k^{-}$ contains the points $U^{k}p_A$, $U^{k+1}p_A$, $U^{k}p_B$ and $U^{k+1}p_B$.
\end{itemize}
\end{cor}

\begin{lemme}\label{lemme_inter_jk+_jk-}
 The intersections $\mathcal{J}_k^{+} \cap \mathcal{J}_k^{-} $ and $\mathcal{J}_k^{-} \cap \mathcal{J}_{k+1}^{+} $ are Giraud disks. If $\alpha_2 \neq 0$, their boundary is a smooth circle in $\dh2c$.
\end{lemme}
\begin{proof}
The hypotheses of Proposition \ref{prop_inter_biss_sym} are satisfied for the intersection $\mathcal{J}_k^{+} \cap \mathcal{J}_k^{-} $, since $p_V = Sp_U$ and $p_W = S^2p_U$. It is the same for $\mathcal{J}_k^{-} \cap \mathcal{J}_{k+1}^{+} $ since $p_V = Tp_U$ and $p_W = T^2p_U$. Hence, it is enough to compute $\frac{\langle p_U \boxtimes p_V , p_U \boxtimes p_V \rangle}{\langle p_U \boxtimes p_V , p_W \boxtimes p_U \rangle}$ and compare it to $\frac{2}{3}$; the other case is analogous.

We compute:
\begin{eqnarray*}
 \langle p_U \boxtimes p_V , p_U \boxtimes p_V \rangle = -4\cos^2(\alpha_2)(4\cos^2(\alpha_2) - 3) \\
 \langle p_U \boxtimes p_V , p_W \boxtimes p_U \rangle = -6\cos^2(\alpha_2)
\end{eqnarray*}
 The quotient is then equal to $\frac{2}{3}(4\cos^2(\alpha_2)-3)$, which is $\leq \frac{2}{3}$, with an equality if and only if $\alpha_2 = 0$.

%\textcolor{blue}{Détailler le calcul ?}
\end{proof}

\begin{cor} For all $k \in \mathbb{Z}$  we have: 
\begin{itemize}

 \item $\mathcal{J}_k^{+} \cap \mathcal{J}_k^{-} $ is a Giraud disk containing the points $U^{k}p_A , U^{k}p_B$ and $U^{k+1}p_A$.
 
 \item $\mathcal{J}_k^{-} \cap \mathcal{J}_{k+1}^{+} $ is a Giraud disk containing the points $U^{k}p_B , U^{k+1}p_A$ and $U^{k+1}p_B$.
 \end{itemize}
\end{cor}

\begin{rem}
 If $\alpha_2 = 0$ then $\frac{2}{3}(4\cos^2(\alpha_2)-3) =  \frac{2}{3}$, and the intersections $\mathcal{J}_k^{+} \cap \mathcal{J}_k^{-} $ and $\mathcal{J}_k^{-} \cap \mathcal{J}_{k+1}^{+} $ are hexagons with their opposite vertices identified. It is the limit case of Proposition \ref{prop_inter_biss_sym}; the boundary of each hexagon consists of three $\CC$-circles.
\end{rem}
 
 We are going to recall the definition of the three dimensional faces of the Ford domain in $\h2c$. They are contained in the bisectors $\mathcal{J}_k^{\pm}$ and will be deformed. Their boundary at infinity will consist of bigons and quadrilaterals, that will border the domain $\partial_{\infty}D(\alpha_2)$ of $\dh2c$. This domain, endowed with the side pairing given by the group $\Gamma$, gives \CR {} structures on some Dehn surgeries on the Whitehead link complement.

\begin{defn}
 For $k \in \ZZ$, we define the 3-face $\mathcal{F}_k^+$ as
\begin{equation*}
 \mathcal{F}_k^+ = \left\{ [z] \in \mathcal{J}_k^+ \mid |\langle z, p_U \rangle| \leq \min (|\langle z, U^k p_W \rangle| , |\langle z, U^{k-1}p_W \rangle| )  \right\}.
\end{equation*} 
  Its boundary in $\mathcal{J}_k^+$ is then given by $\mathcal{J}_k^+ \cap (\mathcal{J}_k^- \cup \mathcal{J}_{k-1}^-)$. We also define the 3-face $\mathcal{F}_k^-$ as
\begin{equation*}
 \mathcal{F}_k^- = \left\{ [z] \in \mathcal{J}_k^- \mid |\langle z, p_U \rangle| \leq \min (|\langle z, U^k p_V \rangle| , |\langle z, U^{k+1}p_V \rangle| )  \right\}.
\end{equation*} 
  Its boundary in $\mathcal{J}_k^-$ is then given by $\mathcal{J}_k^- \cap (\mathcal{J}_k^+ \cup \mathcal{J}_{k+1}^+)$.
\end{defn}

\begin{rem}
The boundary of the face $\mathcal{F}_k^+$ is, a priori, the union of the two Giraud disks $\mathcal{J}_k^+ \cap \mathcal{J}_k^- $ and $\mathcal{J}_k^+ \cap  \mathcal{J}_{k-1}^-$. We are going to show, in the following section, that it is indeed the case, and that, during the whole deformation, these two disks are bi-tangent at infinity.
\end{rem}

\subsection{Symmetry}\label{subsect_symetrie}
 The domain that we are going to deform admits a symmetry exchanging the faces $\mathcal{F}_k^+$ and $\mathcal{F}_k^-$. Thanks to this symmetry and to the invariance of the domain by $[U]$, it will be enough to check most of the statements only for the bisector $\mathcal{J}_0^+$ or $\mathcal{J}_0^-$ to prove them for the whole family.

Consider the involution $I$ of $\mathrm{U}(2,1)$ given by:
\[ I =
\begin{pmatrix}
1 & 0 & 0 \\
-\sqrt{2}e^{i\alpha_2} & -1 & 0 \\
-1 & -\sqrt{2}e^{-i\alpha_2} & 1
\end{pmatrix}
.\]

It satisfies $Ip_U = p_U$, $Ip_V = p_W$ and $Ip_W = p_V$. Furthermore, we have $IUI = U^{-1}$; we deudce that for all $k \in \ZZ$, $IU^k p_V = IU^kI p_W = U^{-k}p_W$, and hence:

 \begin{equation*}
  \forall k \in \ZZ \hspace{0.5cm} I\mathcal{J}_k^+ = \mathcal{J}_{-k}^-
 \end{equation*}

 In particular, the action of $[I]$ on $\cp2$ exchanges the bisectors $ \mathcal{J}_0^+$ and $ \mathcal{J}_0^-$.

\subsection{The intersection $\mathcal{F}_0^- \cap \mathcal{F}_{-1}^-$}

 We begin by studying the intersection of 3-faces of the form $\mathcal{F}_k^\pm \cap \mathcal{F}_{k+1}^\pm$ in $\con{\h2c}$. We are going to show that these two faces are bi-tangent in $\dh2c$, and that they don't intersect elsewhere. This will show the condition on the topology of faces (TF) of Section \ref{sect_strategie_de_preuve}. By symmetry, it is enough to study the intersection $\mathcal{F}_0^- \cap \mathcal{F}_{-1}^-$. We will use the following lemma, which determines the intersection $\mathfrak{E}(p_U , p_V) \cap \mathfrak{E}(p_W , U^{-1}p_W)$, in order to study the intersection of the bisectors $\mathcal{J}_0^-$ and $\mathcal{J}_{-1}^-$.

\begin{lemme}\label{lemme_extors_equilibres}
 The extors $\mathfrak{E}(p_U , p_V)$ and $\mathfrak{E}(p_W , U^{-1}p_W)$ form a balanced pair. Their intersection is the union of a real plane $\mathfrak{m}$ and a complex line $l$ given by:
 
 \begin{equation*}
  \mathfrak{m} = \left\{
 \begin{bmatrix}
 z_1 \\ z_2 \\ 1
\end{bmatrix}
  \in \cp2 \mid z_1 \in \RR , z_2 \in i \RR
   \right\}
   \cup 
   \left\{
   \begin{bmatrix}
 1 \\ z_2 \\ 0
\end{bmatrix}
  \in \cp2 \mid z_2 \in i \RR
   \right\}
    \cup 
   \left\{
   \begin{bmatrix}
 0 \\ 1 \\ 0
\end{bmatrix}
   \right\}
\end{equation*}  

\begin{equation*}
 l = \begin{bmatrix}
 \sin(\alpha_2) \\ -i \frac{\sqrt{2}}{2} \\ - \sin(\alpha_2)
\end{bmatrix}^{\perp}
\end{equation*}
 
\end{lemme}

\begin{proof}

 Consider the vectors 
\begin{equation*}
 f = \begin{pmatrix}
 -1 \\ -i \frac{\sqrt{2}}{2}\sin(\alpha_2) \\ 1
\end{pmatrix}
\text{ and }
 f' = \begin{pmatrix}
 1 \\ 0 \\ 1
\end{pmatrix}.
\end{equation*} 

 We have $\langle f, p_U \rangle = \langle f , p_V \rangle = \langle f', p_W \rangle = \langle f' , U^{-1}p_W \rangle = 0$, hence $[f]$ is the focus of $\mathfrak{E}(p_U , p_V)$ and $[f']$ is the focus of $\mathfrak{E}(p_W , U^{-1}p_W)$. We compute:
 
 \begin{eqnarray*}
  \langle p_u , f' \rangle &=& 2 \cos(\alpha_2)e^{-i\alpha_2} \\
  \langle p_V , f' \rangle &=& -2 \cos(\alpha_2)e^{-i\alpha_2} \\
  \langle p_W , f \rangle &=& (i \sin(\alpha_2)(4e^{-2i\alpha_2}+2)-2e^{-i\alpha_2})e^{-i\alpha_2}\\
  \langle U^{-1}p_W , f \rangle &=& (i \sin(\alpha_2)(4e^{2i\alpha_2}+2)+2e^{i\alpha_2})e^{-i\alpha_2}\\
  &=& -e^{-2i\alpha_2} \con{\langle p_W , f \rangle}
 \end{eqnarray*}

 Consequently $[f'] \in \mathfrak{E}(p_U,p_V)$ and $[f] \in \mathfrak{E}(p_W,U^{-1}p_W)$. Hence it is a balanced pair. By Theorem \ref{thm_inter_extors_equilibres}, the intersection of the two extors is given by a complex line $l$ and an $\RR$-plane. 
 The complex line is given by $l = l_{[f],[f']}$; we easily check that we have
 \begin{equation*}
 l = l_{[f],[f']} = \begin{bmatrix}
 \sin(\alpha_2) \\ -i \frac{\sqrt{2}}{2} \\ - \sin(\alpha_2)
\end{bmatrix}^{\perp}.
\end{equation*}
 
 We know that the points $[f],[f']$ and $[p_A]$ lie in the $\RR$-plane $\mathfrak{m}$ given in the statement. Since these points are not aligned, there is at most one $\RR$-plane containing them. By Corollary \ref{cor_incidences}, we know that $[p_A] \in \mathfrak{E}(p_U,p_V) \cap \mathfrak{E}(p_U,p_W) \cap \mathfrak{E}(p_U , U^{-1}p_W)$, and hence that $[p_A]$ is in $\mathfrak{E}(p_U,p_V) \cap \mathfrak{E}(p_W , U^{-1}p_W)$. Since $[p_A] \notin l$, the $\RR$-plane $\mathfrak{m}$ passing by $[f],[f']$ and $[p_A]$ is the $\RR$-plane of the intersection of the balanced pair of extors.
\end{proof}

 In order to understand the intersection of the faces $\mathcal{F}_0^{-}$ and $\mathcal{F}_{-1}^{-}$, we begin by considering the triple intersection of extors $\mathfrak{E}_0^+ \cap \mathfrak{E}_0^{-} \cap \mathfrak{E}_{-1}^{-}$. We will use in a crucial way the following lemma. 

\begin{lemme}\label{lemme_triple_inter_extors}
 The triple intersection of extors $\mathfrak{E}_0^+ \cap \mathfrak{E}_0^{-} \cap \mathfrak{E}_{-1}^{-}$ consists of two topological circles, one contained in the $\RR$-plane $\mathfrak{m}$ and the other contained in the complex line $l$. The triple intersection of bisectors $\mathcal{J}_0^+ \cap \mathcal{J}_0^{-} \cap \mathcal{J}_{-1}^{-}$ is the set $\{[p_A] , [p_B]\}$.
\end{lemme}
\begin{proof}
 By Corollary \ref{cor_incidences}, we know that $[p_A]$ and $[p_B]$ are in the triple intersection. Consider the triple intersection of extors $\mathfrak{E}_0^+ \cap \mathfrak{E}_0^{-} \cap \mathfrak{E}_{-1}^{-}$.
  We have:
\begin{eqnarray*}
\mathfrak{E}_0^+ \cap \mathfrak{E}_0^{-} \cap \mathfrak{E}_{-1}^{-} &
 = & \mathfrak{E}(p_U,p_V) \cap \mathfrak{E}(p_U,p_W) \cap \mathfrak{E}(p_U,U^{-1}p_W) \\
 &=& \mathfrak{E}(p_U,p_V) \cap \mathfrak{E}(p_U,p_W) \cap \mathfrak{E}(p_W,U^{-1}p_W) \\
 &=& \mathfrak{E}(p_U,p_W) \cap (\mathfrak{m} \cup l)
\end{eqnarray*}  
  Where $\mathfrak{m}$ and $l$ are the real plane and the complex line of Lemma \ref{lemme_extors_equilibres}. We consider first the intersection $\mathfrak{E}(p_U,p_W) \cap \mathfrak{m}$ with the real plane.
 For $r,s \in \mathbb{R}$, let
 \[q_{r,s} = \begin{pmatrix}
 r \\ i\sqrt{2}s \\ 1
\end{pmatrix} 
\text{ and }
 q_{r} = \begin{pmatrix}
 r \\ i\sqrt{2} \\ 0
\end{pmatrix} .\] 
 
The $\RR$-plane $\mathfrak{m}$ is then equal to the set
$ \{[p_A]\} \cup \left\{\right [q_{r}] \mid r \in \mathbb{R}  \} \cup \left\{ [q_{r,s}] \mid r,s \in \mathbb{R} \right\}.$
 Let us show that $\mathfrak{m} \cap \mathcal{J}_0^{-} = \{[p_A] , [p_B]\}$. Let $r,s \in \mathbb{R}$. 
 We have:  
 \begin{eqnarray*}
 | \langle p_U , q_{r} \rangle |^2 &=&  r^2 + 2 r\sin (\alpha_2)+ 1 \\
 | \langle p_W , q_{r} \rangle |^2 &=&  r^2 + 2 r\sin (\alpha_2)+ 1 + 8\cos^2(\alpha_2) \\
 \end{eqnarray*}

Hence, the points of the form $[q_r]$ are not in the triple intersection.
  We also compute:
 \begin{eqnarray*}
  |\langle p_U , q_{r,s} \rangle|^2 &=& r^2 + s^2 + 1 + 2r(2\cos(\alpha_2)^2 - 1) + 2(r-1)s\sin(\alpha_2) \\
| \langle p_W , q_{r,s} \rangle |^2 &=& (8  \cos (\alpha_2)^2 + 1) s^2 + r^2 + 2 (r-1) s \sin (\alpha_2) - 2r + 1.
\end{eqnarray*} 
 We deduce that $| \langle p_W , q_{r,s} \rangle |^2 = | \langle p_U , q_{r,s} \rangle |^2$ if and only if $ 4\cos^2(\alpha_2)r = 8\cos^2(\alpha_2)s^2$. Since $\alpha_2 \neq \pm \frac{\pi}{2}$, this condition is equivalent to $2s^2 - r = 0$. But we know that $\langle q_{r,s} , q_{r,s} \rangle = 2s^2 + 2r$. If this point lies in $\mathcal{J}_0^{-}$, then $s^2 + r \leq 0$ and $2s^2 - r \geq 3s^2 \geq 0$, with equality if and only if $r=s=0$.
  Hence $\mathfrak{m}  \cap \mathcal{J}_0^{-} = \{[p_A] , [p_B]\}$.
 
 Consider now the intersection with the complex line. 
 The vectors 
$
\begin{bmatrix}
1 \\ 0 \\ 1
\end{bmatrix}$
 and 
$\begin{bmatrix}
-1 \\ 2\sqrt{2}i\sin(\alpha_2) \\ 1
\end{bmatrix}
$ 
 form a basis of $\begin{bmatrix}
 \sin(\alpha_2) \\  -i\frac{\sqrt{2}}{2} \\ - \sin(\alpha_2)
\end{bmatrix}^{\perp} $. For $\mu \in \CC$, let
$q_\mu = \begin{pmatrix}
 -1 + \mu \\  2\sqrt{2}i\sin(\alpha_2) \\ 1+ \mu
\end{pmatrix} $.
 We have:
 \begin{eqnarray*}
  |\langle p_U , q_{\mu} \rangle|^2 &=& 4 \cos^2(\alpha_2) |\mu|^2 \\
| \langle p_W , q_{\mu} \rangle |^2 &=& 4(9 - 8\cos^2(\alpha_2))\cos^2(\alpha_2).
\end{eqnarray*} 
 The point $[q_\mu]$ is hence in the triple intersection if and only if $|\mu|^2 = 9 - 8\cos^2(\alpha_2)$.
But  $\langle q_{\mu} , q_{\mu} \rangle = 2(|\mu|^2 - 4\cos^2(\alpha_2) + 3)$, so for the points in the triple intersection this quantity equals $24\sin^2(\alpha_2)$. Hence, these points never lie in $\con{\h2c}$.
 
 We deduce that the triple intersection $\mathcal{J}_0^+ \cap \mathcal{J}_0^{-} \cap \mathcal{J}_{-1}^{-}$ is the set $\{[p_A] , [p_B]\}$.
\end{proof}

\begin{figure}[htbp]
\center
\begin{subfigure}{0.4\textwidth}
 \includegraphics[width= 6cm]{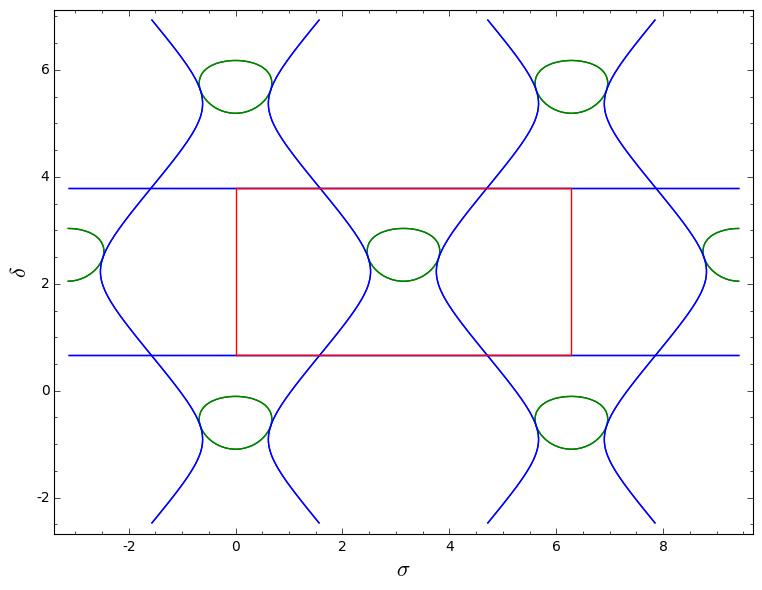}
      \caption{$\alpha_2 = \alpha_2^{\lim} \simeq 0,91$}
 \end{subfigure} \hspace{1cm}
 \begin{subfigure}{0.4\textwidth}
 \includegraphics[width= 6cm]{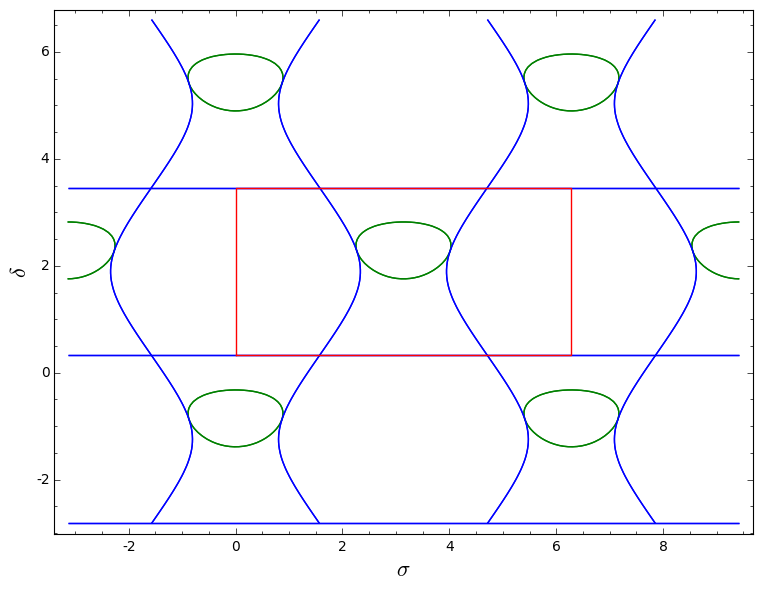}
      \caption{$\alpha_2 = 0,7$}
 \end{subfigure}  \\
 \begin{subfigure}{0.4\textwidth}
 \includegraphics[width= 6cm]{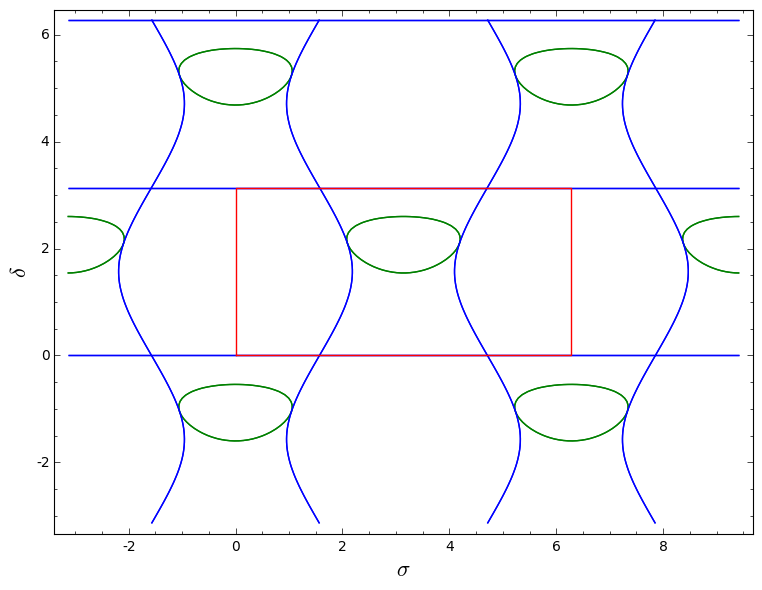}
      \caption{$\alpha_2 = \frac{\pi}{6}$}
 \end{subfigure} \hspace{1cm}
\begin{subfigure}{0.4\textwidth}
 \includegraphics[width= 6cm]{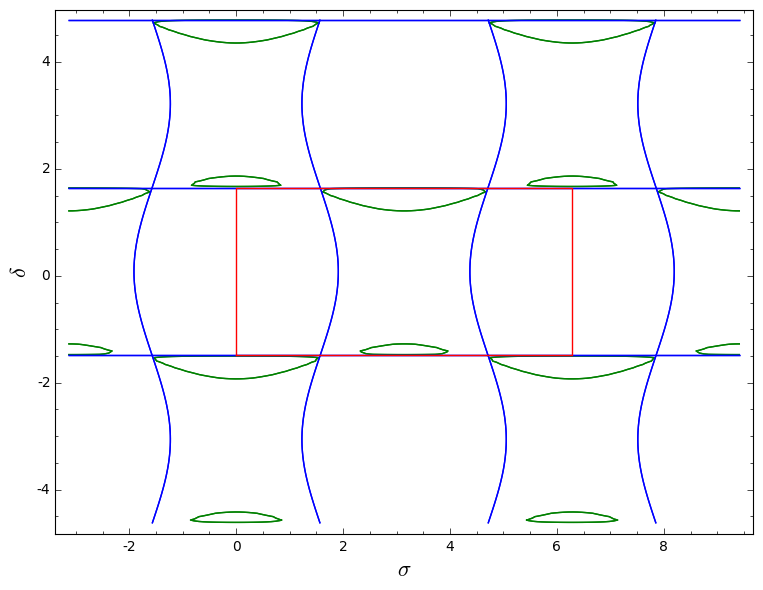}
      \caption{$\alpha_2 = 0,02$}\label{fig_inter_biss_non_connexe}
 \end{subfigure} \\
 \begin{subfigure}{0.4\textwidth}
 \includegraphics[width= 6cm]{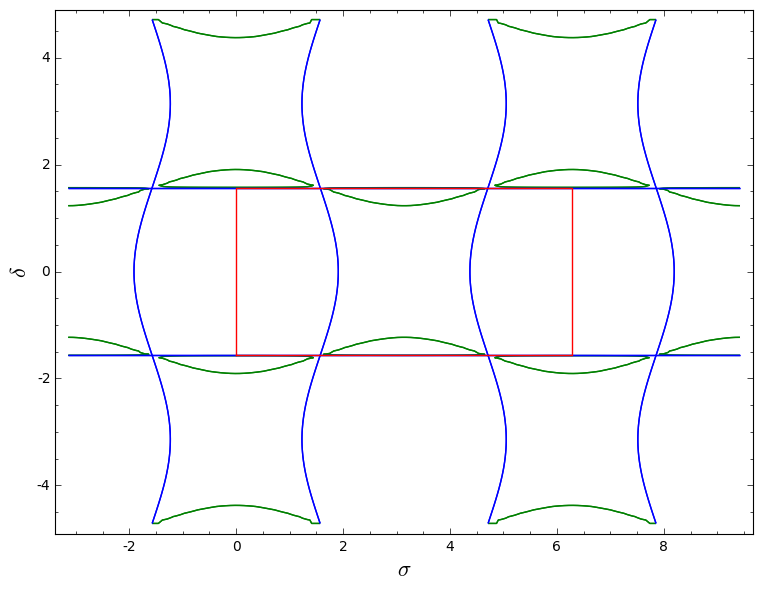}
      \caption{$\alpha_2 = 0$}
 \end{subfigure}
\caption{The curves of equations (\ref{eqn_inter_r_plan}) and (\ref{eqn_inter_c_plan}) (blue), the curve of the intersection with $\dh2c$ (green) and the rectangle of study (red) for $\alpha_2 \in \{0 ; 0.02 ; \frac{\pi}{6} ; 0.7 ; \alpha_2^{\lim}\}$. \label{fig_inter_biss_bon_cote}}
\end{figure}

By studying the torus $\mathfrak{T} = \mathfrak{E}_0^{-} \cap \mathfrak{E}_{-1}^{-} \subset \cp2$, and its intersections with $\mathfrak{E}_0^{+}$ and $\con{\h2c}$,
we obtain at last the following proposition. The combinatorics of these intersections is given by Figure \ref{fig_inter_biss_bon_cote}.

\begin{prop}\label{prop_inter_faces_f0-f-1-}
 The intersection $\mathcal{F}_0^{-} \cap \mathcal{F}_{-1}^{-}$ is reduced to $\{[p_A] , [p_B]\}$.
\end{prop}

\begin{proof}
 We are going to study the torus $\mathfrak{T} = \mathfrak{E}_0^{-} \cap \mathfrak{E}_{-1}^{-} \subset \cp2$, and its intersections with $\mathfrak{E}_0^{+}$ and $\con{\h2c}$. We know, by Lemma \ref{lemme_triple_inter_extors}, that $\mathfrak{T} \cap \mathfrak{E}_0^{+}$ is the union of two circles cutting $\mathfrak{T}$ into two pieces. We are going to show that $\mathfrak{T} \cap \con{\h2c}$ is always in the same side of these two circles, excepted the points $[p_A]$ and $[p_B]$; we will conclude the proof by a continuity argument.
 
 We begin by parametrizing the torus $\mathfrak{T}$. We know that $\mathfrak{T} = \mathfrak{E}(p_U , p_W) \cap \mathfrak{E}(p_U , U^{-1}p_W)$, hence we can parametrize it, by Proposition \ref{prop_param_inter_extors_deseq}, by 
 \begin{equation*}
 \mathfrak{T} = \left\{ [(p_W -e^{i\theta}p_U)\boxtimes (U^{-1}p_W -e^{i\phi}p_U)] \mid (\theta, \phi) \in (\RR/2\pi\ZZ)^2 \right\}
\end{equation*}  

 Let us compute the coordinates of these points. We have:
 \begin{equation*}
 (p_W -e^{i\theta}p_U)\boxtimes (U^{-1}p_W -e^{i\phi}p_U) = 
 p_W\boxtimes U^{-1}p_W -e^{-i\theta}p_U\boxtimes U^{-1}p_W -e^{-i\phi} p_W\boxtimes p_U
 \end{equation*}
 We compute:

\begin{eqnarray*}
 p_W\boxtimes U^{-1}p_W 
 &=&
  \sqrt{2}\cos(\alpha_2)e^{-2i\alpha_2}
  \begin{pmatrix}
 -3 \\ 0 \\ -3
\end{pmatrix} \\
  p_U\boxtimes U^{-1}p_W 
 &=&
  \sqrt{2}\cos(\alpha_2)e^{-2i\alpha_2}
  \begin{pmatrix}
 2e^{2i\alpha_2} \\ \sqrt{2}e^{i\alpha_2} \\ -1
\end{pmatrix} \\
 p_W\boxtimes p_U 
 &=&
  \sqrt{2}\cos(\alpha_2)e^{-2i\alpha_2}
  \begin{pmatrix}
 -1 \\ \sqrt{2}e^{-i\alpha_2} \\ 2e^{-2i\alpha_2}
\end{pmatrix}
\end{eqnarray*} 
 
 Let, for $(\theta,\phi) \in \RR^2/\ZZ^2$,
 \begin{eqnarray*}
  v(\theta , \phi) 
  &=& 
  \begin{pmatrix}
 3 \\ 0 \\ 3
\end{pmatrix}
- e^{-i\theta}
\begin{pmatrix}
 2e^{2i\alpha_2} \\ \sqrt{2}e^{i\alpha_2} \\ -1
\end{pmatrix}
-e^{-i\phi}
\begin{pmatrix}
 -1 \\ \sqrt{2}e^{-i\alpha_2} \\ 2e^{-2i\alpha_2}
\end{pmatrix} \\
&=&
\begin{pmatrix}
 -3 -2e^{i(2\alpha_2 - \theta)} + e^{-i\phi}\\
  -\sqrt{2}(e^{i(\alpha_2 - \theta)} + e^{i(-\alpha_2 - \phi )}) \\
   -3 + e^{-i\theta} -2e^{i(-2\alpha_2 - \phi)}
\end{pmatrix}
 \end{eqnarray*}
 The vector $v(\theta,\phi)$ is a multiple of $(p_W -e^{i\theta}p_U)\boxtimes (U^{-1}p_W -e^{i\phi}p_U)$. Hence, we have $\mathfrak{T} = \left\{ [v(\theta,\phi)] \mid (\theta, \phi) \in \RR^2/\ZZ^2 \right\}$.
 
 In order to simplify the following computations, we are going to change variables. Let
\[
  \sigma = \frac{\theta + \phi}{2} \text{ and }
  \delta = \frac{\theta - \phi}{2}.
\]
 With these new variables, we have
 \begin{equation*}
  v(\theta,\phi) = 
  e^{-i\sigma}
  \begin{pmatrix}
  -3e^{i\sigma} -2e^{i(2\alpha_2 - \delta)} + e^{i\delta}\\
  -2\sqrt{2}\cos(\alpha_2 - \delta) \\
   -3e^{i\sigma} -2e^{i(-2\alpha_2 + \delta)} + e^{-i\delta}
\end{pmatrix}.
 \end{equation*}
 
 We study now the intersection of $\mathfrak{T}$ with $\mathfrak{E}_0^+$. This intersection is given by the intersection of $\mathfrak{T}$ with the real plane $\mathfrak{m}$ and the complex line $l$ of Lemma \ref{lemme_extors_equilibres}.
 
  Consider first the intersection with $\mathfrak{m}$. The point $[v(\theta,\phi)]$ is in the $\RR$-plane $\mathfrak{m}$ if and only if the quotients of the first and third coefficients of the vector by the second one are imaginary, which is equivalent to:
  \begin{equation}\label{eqn_inter_r_plan}
  -3\cos(\sigma) - 2\cos(2\alpha_2 - \delta) + \cos(\delta) = 0
  \end{equation}
 Notice that if $\alpha_2 \in ]0, \frac{\pi}{2}[$, this equation has no solutions in $\delta$ if $\sigma = 0$.  
  Consider now the intersection with the complex line $l$. The point $[v(\theta,\phi)]$ is in the complex line $l$ if and only if
  $\left\langle v , \begin{bmatrix}
 \sin(\alpha_2) \\  -i\frac{\sqrt{2}}{2} \\ - \sin(\alpha_2)
\end{bmatrix} \right\rangle = 0$, which can be written in coordinates as:

\begin{equation*}
 2i( \cos(\alpha_2 - \delta) - \sin(\alpha_2)(2\sin(2\alpha_2 - \delta)-\sin(\delta)) = 0
\end{equation*}
but
\begin{eqnarray*}
 & & ( \cos(\alpha_2 - \delta) - \sin(\alpha_2)(2\sin(2\alpha_2 - \delta)-\sin(\delta))   
   \\
  &=&
 \cos(\alpha_2)\cos(\delta) + 2\sin(\alpha_2)\sin(\delta) - 2\sin(\alpha_2)(\sin(2\alpha_2)\cos(\delta) - \cos(2\alpha)\sin(\delta)) \\
   &=&
 \cos(\alpha_2)\cos(\delta)(1-4\sin^2(\alpha_2)) + 4\cos^2(\alpha_2)\sin(\alpha_2)\sin(\delta)  \\
 &=&
 \cos(\alpha_2)
 (-\cos(\delta) +2\cos(\delta)\cos(2\alpha_2) + 2\sin(2\alpha_2)\sin(\delta) ) \\
 &=&
 \cos(\alpha_2)
 (-\cos(\delta) +2\cos(2\alpha_2-\delta) )  \\
\end{eqnarray*}

 Hence the point $v(\theta,\phi)$ is in the complex line $l$ if and only if
 \begin{equation}\label{eqn_inter_c_plan}
 2\cos(2\alpha_2 - \delta) - \cos(\delta) = 0
 \end{equation}
 
If $\alpha_2 \neq 0$, we can re-write this condition is the following way:
\begin{equation}
 \tan(\delta) = \frac{1-2\cos(2\alpha_2)}{2\sin(2\alpha_2)}
\end{equation}
 
 Denoting by $\delta_0 = \arctan \left(\frac{1-2\cos(2\alpha_2)}{2\sin(2\alpha_2)} \right)$, a point of $\mathfrak{T}$ is in the complex line $l$ if and only if $\delta = \delta_0 \mod \pi$.
 From now on, we will consider the domain $\{(\delta,\sigma)\in \RR^2 \mid \delta_0 \leq \delta \leq \delta_0 + \pi , 0 \leq \sigma \leq 2\pi\}$ as a chart to study $\mathfrak{T}$, and we will denote by $v(\sigma,\delta)$ the parametrization.
 
 In order to study the intersection of $\con{\h2c}$ with $\mathfrak{T}$, we are going to study the function
 \begin{equation*}
 h : (\sigma,\delta) \mapsto \langle v(\sigma,\delta), v(\sigma,\delta) \rangle
 \end{equation*}

We have:
 \begin{eqnarray*}
 \frac{\partial h}{\partial \sigma}(\sigma , \delta)
 &=&
 2 \Re \left( \left\langle \frac{\partial(e^{i\sigma}v(\sigma,\delta))}{\partial \sigma} , e^{i\sigma}v(\sigma,\delta) \right\rangle \right) \\
 &=&
 2 \Re \left( \left\langle
 \begin{pmatrix}
  -3ie^{i\sigma}\\
  0 \\
   -3ie^{i\sigma} 
\end{pmatrix}
,
\begin{pmatrix}
  -3e^{i\sigma} -2e^{i(2\alpha_2 - \delta)} + e^{i\delta}\\
  -2\sqrt{2}\cos(\alpha_2 - \delta) \\
   -3e^{i\sigma} -2e^{i(-2\alpha_2 + \delta)} + e^{-i\delta}
\end{pmatrix}
 \right\rangle \right) \\
 &=&
 6 \left( 
 2(\sin(2\alpha_2-\delta-\sigma) + \sin(-2\alpha_2+\delta-\sigma))
 + \sin(\sigma - \delta) + \sin(\delta + \sigma)
  \right) \\
 &=&
 - 12 \sin(\sigma) (2\cos(2\alpha_2 - \delta) - \cos(\delta))
\end{eqnarray*}  
 
 Hence, this partial derivative is zero if and inly if $\delta \in \{ \delta_0 , \delta_0 + \pi \}$ where $\sigma \in \{0 , \pm \pi \} $.
 
 Fix $\delta_1 \in ]\delta_0 , \delta_0 + \pi[$. We have, in this case, $2\cos(2\alpha_2 - \delta_1) - \cos(\delta_1) > 0$. 
 
 Consider the function 
 \begin{equation*}
 h_{\delta_1} : 
 \begin{array}{rcl}
 [-\pi,\pi] & \rightarrow & \RR \\
 \sigma & \mapsto & h(\sigma,\delta_1)
 \end{array}
 \end{equation*}

Since $h_{\delta_1}'(\sigma) = \frac{\partial h}{\partial \sigma}(\sigma , \delta_1)$ has the same sign that $-\sin(\sigma)$, the function $h_{\delta_1}$ is decreasing on $[0 , \pi]$ and increasing on $[\pi,2\pi]$. Hence the values of $\sigma$ for which $h_{\delta_1}(\sigma) \leq 0$ form an interval (eventually empty) centred at $\pi$. 
 By Lemma \ref{lemme_triple_inter_extors}, we know that $\mathfrak{T} \cap \mathfrak{E}_0^+ \cap \con{\h2c}$ is reduced to $\{[p_A],[p_B]\}$, which are in the $\RR$-plane $\mathfrak{m}$. Hence the set $\mathfrak{T}  \cap \con{\h2c} \setminus \{[p_A],[p_B]\}$ is in the same connected component of $\mathfrak{T} \setminus \mathfrak{E}_0^+$ that the interval of points $v(\pi,\delta)$ for $\delta_0 < \delta < \delta_0 + \pi$. 
 
 When $\alpha_2 = \alpha_2^{\lim}$, Parker and Will show in \cite{parker_complex_2017a} that $\mathcal{F}_0^{-} \cap \mathcal{F}_{-1}^{-}$ is reduced to $\{[p_A] , [p_B]\}$. Hence, we know that for this parameter $\alpha_2$ the set $\mathfrak{T}  \cap \con{\h2c} \setminus \{[p_A],[p_B]\}$ is in the connected component of $\mathfrak{T} \setminus \mathfrak{E}_0^+$ that is not contained in $\mathcal{F}_0^-$. By continuity of the deformation, it is also true for all the parameters $\alpha_2 \in ]0 , \frac{\pi}{2}[$.
 
 We deduce that the intersection $\mathcal{F}_0^{-} \cap \mathcal{F}_{-1}^{-}$ is reduced to $\{[p_A] , [p_B]\}$.
\end{proof}

\begin{rem}
 The intersection of the bisectors $\mathcal{J}_0^{-}$ and $\mathcal{J}_{-1}^{-}$ is not always connected. For $\alpha_2$ close to $0$, it has two connected components, as shown in Figure \ref{fig_inter_biss_non_connexe}.
\end{rem}

\subsection{The faces in $\dh2c$ are well defined}
 Now, we are going to show that the 2-faces in $\dh2c$ are well defined and that the local incidences of the bisectors are the same as the ones of the Parker-Will structure for $\alpha_2 \in ]0 , \frac{\pi}{2}[ $. This fact will almost show the local combinatorics condition (LC) stated in Section \ref{sect_strategie_de_preuve}. We need to show that each spinal surface of the form $\partial_\infty \mathcal{F}_k^{\pm}$ is cut into a quadrilateral and a bigon with vertices in the orbits of $[p_A]$ and $[p_B]$ by powers of $[U]$.
 We take as starting point and inspiration the proof of Parker and Will in \cite{parker_complex_2017a}.
 Proposition \ref{prop_inter_faces_f0-f-1-} gives immediately the following Lemma:

\begin{lemme}\label{lemme_tangences_faces}
 The Giraud disks $\mathcal{J}_0^+ \cap \mathcal{J}_0^{-}$ and $\mathcal{J}_0^+ \cap \mathcal{J}_{-1}^{-}$ are tangent at $[p_A]$ and at $[p_B]$.
\end{lemme}

\begin{figure}[ht]
 \center
 \includegraphics[width = 8cm]{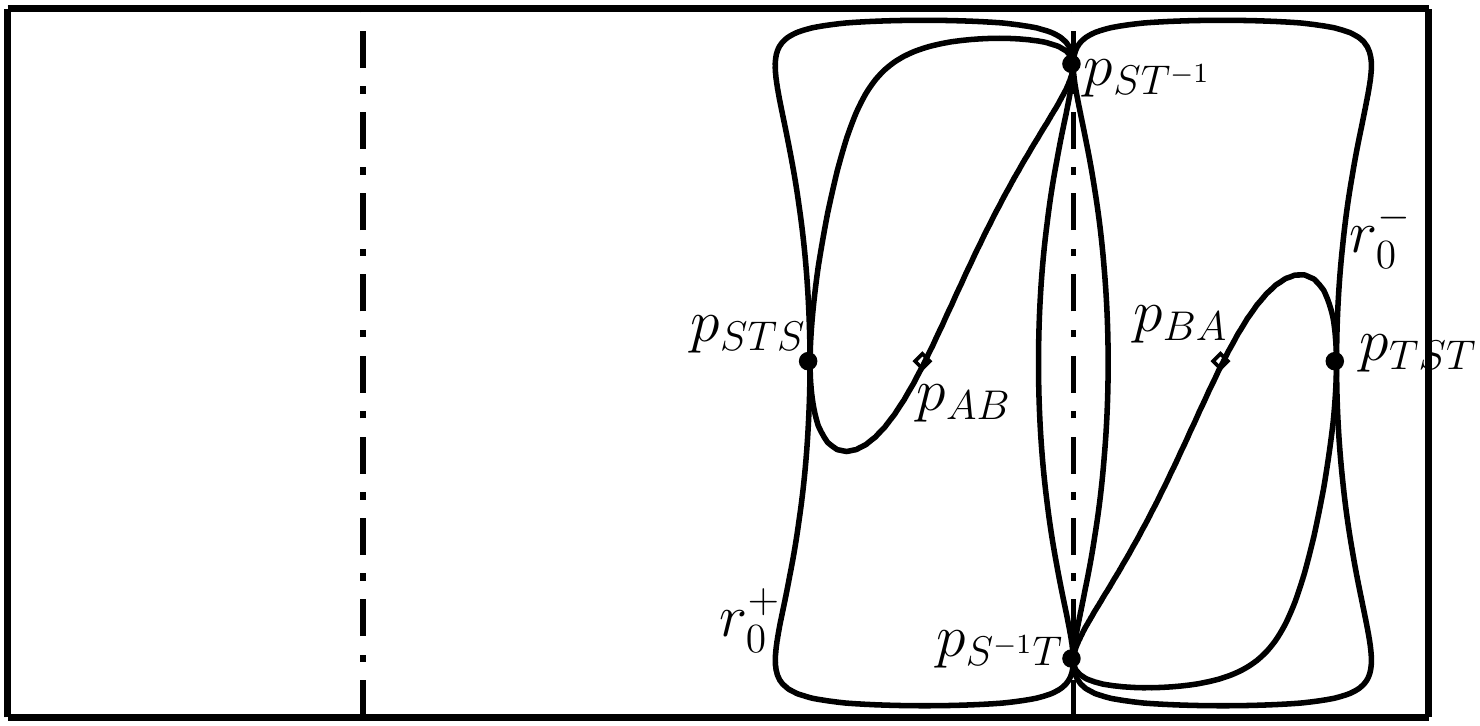}
 \caption{(From \cite{parker_complex_2017a}) The traces of the bisectors near $\mathcal{I}_0^+$ on $\partial_\infty \mathcal{I}_0^+$, in geographical coordinates on the sphere $\partial_\infty \mathcal{I}_0^+$. We see a bigon with vertices $p_{ST^{-1}}$ and $p_{S^{-1}T}$ and a quadrilateral with vertices $p_{ST^{-1}}$, $p_{TST}$, $p_{S^{-1}T}$ and $p_{STS}$.} \label{figure_ParkerWill}
\end{figure}

\begin{prop}\label{prop_faces_def1}
 For all $\alpha_2 \in ]0, \frac{\pi}{2}[$ and for all $k \in \mathbb{Z}$, the bigons $B_k^+$ and $B_k^-$, as well as the quadrilaterals $Q_k^+$ and $Q_k^-$, are well defined. If $\alpha_2 \in ]0, \frac{\pi}{6}[$, then the quadrilateral $Q_k^\pm$ is of genus $1$, i.e. that it is diffeomorphic to a torus minus a disk. 
\end{prop}
\begin{proof}
 By symmetry, it is enough to show that the bigon $B_0^+$ and the quadrilateral $Q_0^+$ are well defined through the deformation. By Lemma \ref{lemme_tangences_faces}, the Giraud circles $\partial_\infty\mathcal{J}_0^+ \cap \partial_\infty\mathcal{J}_0^{-}$ and $\partial_\infty\mathcal{J}_0^+ \cap \partial_\infty\mathcal{J}_{-1}^{-}$ are bi-tangent and cut the spinal surface $\partial_\infty\mathcal{J}_0^+$ in four connected components, as in Figure \ref{figure_ParkerWill}, which is taken from \cite{parker_complex_2017a} and traced for the parameter $\alpha_2 = \alpha_2^{\lim} $.
  The points $[p_A]$ and $[p_B]$ cut those Giraud circles into two arcs each. The arc of $\partial_\infty\mathcal{J}_0^+ \cap \partial_\infty\mathcal{J}_0^{-}$ containing $[Up_A]$ and the arc of  $\partial_\infty\mathcal{J}_0^+ \cap \partial_\infty\mathcal{J}_{-1}^{-}$ containing $[U^{-1}p_B]$ border a "quadrilateral" with vertices $[p_A], [U p_A], [p_B]$ and $[U^{-1}p_B]$. The two other arcs border a "bigon" with vertices $[p_A]$ and $[p_B]$.
  
  It remains to do a topological verification. Indeed, if  $\alpha_2 \in ]\frac{\pi}{6} , \frac{\pi}{2}[$, then the spinal surface $\partial_\infty J_0^+$ is a smooth sphere; if $\alpha_2 = \frac{\pi}{6}$, then $\partial_\infty J_0^+$ is a sphere with a singular point (which is the focus of the bisector); and if $\alpha_2 \in ]0 , \frac{\pi}{6}[ $, then $\partial_\infty J_0^+$ is a torus. In the two first cases, the two bi-tangent Giraud circles cut $\partial_\infty J_0^+$ into four topological disks, but in the last case we obtain three disks and a torus minus a disk. An ideal picture is given in Figures \ref{fig_paniers_1} and \ref{fig_paniers_2}. In order to identify the component that becomes of genus $1$ while deforming, it is enough to check that, when $\alpha_2 = \frac{\pi}{6}$, the singular point is in the interior of the quadrilateral $Q_0^+$.
  
  When $\alpha_2 = \frac{\pi}{6}$, we have:
  \begin{equation*}
  [p_U] = 
  \begin{bmatrix}
  1 \\ -\frac{\sqrt{2}}{2}e^{i\frac{\pi}{6}} \\ e^{i\frac{\pi}{3}}
  \end{bmatrix}
  \text{ , }
  [p_V] = 
  \begin{bmatrix}
  -e^{i\frac{\pi}{3}} \\ -\frac{\sqrt{2}}{2}e^{i\frac{\pi}{6}} \\ -1
  \end{bmatrix}
  \text{ , }
  [Up_A] = 
  \begin{bmatrix}
  1 \\ -\sqrt{2}e^{i\frac{\pi}{6}} \\ -1
  \end{bmatrix}
  \text{ and }
  [U^{-1}p_B] = 
  \begin{bmatrix}
  1 \\ \sqrt{2}e^{-i\frac{\pi}{6}} \\ -1
  \end{bmatrix}.
  \end{equation*}
  
  Hence, the focus of $\mathfrak{E}_0^+ = \mathfrak{E}(p_U,p_V)$ is the point
  \begin{equation*}
   [f] = \begin{bmatrix}
  1 \\ -i\sqrt{2} \\ -1
  \end{bmatrix}.
  \end{equation*}
  
  The points $[f]$, $[Up_A]$ and $[U^{-1}p_B]$ are hence aligned, in the same slice of the extor $\mathfrak{E}_0^+$. Consider the intersection of this complex line with $\dh2c$. It is the $\CC$-circle $\{ [q_\theta] \mid \theta \in [0,2\pi]\}$, where
  \begin{equation*}
   q_\theta = 
   \begin{pmatrix}
  1 \\ \sqrt{2}e^{i\theta} \\ -1
  \end{pmatrix}.
  \end{equation*}
  
  We have
  \begin{eqnarray*}
  \langle p_U , q_\theta \rangle = \langle p_V , q_\theta \rangle 
  &=& -1 - e^{i (\theta - \frac{\pi}{6})} + e^{-i \frac{\pi}{3}}
  \\
  &=& -e^{i\frac{\pi}{3}} - e^{i (\theta - \frac{\pi}{6})} \\
  &=& -2\cos(\frac{\pi}{4} + \frac{\theta}{2})e^{i(\frac{\pi}{12} + \frac{\theta}{2})}
  \end{eqnarray*}
 And hence:
 \begin{equation*}
 |\langle p_U , q_\theta \rangle|^2 = |\langle p_V , q_\theta \rangle|^2 = 4\cos^2(\frac{\pi}{4} + \frac{\theta}{2}) = 2(1 + \sin(\theta))
\end{equation*}   
  
 Furthermore, we have,
 
\begin{equation*}
 p_W = e^{i\frac{\pi}{3}}
 \begin{pmatrix}
  -1 \\ \sqrt{2} e^{i\frac{\pi}{6}} + \frac{\sqrt{2}}{2}e^{-i\frac{\pi}{6}} \\ 1
  \end{pmatrix}
  \text{ and }
   U^{-1}p_W = 
 \begin{pmatrix}
  1 \\ \frac{\sqrt{2}}{2}e^{i\frac{\pi}{6}} + \sqrt{2}e^{-i\frac{\pi}{6}} \\ -1
   \end{pmatrix}.
\end{equation*} 

We compute then 
  \begin{eqnarray*}
  |\langle p_W , q_\theta \rangle |^2
  &=& 
  |2 + 2e^{i(\theta - \frac{\pi}{6})} + e^{i(\theta + \frac{\pi}{6})}  |^2 \\
  &=&
  6\sqrt{3}\cos(\theta) + 2\sin(\theta) + 11
  \\
  |\langle U^{-1}p_W , q_\theta \rangle |^2
  &=& 
  |-2 + e^{i(\theta - \frac{\pi}{6})} + 2e^{i(\theta + \frac{\pi}{6})}  |^2 \\
  &=& 
  -6\sqrt{3}\cos(\theta) + 2\sin(\theta) + 11
  \end{eqnarray*}
  
  We deduce that $[q_\theta] \in \mathcal{J}_0^+$ if and only if $\theta \in [\frac{\pi}{6} , \frac{2\pi}{3}] \cup [\frac{4\pi}{3} , \frac{11\pi}{6}]$. But $Up_A = q_{\frac{4\pi}{3}}$, $U^{-1}p_B = q_{\frac{11\pi}{6}}$ and $f = q_{\frac{3\pi}{2}}$ hence $[f]$ is in the interior of $Q_0^+$.
\end{proof}

\begin{figure}[htbp]
\center
 \includegraphics[width = 4cm]{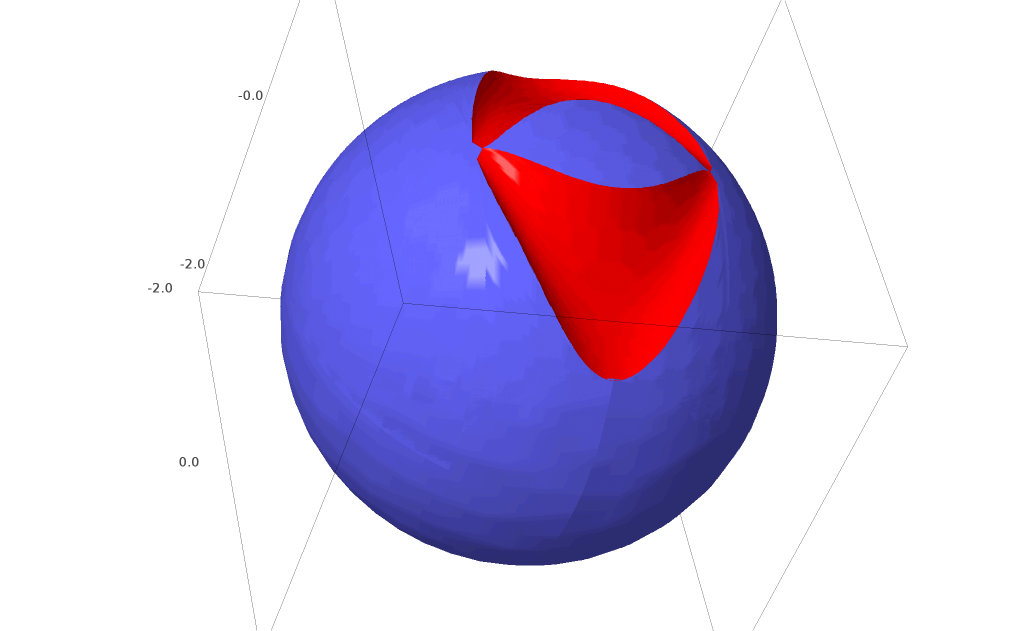}
 \includegraphics[width = 4cm]{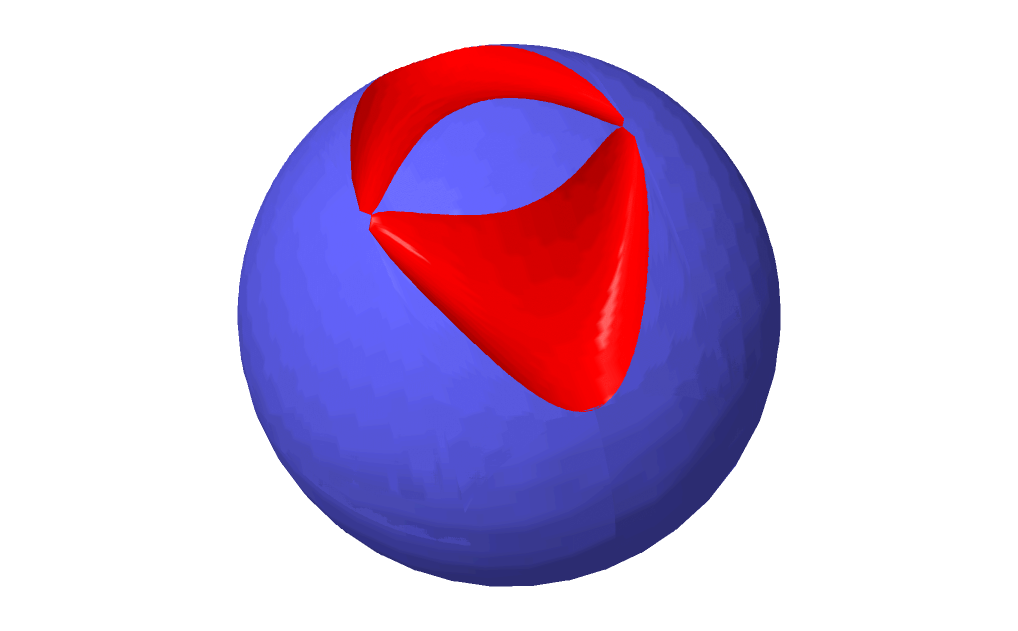}
 \includegraphics[width = 4cm]{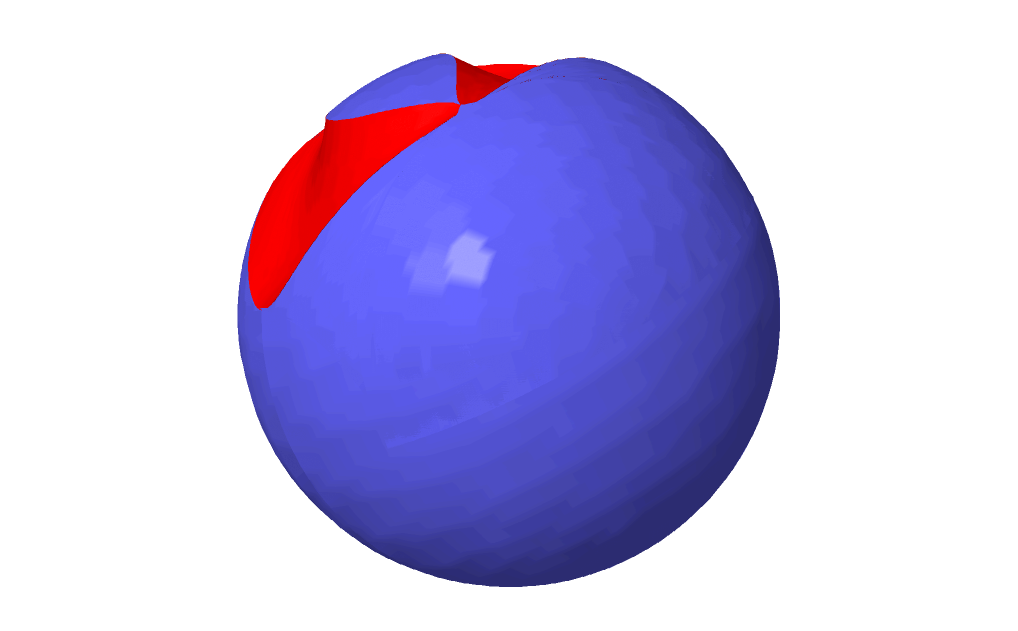}
 \caption{Ideal pictures of a face $\mathcal{F}_k^{\pm}$. The blue region is in $\dh2c$; the two red regions are Giraud disks in $\h2c$. The blue region is formed by a bigon $B_k^{\pm}$ and a quadrilateral $Q_k^{\pm}$. \label{fig_paniers_1}}
\end{figure}

\begin{figure}[htbp]
\center
 \includegraphics[width = 4cm]{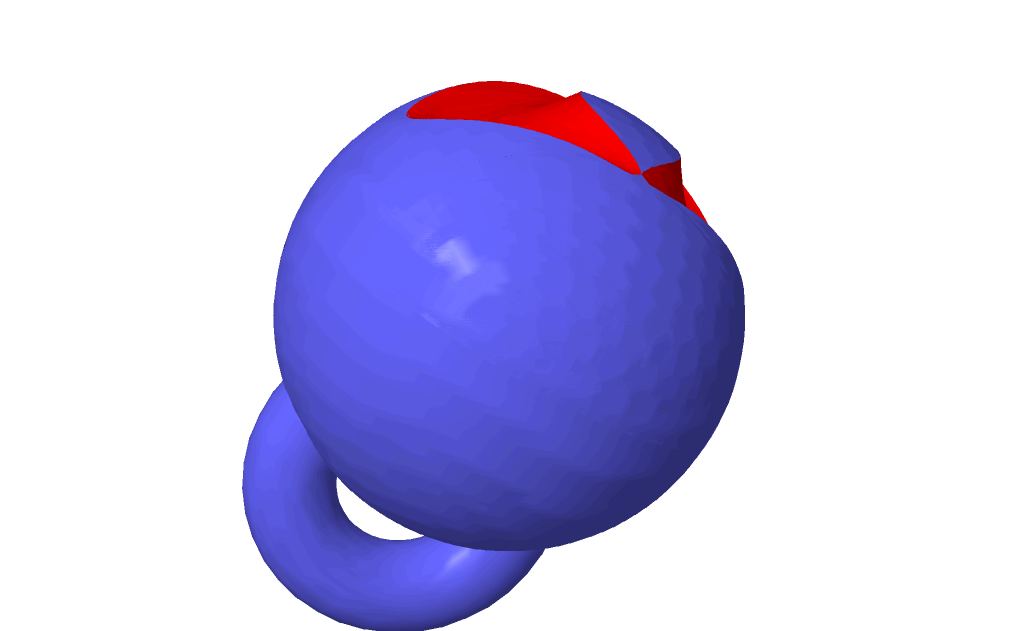}
 \includegraphics[width = 4cm]{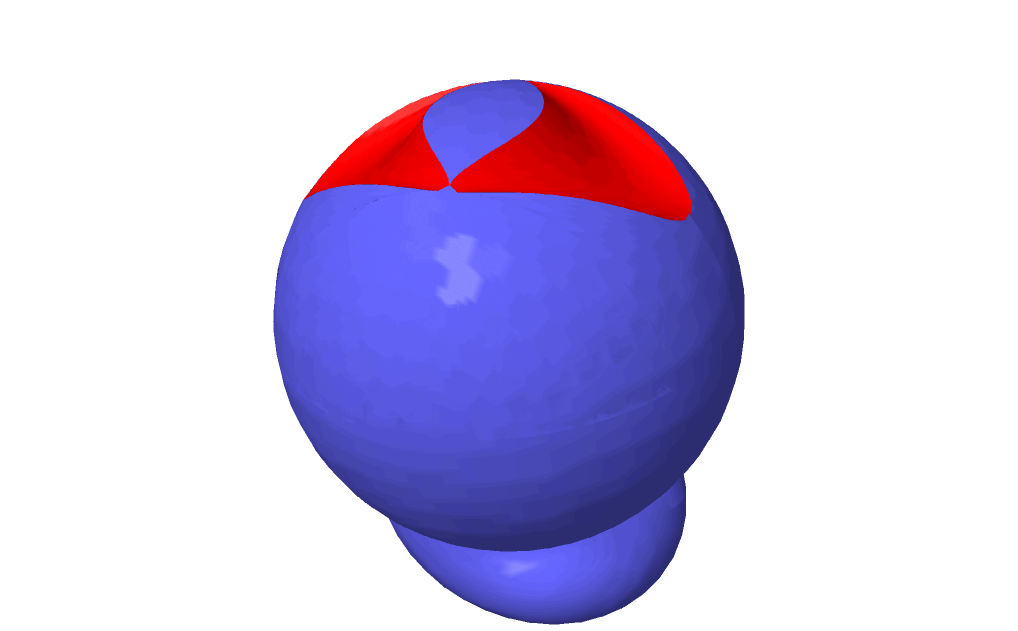} 
 \includegraphics[width = 4cm]{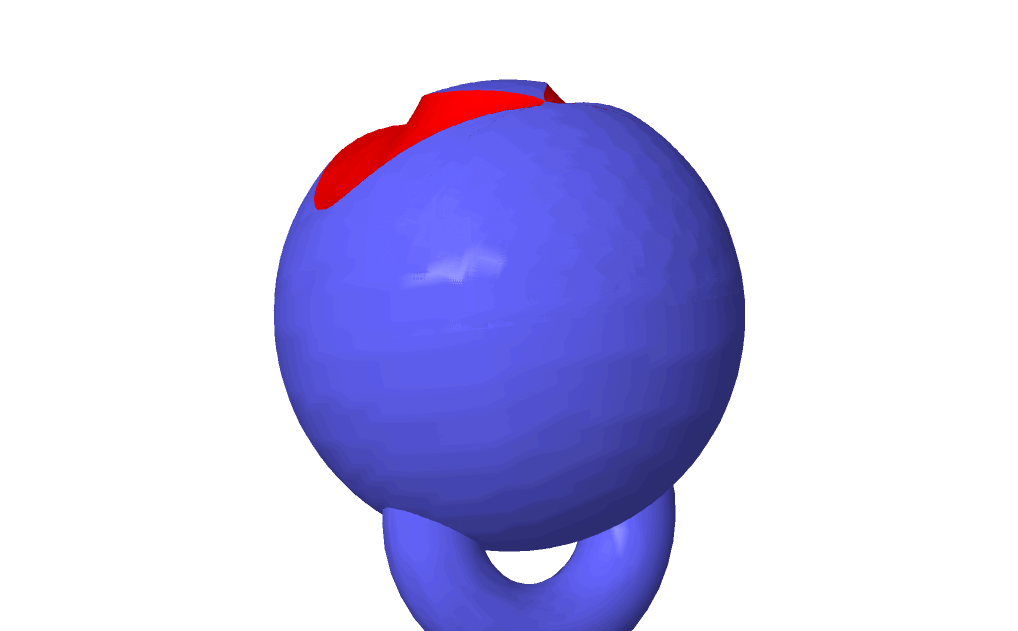}
 \caption{Ideal pictures of a face $\mathcal{F}_k^{\pm}$ for $\alpha_2 \in ] 0 , \frac{\pi}{6} [$. The blue region is in $\dh2c$; the two red regions are Giraud disks in $\h2c$. This time, the boundary at infinity of the face is on a torus, and we have a singular point in $\h2c$ which is not in the picture. The quadrilateral $Q_k^{\pm}$ has now a handle. \label{fig_paniers_2}}
\end{figure}

\section{Local combinatorics (LC)}\label{sect_combi_locale}
By combining the results of the last section, we can now show the local combinatorics condition (LC) stated in Section \ref{sect_strategie_de_preuve}. We are going to show that the local combinatorics of the faces stays constant through the deformation. More precisely, we have the two following propositions.
 The first is about the three-dimensional faces $\mathcal{F}_k^{\pm} \subset \h2c$; the second about their boundary at infinity, composed by bigons and quadrilaterals in $\dh2c$.
 
\begin{prop} We have the following intersections of the $3$-faces:
\begin{enumerate}
 \item $\mathcal{F}_0^+$ intersects $\mathcal{F}_0^-$ and $\mathcal{F}_{-1}^-$ along Giraud disks.
 
  \item $\mathcal{F}_0^-$ intersects $\mathcal{F}_0^+$ and $\mathcal{F}_{1}^+$ along Giraud disks.
 
 \item $\mathcal{F}_0^+ \cap \mathcal{F}_1^+$ is reduced to $\{[p_B], [Up_A]\}$
 
  \item $\mathcal{F}_0^- \cap \mathcal{F}_{-1}^-$ is reduced to $\{[p_B], [p_A]\}$
 \end{enumerate}
 \end{prop}
 
\begin{proof}
 The two first points are given by Lemma \ref{lemme_inter_jk+_jk-}. The third and the fourth are symmetrical, and given by Proposition \ref{prop_inter_faces_f0-f-1-}. 
 \end{proof}

 By considering the boundary at infinity, the next proposition follows:
\begin{prop}
 The faces $Q_0^+$ and $B_0^+$ intersect the faces contained in $\mathcal{J}_0^-$, $\mathcal{J}_{-1}^-$, $\mathcal{J}_{-1}^+$ and $\mathcal{J}_1^+$ exactly as in Figure \ref{combi_ford_wlc_apres}. More precisely:
 \begin{enumerate}
  \item $B_0^+$ intersects $Q_0^-$ and $Q_{-1}^-$ in the two arcs of its boundary.
  \item $Q_0^+$ intersects $Q_0^-$ , $B_0^-$, $Q_{-1}^-$ and $B_{-1}^-$ in the four arcs of its boundary.
  \item The intersection of $B_0^+$ and $Q_0^+$ with $B_{-1}^+$ and $Q_{-1}^+$ is reduced to $\{[U^{-1}p_B], [p_A]\}$.
    \item The intersection of $B_0^+$ and $Q_0^+$ with $B_{1}^+$ and $Q_{1}^+$ is reduced to $\{[p_B], [Up_A]\}$
  
\end{enumerate}  
\end{prop}

 By symmetry, we obtain the local combinatorics for the faces contained in $\mathcal{J}_0^-$:
 \begin{prop}
 The faces $Q_0^-$ and $B_0^-$ intersect the faces contained in $\mathcal{J}_0^+$, $\mathcal{J}_{1}^+$, $\mathcal{J}_{-1}^-$ and $\mathcal{J}_1^-$ exactly as in Figure \ref{combi_ford_wlc_apres}. More precisely:
 \begin{enumerate}
  \item $B_0^-$ intersects $Q_0^+$ and $Q_{1}^+$ in the two arcs of its boundary.
  \item $Q_0^-$ intersects $Q_0^+$ , $B_0^+$, $Q_{1}^+$ and $B_{1}^+$ in the four arcs of its boundary.
  \item The intersection of $B_0^-$ and $Q_0^-$ with $B_{-1}^-$ and $Q_{-1}^-$ is reduced to $\{[p_A], [p_B]\}$.
    \item The intersection of $B_0^-$ and $Q_0^-$ with $B_{1}^-$ and $Q_{1}^-$ is reduced to $\{[Up_A], [Up_B]\}$
  
\end{enumerate}  
\end{prop}

\section{Global combinatorics (GC)}\label{sect_combi_globale}
 At last, it remains to check the global combinatorics condition (GC) of the strategy of proof of Section \ref{sect_strategie_de_preuve}. This point is more technical that the two preceding ones: we are going to use explicit projections on visual spheres in order to show it. We begin by setting a strategy of proof for the condition (GC).

\subsection{Strategy}
 We want to show that, through the deformation, if $[U]$ is loxodromic or elliptic of type $(\frac{1}{n},\frac{-1}{n})$ with $n \geq 9$, then the intersections of the
faces $\mathcal{F}_k^+$ and $\mathcal{F}_k^-$ and of their boundaries at infinity are exactly the ones described by the local combinatorics. Since the domain is invariant by $[U]$ and since we have described the combinatorics of the intersections of the faces $\mathcal{F}_k^\pm$ as well as of
   $B_0^{\pm}$ and $Q_0^{\pm}$
    with the faces contained in the neighboring bisectors,  it is enough to show the two following propositions to obtain the global combinatorics of the intersections of the faces.
 
 \begin{prop}\label{prop_inter_biss_lox}
  If $[U]$ is loxodromic, then
\begin{itemize}
 \item $\mathcal{J}_0^+$ intersects $\mathcal{J}_k^+$ if and only if $k \in \{-1,0,1\}$.
 \item $\mathcal{J}_0^-$ intersects $\mathcal{J}_k^-$ if and only if $k \in \{-1,0,1\}$.
  \item $\mathcal{J}_0^+$ intersects $\mathcal{J}_k^-$ if and only if $k \in \{-1,0,1\}$. Furthermore, $\mathcal{J}_0^+ \cap \mathcal{J}_1^- = \{[Up_A]\}$ and $\mathcal{J}_0^+ \cap \mathcal{J}_{-1}^- = \{[U^{-1}p_B]\}$.
\end{itemize}    
 \end{prop}  

 \begin{prop}\label{prop_inter_biss_ell}
  If $[U]$ is elliptic of order $\geq 9$, then
\begin{itemize}
 \item $\mathcal{J}_0^+$ intersects $\mathcal{J}_k^+$ if and only if $k \equiv -1,0,1 \mod n$.
 \item $\mathcal{J}_0^-$ intersects $\mathcal{J}_k^-$ if and only if $k \equiv -1,0,1 \mod n$.
  \item $\mathcal{J}_0^+$ intersects $\mathcal{J}_k^-$ if and only if $k \equiv -1,0,1 \mod n$. Furthermore, $\mathcal{J}_0^+ \cap \mathcal{J}_1^- = \{[Up_A]\}$ and $\mathcal{J}_0^+ \cap \mathcal{J}_{-1}^- = \{[U^{-1}p_B]\}$.
\end{itemize}    
 \end{prop} 
 
 In order to show these propositions, we will project the bisectors $\mathcal{J}_k^{\pm}$ on the visual sphere of $[p_U]$. Then, we will obtain a family of disks invariant by the action of $[U]$ on the visual sphere. In the case where $[U]$ is loxodromic, this projection will be enough to show Proposition \ref{prop_inter_biss_lox}. In the case where $[U]$ is elliptic, we will need to refine the argument before completing the proof of Proposition \ref{prop_inter_biss_ell}.

 Recall that in \cite{parker_wang_xie}, Parker, Wang and Xie give a proof of Proposition \ref{prop_inter_biss_ell} for $n \geq 4$. The correspondence between their notation and ours is given by $\mathcal{J}_k^{+} = \mathcal{B}_{-2k}$ and $\mathcal{J}_k^{-} = \mathcal{B}_{-2k-1}$. Nevertheless, we are going to give a proof of Proposition \ref{prop_inter_biss_ell}, but using different tools, that could be applied besides from representations coming from triangle groups. However, our method will not allow us to to reach the cases $n = 4,5,6,7,8$, that are treated by Parker, Wang and Xie.

\subsection{First data}
 We establish first some results on the projection of the bisectors $\mathcal{J}_k^{\pm}$ on the visual sphere $L_{[p_U]}$ of $[p_U]$. We are going to identify those projections as disks with boundaries passing by the images of some remarkable points. Recall that we studied the visual spheres of points of $\cp2$ in Section \ref{sect_sphere_visuelle} of Part \ref{part_geom_background}.
At first, we establish a criterion for the tangency of a complex line with a spinal surface.

\begin{lemme}[Tangency criterion]\label{lemme_critere_tangence}
 Let $p,q \in \mathbb{C}^3 \setminus \{0\}$ such that $\langle p, p \rangle = \langle q ,q \rangle \neq 0$ and $\langle p, q \rangle \in \mathbb{R} \setminus \{0\}$. Let $[r] \in \mathfrak{S}(p,q)$. Then, the complex line $l_{[p],[r]}$ is tangent to $\mathfrak{S}(p,q)$ at $[r]$ if and only if there exists $\epsilon \in \{ \pm 1 \}$ such that:
 \begin{enumerate}
  \item $\langle p, r \rangle = \epsilon \langle q, r \rangle$
  \item $\langle q , p \rangle \neq \epsilon \langle p,p \rangle$
 \end{enumerate}
\end{lemme}
\begin{proof}
 We are going to show that $l_{[p] , [r]}$ intersects $\mathfrak{S}(p , q)$ only at $[r]$ if and only if the conditions of the statement are satisfied. An other point of $l_{[p] , [r]} \cap \mathfrak{S}(p , q)$ can be written as $[p + \lambda r]$ with $\lambda \in \mathbb{C}$, and satisfies:
 
 \begin{equation*}
  \langle p + \lambda r , p + \lambda r \rangle = 0 
 \end{equation*}
 \begin{equation*}
  |\langle p , p + \lambda r \rangle|^2 = |\langle q , p + \lambda r \rangle|^2 
 \end{equation*} 
 
 By expanding and using that $\langle r , r \rangle = 0$, that $\langle p, p \rangle = \langle q, q \rangle \in \mathbb{R}$ and that $|\langle p , r \rangle|= |\langle q , r \rangle|$, we obtain:
  \begin{equation}
  \langle p , p \rangle + 2 \Re (\lambda \langle p , r \rangle) = 0 
 \end{equation}
 \begin{equation}
  \langle p , p \rangle \Re (\lambda \langle p , r \rangle) = \Re ( \lambda \langle q , p \rangle \langle q ,  r \rangle)
 \end{equation} 
 Replacing $\Re (\lambda \langle p , r \rangle)$ in the second equation and noticing that $\langle q , p \rangle \in \mathbb{R} \setminus \{0\}$, we have:
  \begin{equation}
  - \frac{1}{2} \langle p , p \rangle = \Re (\lambda \langle p , r \rangle )
 \end{equation}
  \begin{equation}
  - \frac{1}{2}\frac{\langle p , p \rangle ^2}{\langle q , p \rangle} =  \Re ( \lambda  \langle q ,  r \rangle)
 \end{equation} 
  We obtain two equations of real lines in $\mathbb{C}$ for $\lambda$. The intersection of the lines is empty if and only if the lines are different and parallel. Since $|\langle p , r \rangle | = |\langle q , r \rangle|$, this is equivalent to the existence of $\epsilon \in \{\pm 1\}$ such that $\langle p , r \rangle  = \epsilon \langle q , r \rangle$ and $- \frac{1}{2} \langle p , p \rangle \neq \epsilon \left( - \frac{1}{2}\frac{\langle p , p \rangle ^2}{\langle q , p \rangle} \right)$, i.e. $\langle q , p \rangle \neq \epsilon \langle p,p \rangle$. 
 
 \end{proof}
 
 \begin{lemme}
 For $\alpha_2 \neq \pm \alpha_2^{\lim}, \pm \frac{\pi}{2}$, the complex lines $l_{[p_U] , [Up_A]}$ and $l_{[p_U] , [U^{-1}p_B]}$ are tangent to the sphere $\mathfrak{S}(p_U , p_V)$, respectively at $[U p_A]$ and $[U^{-1}p_B]$.
\end{lemme}

\begin{proof}
 By Lemmas \ref{lemme_incidences_pa} and \ref{lemme_incidence_pb}, we know that $[U p_A]$ and $[U^{-1}p_B]$ belong to $\mathfrak{S}(p_U,p_V)$. Furthermore, we have:
\begin{eqnarray*}
 \langle pu , pu \rangle = \langle pv , pv \rangle &=& 4\cos^2(\alpha_2) - \frac{3}{2} \\
 \langle p_U , Up_A \rangle = \langle p_V , Up_A \rangle &=& e^{-2i \alpha_2} \\
 \langle p_U , U^{-1}p_B \rangle = \langle p_V , U^{-1}p_B \rangle &=& 1 \\
 \text{ and } \langle p_V , p_U \rangle &=& - \frac{3}{2}.
\end{eqnarray*}
   By Lemma \ref{lemme_critere_tangence}, we have the tangencies of the statement.
\end{proof}

 Since, by Proposition \ref{prop_proj_biss_sphere_visuelle}, we know that the projection of a bisector of the form $\mathfrak{B}(p,q)$ on the visual sphere $L_{[p]}$ is a disk, we deduce the following corollary:
\begin{cor} The set
%Si $\alpha_2 \in ]\frac{\pi}{6}, \frac{\pi}{2} [$, alors 
$\pi_{[p_U]}(\mathcal{J}^{+}_0)$ is a disk whose boundary contains $\pi_{[p_U]}(l_{[p_U],[Up_A]})$ and $\pi_{[p_U]}(l_{[p_U],[U^{-1}p_B]})$.
\end{cor}

\begin{notat}
 We will denote those disks by $D_k^{\pm}$, in order to have $D_k^{\pm} = \pi_{[p_U]}(\mathcal{J}^{\pm}_0)$.
\end{notat}

Recall that we also have a symmetry associated to the involution $I \in \mathrm{U}(2,1)$ given by:
\[ I =
\begin{pmatrix}
1 & 0 & 0 \\
-\sqrt{2}e^{i\alpha_2} & -1 & 0 \\
-1 & -\sqrt{2}e^{-i\alpha_2} & 1
\end{pmatrix}
.\]

As seen in Section \ref{subsect_symetrie}, it satisfies $Ip_U = p_U$, $Ip_V = p_W$ and $Ip_W = p_V$. The action of $[I]$ on $\cp2$ fixes $[p_U]$ and exchanges the bisectors $ \mathcal{J}_0^+$ and $ \mathcal{J}_0^-$.

\subsection{The chart $\psi_{p'_U,p''_U}$ of $L_{[p_U]}$.}
We are going to make some computations in the chart  $\psi_{p'_U,p''_U}$ of $L_{[p_U]}$ in order to identify the intersections of the disks $D_k^{\pm}$ and deduce the global combinatorics of the intersections of the bisectors $\mathcal{J}_k^{\pm}$.

\begin{notat}
 In order to avoid heavy notation, if $[q] \in \cp2 \setminus \{[p_U]\}$ we will write $\psi(q)$ instead of $\psi_{p'_U,p''_U}(l_{[p_U],[q]}) $.
\end{notat}

\begin{rem} The image by $\psi$ of several remarkable points is easy to compute. Indeed, we have:

\begin{itemize}
 \item  $\psi(p'_U) = \infty$ and $\psi(p''_U) = 0$
 \item  $\psi(p_B) = 1$
 \item  $\psi(p_A) = \frac{-(8\cos^2(\alpha_2)+1) + \delta}{-(8\cos^2(\alpha_2)+1) - \delta} = \frac{\mathrm{tr}(U)+1 - \delta}{\mathrm{tr}(U)+1 + \delta}$.
\end{itemize} 

\end{rem}

The elements $U$ and $I$ of $\mathrm{U}(2,1)$ fix $p_U$ and hence have a natural projective actions on $L_{[p_U]}$. We are going to identify those two actions in the chart $\psi_{p'_U,p''_U}$. We begin by $U$:

\begin{rem}
The action of $U$ on $\cp2$ fixes $[p_U]$, $[p'_U]$ and $[p''_U]$. Hence it acts on $L_{[p_U]}$ by fixing $\psi(p'_U)$ and $\psi(p''_U)$. In the chart $\psi_{p'_U,p''_U}$, the action is hence given by either a rotation, or a homothety with center 0. We will give some details later on the corresponding angle or ratio.
\end{rem}

 Consider now the action of $I$.

\begin{lemme}
 The action of $I$ on $L_{p_U}$ is given, in the chart $\psi_{p'_U , p''_U}$, by $z \mapsto \frac{1}{z}$.
\end{lemme}
 \begin{proof}
 The involution $I$ satisfies
 $I p'_U = p''_U$, $Ip''_U = p'_U$ and $Ip_B = p_B$. Since $\psi(p'_U)= \infty$, $\psi(p''_U)= 0$ and $\psi(p_B) = 1$,
 the involution $I$ has acts on the chart $\psi_{p'_U , p''_U}$ by a projective map fixing $1$ and exchanging $0$ and $\infty$. Hence, it is $z \mapsto \frac{1}{z}$.
 \end{proof}

\subsection{The loxodromic side}\label{subsect_combi_lox}
 In this subsection, we are going to study the case where $[U]$ is loxodromic. We are going to study the relative position of the disks $D_k^{\pm}$ in $L_{[p_U]}$, which will be sufficient to prove the global combinatorics of the intersections of the bisectors $\mathcal{J}_k^{\pm}$. We are going to show that the disks are placed as in Figure \ref{fig_proj_anneaux_lox}.

\begin{figure}[ht]
 \center
 \includegraphics[width = 7cm]{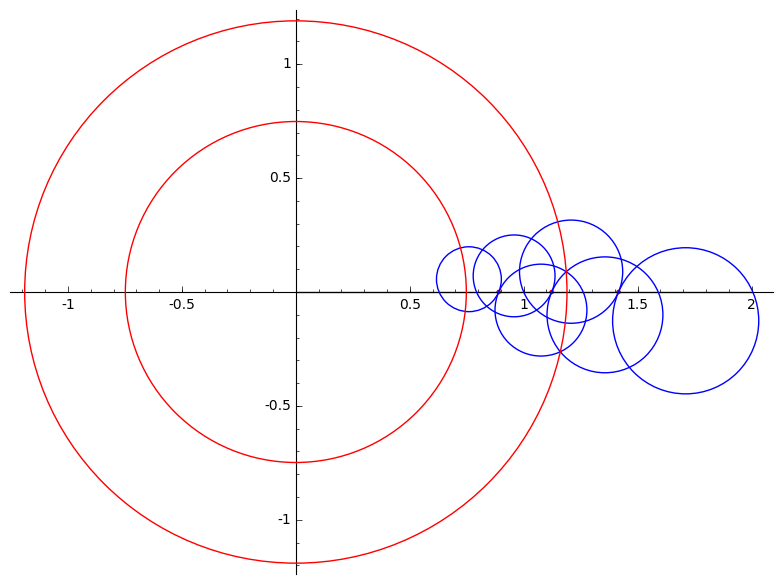}
 \caption{The disks $D_k^{\pm}$ (in blue) for $\alpha_2 = 0.91$. \label{fig_proj_anneaux_lox}}
\end{figure}

 We consider here the parameters $\alpha_2 \in ] \frac{\pi}{6}, \alpha_2^{\lim}  [$. By Remark \ref{rem_alpha2_lox}, we know that for such $\alpha_2$, the element $U$ is loxodromic and has eigenvalues $1 , e^{l}$ and $e^{-l}$. The length $l \in ]0 , \mathrm{argch}(\frac{5}{2}) [$ can also be used as a parameter for the deformation. It is related to $\alpha_2$ by the equation:
 \begin{equation*}
  2\cosh(l) + 1 = \mathrm{tr}(U) = 8\cos^2(\alpha_2)
 \end{equation*}

In this case, recall also that, by setting $\delta = 2\sinh(l)$, we have:
\[
[p'_U] = \begin{bmatrix}
2(2e^{2i\alpha_2}+1) \\
 -\sqrt{2}e^{i\alpha_2}(2e^{2i\alpha_2}+1 + \delta) \\
 -(8\cos^2(\alpha_2)+1) -\delta
\end{bmatrix}
\text{ and }
[p''_U] = \begin{bmatrix}
2(2e^{2i\alpha_2}+1) \\
 -\sqrt{2}e^{i\alpha_2}(2e^{2i\alpha_2}+1 - \delta) \\
 -(8\cos^2(\alpha_2)+1) + \delta
\end{bmatrix}
\]

\begin{rem}
 In the loxodromic side, $[p'_U], [p''_U] \in \dh2c$, and $[p_U], [p'_U], [p''_U]$ form an auto-polar triangle. We compute the following Hermitian products:
\begin{eqnarray*}
 \langle p'_U , p'_U \rangle = \langle p''_U , p''_U \rangle &=& 0 \\
 \langle p_U , p_U \rangle &=& \cosh(l) - 1 \\
 \langle p_U , p'_U \rangle = \langle p_U , p''_U \rangle &=& 0 \\
 \langle p'_U , p''_U \rangle &=& - 16 \sinh^2(l)
\end{eqnarray*}

\end{rem}

\begin{rem} In the case where $U$ is loxodromic, we have $\psi(p_A) =  \frac{\mathrm{tr}(U)+1 - \delta}{\mathrm{tr}(U)+1 + \delta} = \frac{1 + e^{-l}}{1 + e^{l}} = e^{-l}$.
\end{rem}

\begin{rem}
The action of $U$ on $L_{[p_U]}$ fixes $l_{[p_U],[p'_U]}$ and $l_{[p_U],[p''_U]}$. In the chart $\psi_{p'_U,p''_U}$, since $Up'_U = e^{l}p'_U$ and $Up''_U = e^{l}p''_U$,
 the action of $U$ is a homothety centered at $0$ and with ratio $e^{2l}$; it is given by $z \mapsto e^{2l}z$.
\end{rem}

\begin{rem}
 Since $\psi(p_B) = 1$, $\psi(p_A) = e^{-l}$ and $U$ acts by a homothety of ratio $e^{2l}$, we deduce that for all $k\in \mathbb{Z}$ we have
 
  \[\psi(U^kp_B) = e^{2kl} \text{ and } \psi(U^kp_A) = e^{(2k-1)l}.\]
\end{rem}

We deduce from the actions of $U$ and $I$ a result on the intersections:
\begin{prop}\label{prop_tangence_disques_lox}
 The disks $D_0^+$ and $D_1^{-}$ are tangent at the point $\pi_{[p_U]}(l_{[p_U],[Up_A]})$.
 
  The disks $D_0^+$ ans $D_{-1}^{-}$ are tangent at the point $\pi_{[p_U]}(l_{[p_U],[U^{-1}p_B]})$.
\end{prop}

\begin{proof}
Recall that $D_0^+$ is a disk whose boundary contains $\psi(U^{-1}p_B)$ and $\psi(Up_A)$, and that $D_1^{-}$ is a disk whose boundary contains $\psi(U^2p_B)$ and $\psi(Up_A)$.
 Let $\theta$ be the angle at $\psi(Up_A)$ between $\partial D_0^+$ and the real axis. Since the map $z \mapsto \frac{1}{z}$ is conformal, and it exchanges $D_0^+$ and $D_0^-$, the angle between $\partial D_0^-$ and the real axis at $\psi(p_A) = \frac{1}{\psi(Up_A)}$ equals also $\theta$. Since $D_1^{-}$ is obtained from $D_0^-$ by a homothety of ratio $e^{2l}$, the angle between the real axis and $\partial D_1^{-}$ at $\psi(Up_A)$ equals $\theta$. Since $\partial D_1^{-}$ and $\partial D_0^{+}$ intersect the real axis $\psi(Up_A)$ with the same angle, they are tangent.
 
 The proof for the other tangency is analogous.
\end{proof}

\begin{cor}
 The bisectors $ \mathcal{J}_0^+$ and $ \mathcal{J}_1^-$ are tangent at $[Up_A]$.
 The bisectors  $ \mathcal{J}_0^+$ and $ \mathcal{J}_{-1}^-$ are tangent at $[U^{-1}p_B]$.
\end{cor}
\begin{proof}
 By the proposition above, we know that the intersection of $ \mathcal{J}_0^+$ and $ \mathcal{J}_1^-$ is contained in the line $l_{[p_U],[Up_A]}$. By Lemma \ref{lemme_tangences_faces}, this line is tangent to $\partial_\infty \mathcal{J}_0^+$ at $[Up_A]$: hence the intersection contains exactly one point.
\end{proof}

 The following proposition is a key result in order to determine the relative position of the disks $D_k^{\pm}$.
We are going to show that each disk is contained in an annulus centred at $0$  with explicit radii; we will see later that the annuli are obtained by homotheties with center $0$. One of the annuli that we consider is pictured in Figure \ref{fig_proj_anneaux_lox}. The effective bounds that we give in the following proposition will be enough to prove the global combinatorics of the intersections of the disks $D_k^{\pm}$ and the bisectors $\mathcal{J}_k^{\pm}$.
 
\begin{prop}
 The disk $D_0^{+}$ is contained in the annulus of center $0$ and radii $e^{\frac{-5}{2}l}$ and $e^{\frac{3}{2}l}$. 
\end{prop}

\begin{proof}
 We are going to show that the circles of radii $e^{\frac{-5}{2}l}$ and $e^{\frac{3}{2}l}$ do not intersect the disk $D_0^+$.
  This comes to show that, if $\theta \in \mathbb{R}$, then any point with image $e^{i\theta} e^{\frac{3l}{2}}$ or $e^{i\theta} e^{-\frac{5l}{2}}$ by $\psi_{p'_U , p''_U}$ belongs to the bisector $\mathfrak{B}(p_U,p_V)$.
  
  We begin by the first case. The second is analogous; we will make a remark during the computation in order to check it. We want to show that no point of the form $q_{\theta,\mu} = p''_U+ e^{\frac{3l}{2}} e^{i\theta} p'_U + \mu p_U $, where $\theta \in \mathbb{R}$ and $\mu \in \mathbb{C}$, belongs to the bisector $\mathfrak{B}(p_U,p_V)$. We make a proof by contradiction, and suppose that there exists a point of the form $q_{\theta,\mu}$ in $\mathfrak{B}(p_U,p_V)$.
  
 We have two conditions for $q_{\theta,\mu}$ to belong to $\mathfrak{B}(p_U,p_V)$:
 \begin{enumerate}
  \item $\langle q_{\theta,\mu} , q_{\theta ,\mu} \rangle \leq 0$
  \item $|\langle q_{\theta,\mu} , p_U \rangle| = |\langle q_{\theta,\mu} , p_V \rangle|$
 \end{enumerate}
   
   They can be written as:
   
   \begin{enumerate}
    \item $2 \Re( e^{\frac{3l}{2}}e^{i\theta} \langle p''_U , p'_U \rangle) + |\mu|^2\langle p_U , p_U \rangle \leq 0$
    \item $|\mu||\langle p_U , p_U \rangle| = |\langle p''_U , p_V \rangle + e^{\frac{3l}{2}}e^{i\theta} \langle p'_U , p_V \rangle + \overline{\mu} \langle p_U , p_V \rangle|$
   \end{enumerate}
   
   By the computations above, and setting $h_1 = \langle p'_U , p_V \rangle $ and $h_2 = \langle p''_U , p_V \rangle $ to simplify the notation, the conditions become:
   
 \begin{enumerate}
  \item $|\mu|^2 \leq 2e^{\frac{3l}{2}}\frac{16\sinh^2(l)}{\cosh(l) - 1} \cos(\theta) = 32e^{\frac{3l}{2}}(\cosh(l)+1)\cos(\theta)$
  \item $(\cosh(l)-1)|\mu| = |h_2 + e^{\frac{3l}{2}}e^{i\theta}h_1 -\frac{3}{2} \overline{\mu}|$
 \end{enumerate}    

  We make some computations in order to have a more explicit second condition.  
It implies, by triangular inequality, that:
\[(\cosh(l) - 1)|\mu| \geq   |h_2 + e^{\frac{3l}{2}}e^{i\theta}h_1| - \frac{3}{2} |\mu|\]

We deduce that:

\[(\cosh(l )+ \frac{1}{2})|\mu| \geq  |h_2 + e^{\frac{3l}{2}}e^{i\theta}h_1|\]

By condition 1, we deduce that:

\[|h_2 + e^{\frac{3l}{2}}e^{i\theta}h_1|^2 \leq 32e^{\frac{3l}{2}}(\cosh(l )+ \frac{1}{2})^2 (\cosh(l)+1) \cos(\theta)\]

This equation can be re-written in the following way:
\begin{equation}\label{inegalite_cote_lox}
|h_2 + e^{\frac{3l}{2}}e^{i\theta}h_1|^2 \leq 16e^{\frac{3l}{2}}(2\cosh(l )+ 1)^2\cosh^2(\frac{l}{2}) \cos(\theta).
\end{equation}

The left side of the inequality can be written in the form:

\begin{equation}
 |h_2|^2 + 2e^{\frac{3l}{2}}\Re(e^{i\theta}h_1 \con{h_2} ) + e^{3l}|h_1|^2
\end{equation}

\begin{comment}
Let $\phi$ be the argument of $\con{h_2}h_1$. En récrivant la dernière inégalité, on obtient:

\begin{equation}\label{eqn1_test}
 |h_2|^2 + 2e^{\frac{3l}{2}}|h_2||h_2|\cos(\theta + \phi) + e^{3l}|h_1|^2 \leq 8e^{\frac{3l}{2}}(2\cosh(l )+ 1)^2(\cosh(l)+1) \cos(\theta)
\end{equation}
\end{comment}

 We compute those terms.
We have:
\begin{eqnarray*}
 h_1 = \langle p'_U,p_V \rangle &=& (2 \sinh(l)+3 + 2i\sin(2\alpha_2))(e^{2i\alpha_2}+1) \\
  &=& 2\cos(\alpha_2)e^{i\alpha_2}(2 \sinh(l)+3 + 2i\sin(2\alpha_2)) \\
  h_2 = \langle p'_U,p_V \rangle &=& (-2 \sinh(l)+3 + 2i\sin(2\alpha_2))(e^{2i\alpha_2}+1) \\
  &=& 2\cos(\alpha_2)e^{i\alpha_2}(-2 \sinh(l)+3 + 2i\sin(2\alpha_2))
\end{eqnarray*}

We easily deduce, replacing $\sin(2\alpha_2)^2$ by $4\cos^2(\alpha_2)(1-\cos^2(\alpha_2))$ and $\cos^2(\alpha_2)$ by $\frac{1}{8}(2\cosh(l)+1)$, that:
\begin{comment}
\begin{eqnarray*}
 |h_2|^2 &=& 6(2\cosh(l) + 1) (e^{-l} + 1) \\
 |h_1|^2 &=& 6(2\cosh(l) + 1) (e^{l} + 1) \\
 |h_1||h_2| &=& 12(2\cosh(l) + 1)\cosh(\frac{l}{2})
\end{eqnarray*}
\end{comment}

\begin{eqnarray*}
 |h_1|^2 &=& 12e^{\frac{l}{2}}(2\cosh(l) + 1) \cosh(\frac{l}{2}) \\
 |h_2|^2 &=& 12e^{-\frac{l}{2}}(2\cosh(l) + 1) \cosh(\frac{l}{2}) \\
 |h_1||h_2| &=& 12(2\cosh(l) + 1)\cosh(\frac{l}{2})
\end{eqnarray*}

Furthermore,
\begin{equation}
 h_1 \con{h_2} = 2(2\cosh(l)+1)((5-2\cosh(l))(\cosh(l)+1) - 4i\sin(2\alpha_2)\sinh(l))
\end{equation}

We deduce then: 

\begin{eqnarray*}
2\Re(e^{i\theta}h_1\con{h_2}) &=& 4(2\cosh(l)+1) \\ & & \times((5-2\cosh(l))(\cosh(l)+1)\cos(\theta) + 4\sin(2\alpha_2)\sinh(l)\sin(\theta)) \\
&=& 8(2\cosh(l)+1)\cosh(\frac{l}{2}) \\ & & \times ((5-2\cosh(l))\cosh(\frac{l}{2})\cos(\theta) + 4\sin(2\alpha_2)\sinh(\frac{l}{2})\sin(\theta))
\end{eqnarray*}

The inequality (\ref{inegalite_cote_lox}) becomes, after simplifying by $4(2\cosh(l) + 1) \cosh(\frac{l}{2})e^{\frac{3l}{2}}$:
\begin{eqnarray*}
3e^{-2l} + 3e^{2l} + 2((5-2\cosh(l))\cosh(\frac{l}{2})\cos(\theta) + 4\sin(2\alpha_2)\sinh(\frac{l}{2})\sin(\theta)) \\
 \leq 4(2\cosh(l )+ 1)\cosh(\frac{l}{2}) \cos(\theta)
\end{eqnarray*}

Notice that by replacing $\frac{3l}{2}$ by $-\frac{5l}{2}$, we obtain the same inequality.
Passing the terms in $\theta$ to the right side, we have:
\begin{equation}
3e^{-2l} + 3e^{2l} \leq  6(2\cosh(l)-1)\cosh(\frac{l}{2})\cos(\theta) - 8\sin(2\alpha_2)\sinh(\frac{l}{2})\sin(\theta)
\end{equation}

The right side is of the form $a\cos(\theta) +b\sin(\theta)$. Its maximum value is $\sqrt{a^2 + b^2}$. In particular, we have:
\begin{equation}
(3e^{-2l} + 3e^{2l})^2 \leq 36(2\cosh(l)-1)^2\cosh^2(\frac{l}{2}) + 64 \sin^2(2\alpha_2)\sinh^2(\frac{l}{2})
\end{equation}

By replacing $\sin(2\alpha_2)^2$ by its expression in terms of $\cosh(l)$ we see that the right side, after linearisation, equals $16\cosh(3l)+16\cosh(2l)-16\cosh(l)+20$. We deduce that

\begin{equation}
36\cosh(2l)^2 \leq 16\cosh(3l)+16\cosh(2l)-16\cosh(l)+20
\end{equation}

By linearising the left side and dividing by 2, we obtain:
\begin{equation}
9\cosh(4l) +8\cosh(l) \leq 8\cosh(3l)+8\cosh(2l)+1 
\end{equation}

\begin{equation}
 8\cosh(4l)+8\cosh(l) +\cosh(4l)  \leq  8\cosh(3l) + 8\cosh(2l) +1
\end{equation}

Since the function $\cosh$ is convex, we have:

\begin{equation}
8\cosh(4l)+8\cosh(l) \geq 8\cosh(3l) + 8\cosh(2l)
\end{equation}
and, on the other hand, $\cosh(4l) > 1$, which gives a contradiction.

\end{proof}

Since $D_1^-$ is the image of $D_0^+$ by $z \mapsto \frac{1}{z}$, we have the following corollary:
\begin{cor}
 The disk $D_0^-$ is contained in the open annulus of center $0$ and radii $e^{\frac{-3l}{2}}$ and $e^{\frac{5l}{2}}$
\end{cor}

\begin{prop} The disks $D_k^{\pm}$ intersect in the following way:
 \begin{enumerate}
 \item $D_0^+$ intersects $D_k^+$ if and only if $k \in \{-1,0,1\}$.
 \item $D_0^-$ intersects $D_k^-$ if and only if $k \in \{-1,0,1\}$.
 \item $D_0^+$ intersects $D_k^-$ if and only if $k \in \{-1,0,1\}$. Furthermore, $D_0^+$ is tangent to $D_1^-$ and $D_0^+$.
\end{enumerate} 
\end{prop}
\begin{proof}
 By an immediate induction, we know that the disk $D_k^+$ is contained in the open annulus of radii $e^{2kl+\frac{3l}{2}}$ and $e^{2kl - \frac{5l}{2}}$ and that the disk $D_k^-$ is contained in the open annulus of radii  $e^{2kl+\frac{3l}{2}}$ and $e^{2kl - \frac{5l}{2}}$.
 If $|k|>0$, the disks $D_k^+$ and $D_0^+$ are then in two disjoint annuli. The same happens for the disks $D_k^-$ and $D_0^-$, as well as for $D_k^-$ and $D_0^+$.
 Hence it only remains to prove that the expected intersections exist. By Proposition \ref{prop_inter_biss_lox}, we know that
\begin{itemize}
 \item  $\mathcal{J}_0^+$ intersects $\mathcal{J}_1^+$ and $\mathcal{J}_{-1}^+$
 \item  $\mathcal{J}_0^-$ intersects $\mathcal{J}_1^-$ and $\mathcal{J}_{-1}^-$
 \item  $\mathcal{J}_0^+$ intersects $\mathcal{J}_0^-$,
\end{itemize} 
which proves the incidences.
 Finally, by Proposition \ref{prop_tangence_disques_lox}, we know that $D_0^+$ is tangent to $D_1^-$ and $D_0^+$.
\end{proof}

With the projection described above, we obtain Proposition \ref{prop_inter_biss_lox}. Let us recall its statement:

 \begin{prop}\label{prop_inter_globale_lox}
  If $U$ is loxodromic, then
\begin{enumerate}
 \item $\mathcal{J}_0^+$ intersects $\mathcal{J}_k^+$ if and only if $k \in \{-1,0,1\}$.
 \item $\mathcal{J}_0^-$ intersects $\mathcal{J}_k^-$ if and only if $k \in \{-1,0,1\}$.
  \item $\mathcal{J}_0^+$ intersects $\mathcal{J}_k^-$ if and only if $k \in \{-1,0,1\}$. Furthermore, $\mathcal{J}_0^+ \cap \mathcal{J}_1^- = \{[Up_A]\}$ and $\mathcal{J}_0^+ \cap \mathcal{J}_{-1}^- = \{[U^{-1}p_B]\}$.
\end{enumerate}    
 \end{prop}

\subsection{The elliptic side}\label{subsect_combi_ell}
 We are going to study now the case where $[U]$ is elliptic. As in the loxodromic case, we are going to determine the relative position of the disks  $D_{k}^{\pm}$, that will be obtained this time by rotations. Some examples of the relative position of the disks $D_{k}^{\pm}$ are pictured in Figures \ref{fig_proj_partielle_ell} and \ref{fig_proj_complete_ell}.
 However, in this case we will need to slightly refine the argument and consider the real visual sphere in order to determine the global combinatorics of the intersections of the bisectors $\mathcal{J}_k^{\pm}$.
 We consider here the parameters $\alpha_2 \in ]\alpha_2^{\lim} , \frac{\pi}{2} [$. By Remark \ref{rem_alpha2_ell}, we know that for such a parameter $\alpha_2$, the element $U$ is elliptic and has eigenvalues $1 , e^{i\beta}$ and $e^{-i\beta}$. The angle $\beta \in ]0 , \frac{2\pi}{3} [$ can also be used as a parameter for the deformation. It is related to $\alpha_2$ by the equation:
 \begin{equation*}
  2\cos(\beta) + 1 = \mathrm{tr}(U) = 8\cos^2(\alpha_2)
 \end{equation*}

In this case, also recall that, by setting $\delta = 2\sin(\beta)$, we have:
\[
[p'_U] = \begin{bmatrix}
2(2e^{2i\alpha_2}+1) \\
 -\sqrt{2}e^{i\alpha_2}(2e^{2i\alpha_2}+1 + \delta) \\
 -(8\cos^2(\alpha_2)+1) -\delta
\end{bmatrix}
\text{ and }
[p''_U] = \begin{bmatrix}
2(2e^{2i\alpha_2}+1) \\
 -\sqrt{2}e^{i\alpha_2}(2e^{2i\alpha_2}+1 - \delta) \\
 -(8\cos^2(\alpha_2)+1) + \delta
\end{bmatrix}.
\]

\begin{figure}[ht]
 \center
 \includegraphics[width = 7cm]{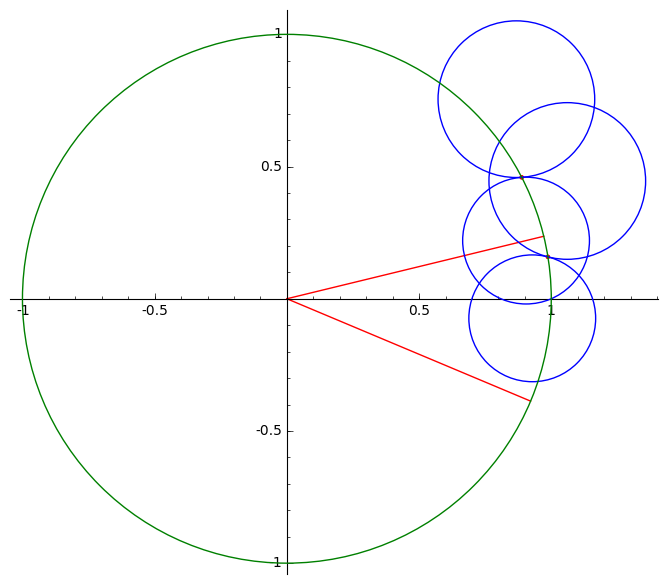}
 \caption{Projection of the bisectors on $L_{[p_U]}$ for $\alpha_2 = 0,915$. \label{fig_proj_partielle_ell}}
\end{figure}

\begin{figure}[htbp]
\center
\begin{subfigure}{0.4\textwidth}
 \includegraphics[width = 7cm]{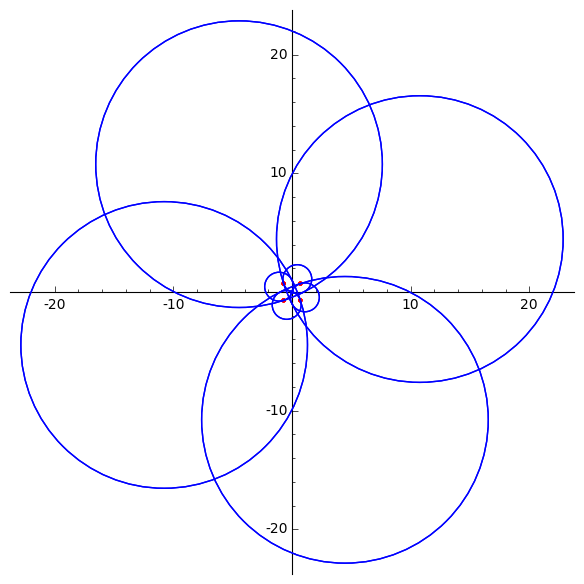}
 \caption{$n=8$}
 \end{subfigure} \hspace{1cm}
 \begin{subfigure}{0.4\textwidth}
 \includegraphics[width = 7cm]{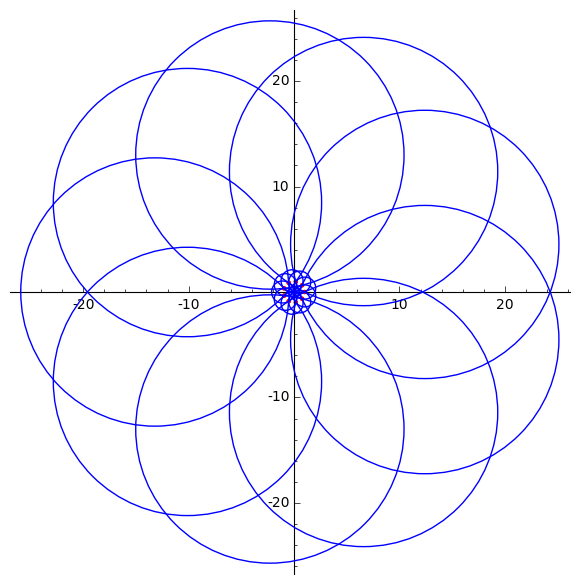}
 \caption{$n=9$}
 \end{subfigure}  \\
 \begin{subfigure}{0.4\textwidth}
 \includegraphics[width = 7cm]{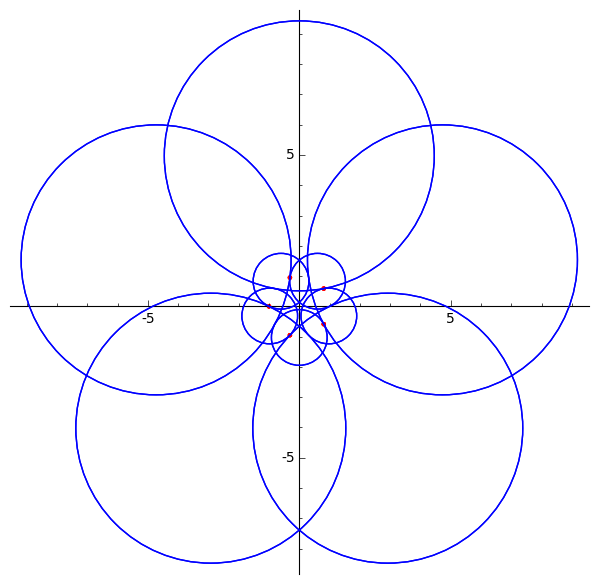}
 \caption{$n=10$}
 \end{subfigure} \hspace{1cm}
 \begin{subfigure}{0.4\textwidth}
 \includegraphics[width = 7cm]{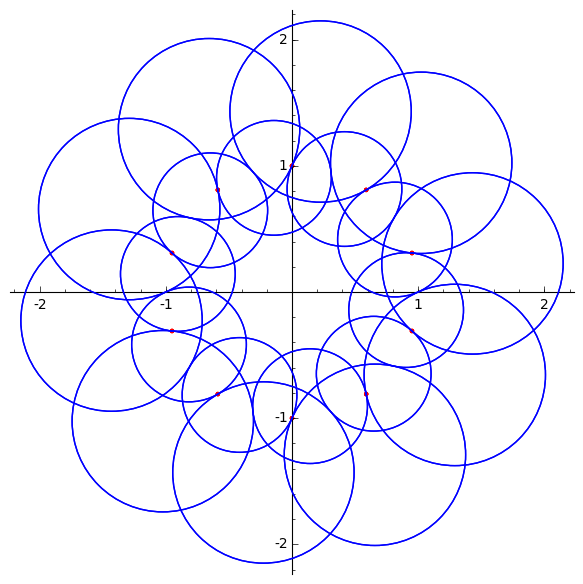}
 \caption{$n=20$}
 \end{subfigure} 
\caption{Projection on the visual sphere for $n = 8,9,10$ and $20$. \label{fig_proj_complete_ell}}
\end{figure}

 We make now some remarks about the actions on $L_{[p_U]}$ and the images of some points in the chart $\psi$.

\begin{rem} If $U$ is elliptic, we have $\psi(p_A) =  \frac{\mathrm{tr}(U)+1 - \delta}{\mathrm{tr}(U)+1 + \delta} = \frac{1 + e^{-i\beta}}{1 + e^{i\beta}} = e^{-i\beta}$.
\end{rem}

\begin{rem}
The action of $U$ on $L_{[p_U]}$ fixes $l_{[p_U],[p'_U]}$ and $l_{[p_U],[p''_U]}$. In the chart $\psi_{p'_U,p''_U}$, since $Up'_U = e^{i\beta}p'_U$ and $Up''_U = e^{-i\beta}p''_U$,
 the action of $U$ is a rotation of center $0$ and angle $2 \beta$; it is given by $z \mapsto e^{2i\beta}z$.
\end{rem}

\begin{rem}
 Since $\psi(p_B) = 1$, $\psi(p_A) = e^{-i\beta}$ and since $U$ acts by a rotation of angle $2\beta$, we deduce that for all $k\in \mathbb{Z}$, we have
 
  \[\psi(U^kp_B) = e^{2ik\beta} \text{ and } \psi(U^kp_A) = e^{(2k-1)i\beta}.\]
\end{rem}

We can deduce from the actions of $U$ and $I$ a result on the intersections :
\begin{prop}
 The disks $D_0^+$ and $D_1^{-}$ are tangent at the point $\pi_{[p_U]}(l_{[p_U],[Up_A]})$.
 
  The disks $D_0^+$ and $D_{-1}^{-}$ are tangent at the point $\pi_{[p_U]}(l_{[p_U],[U^{-1}p_B]})$.
\end{prop}

\begin{proof}
Recall that $D_0^+$ is a disk whose boundary contains $\psi(U^{-1}p_B)$ and $\psi(Up_A)$, and that $D_1^{-}$ is a disk whose boundary contains $\psi(U^2p_B)$ and $\psi(Up_A)$.
 Let $\theta$ be the angle at $\psi(Up_A)$ between $\partial D_0^+$ and the unit circle. Since the map $z \mapsto \frac{1}{z}$ is conformal, and it exchanges $D_0^+$ and $D_0^-$, the angle between $\partial D_0^-$ and the unit circle at $\psi(p_A) = \frac{1}{\psi(Up_A)}$ is also equal to $\theta$. Since $D_1^{-}$ is obtained from $D_0^-$ by a rotation of angle $2\beta$, the angle between the unit circle and $\partial D_1^{-}$ at $\psi(Up_A)$ equals $\theta$. Since $\partial D_1^{-}$ and $\partial D_0^{+}$ intersect the unit circle at $\psi(Up_A)$ with the same angle, they are tangent.
 
 The proof for the other tangency is analogous.
\end{proof}

\begin{cor}\label{cor_tangence_bissec_ell}
 The bisectors $ \mathcal{J}_0^+$ and $ \mathcal{J}_1^-$ are tangent at $[Up_A]$.
 The bisectors  $ \mathcal{J}_0^+$ and $ \mathcal{J}_{-1}^-$ are tangent at $[U^{-1}p_B]$.
\end{cor}
\begin{proof}
 By the proposition above, we know that the intersection of $ \mathcal{J}_0^+$ and $ \mathcal{J}_1^-$ is contained in the line $l_{[p_U],[Up_A]}$. By Lemma \ref{lemme_tangences_faces}, this line is tangent to $\partial_\infty \mathcal{J}_0^+$ at $[Up_A]$, hence the intersection contains exactly one point.
\end{proof}

In the same way as for the loxodromic side, the key proposition for controlling the intersections of the disks $D_k^{\pm}$ is given by the following statement, that bounds the angular diameter of the disks seen from $0$. The action of $[U]$ by rotations is then analogous to the action by homotheties that we have studied in the loxodromic side. The annuli correspond to angular sectors: an example is pictured in Figure \ref{fig_proj_partielle_ell}.

\begin{prop}
 If $\beta \geq \frac{2\pi}{9}$, the disks $D_k^{\pm}$ have an angular diameter $< 4\beta$ from $0$.
\end{prop}

\begin{proof}
 Since $D_k^{\pm}$ is the image of $D_0^+$ by a rotation centered at $0$ and eventually $z \mapsto \frac{1}{z}$, it is enough to show that the angular diameter from $0$ of the disk $D_0^+$ is $< 4 \beta$.
 We are going to show that the real half-lines with arguments $\frac{3\beta}{2}$ and $-\frac{5\beta}{2}$ do not intersect the disk $D_0^+$. These two half-lines are traced in Figure \ref{fig_proj_partielle_ell}.
  This comes to show that if $k \in \mathbb{R}^{+}$, then no point with image $k e^{\frac{3i\beta}{2}}$ or $k e^{-\frac{5i\beta}{2}}$ by $\psi_{p'_U , p''_U}$ belongs to the bisector $\mathfrak{B}(p_U,p_V)$.
  
  We begin with the first case. We want to show that a point of the form $q_{k,\mu} = p''_U+ ke^{\frac{3\beta}{2}} p'_U + \mu p_U $, where $k \in \mathbb{R}^{+}$ and $\mu \in \mathbb{C}$, does not belong to the bisector $\mathfrak{B}(p_U,p_V)$. 
  
 We have two conditions for $q_{k,\mu}$ to belong to $\mathfrak{B}(p_U,p_V)$:
 \begin{enumerate}
  \item $\langle q_{k,\mu} , q_{k ,\mu} \rangle \leq 0$
  \item $|\langle q_{k,\mu} , p_U \rangle| = |\langle q_{k,\mu} , p_V \rangle|$
 \end{enumerate}
   
   They can be written as:
   
   \begin{enumerate}
    \item $\langle p''_U , p''_U \rangle + k^2 \langle p'_U , p'_U \rangle + |\mu|^2\langle p_U , p_U \rangle \leq 0$
    \item $|\mu||\langle p_U , p_U \rangle| = |\langle p''_U , p_V \rangle + ke^{-\frac{3\beta}{2}} \langle p'_U , p_V \rangle + \overline{\mu} \langle p_U , p_V \rangle|$
   \end{enumerate}
   
   By the computations above, setting $h_1 = \langle p'_U , p_V \rangle $ and $h_2 = \langle p''_U , p_V \rangle $ to simplify the notation, the conditions become:
   
 \begin{enumerate}
  \item $|\mu|^2 \geq (1+k^2)\frac{16\sin^2(\beta)}{1-\cos(\beta)} = 16(1+k^2)(1+\cos(\beta))$
  \item $(1-\cos(\beta))|\mu| = |h_2 + ke^{-\frac{3\beta}{2}}h_1 -\frac{3}{2} \overline{\mu}|$
 \end{enumerate}    

 We make some computations to have more details on the second condition.  
It implies, by the triangle inequality, that:
\[(1-\cos(\beta))|\mu| \geq  \frac{3}{2} |\mu| - |h_2 + ke^{-\frac{3\beta}{2}}h_1|\]

We deduce that:

\[(\cos(\beta )+ \frac{1}{2})|\mu| \leq  |h_2 + ke^{-\frac{3\beta}{2}}h_1|\]

By condition 1, we deduce that:

\[|h_2 + ke^{-\frac{3\beta}{2}}h_1|^2 \geq 16(\cos(\beta )+ \frac{1}{2})^2 (1+k^2)(1+\cos(\beta))\]

\[|h_2 + ke^{-\frac{3\beta}{2}}h_1|^2 \geq 4(2\cos(\beta )+ 1)^2 (1+k^2)(1+\cos(\beta)).\]

Let us compute the quantity $|h_2 + ke^{-\frac{3\beta}{2}}h_1|^2$. We have:

\[|h_2 + ke^{-\frac{3\beta}{2}}h_1|^2 = |h_2|^2 + k^2|h_1|^2 + 2k \mathrm{Re}(e^{\frac{3\beta}{2}}h_2\overline{h_1})\] 
  
In order to obtain a contradiction, it is enough to show that the following polynomial on $k$ is always negative:
\begin{eqnarray*}
P_\beta(k) &=& |h_2|^2 + k^2|h_1|^2 + 2k \mathrm{Re}(e^{\frac{3\beta}{2}}h_2\overline{h_1}) - 4(2\cos(\beta )+ 1)^2 (1+k^2)(1+\cos(\beta)) \\
&=& (|h_2|^2-4(2\cos(\beta )+ 1)^2(1+\cos(\beta)) \\
 & & 
+ 2k \mathrm{Re}(e^{\frac{3\beta}{2}}h_2\overline{h_1}) +(|h_1|^2-4(2\cos(\beta )+ 1)^2(1+\cos(\beta))k^2
\end{eqnarray*}
 
  In order to study the polynomial $P_\beta$, we need to compute some expressions. We make the computation in the following lemma: 
  
\begin{lemme}
We have:
\begin{enumerate}
 \item $h_2 \overline{h_1} = 6(2\cos(\beta)+1)(e^{i\beta}+1) =  12e^{i\frac{\beta}{2}}(2\cos(\beta)+1)\cos(\frac{\beta}{2})$
 \item $|h_1|^2|h_2|^2 = 72(2\cos(\beta)+1)^2(\cos(\beta) +1)$
 \item $|h_1|^2 + |h_2|^2 = 4(2\cos(\beta)+1)(5-2\cos(\beta))(\cos(\beta) +1)$
\end{enumerate}

\end{lemme}  
  \begin{proof}
  Consider the explicit expressions of $h_1$ and $h_2$ in terms of $\alpha_2$ and $\beta$. We have:
  \begin{eqnarray*}
  h_1 = \langle p'_U , p_V \rangle &=& e^{-i\beta}(2i \sin(2\alpha_2) + \cos(\beta) + \frac{1}{2}) + 4i\sin(2\alpha_2) -1 \\
  h_2 = \langle p''_U , p_V \rangle &=& e^{i\beta}(2i \sin(2\alpha_2) + \cos(\beta) + \frac{1}{2}) + 4i\sin(2\alpha_2) -1 
  \end{eqnarray*}

By developing $h_2 \overline{h_1}$ and writing $\sin^2(2\alpha_2)$ in terms of $\cos(\beta)$, we obtain:
 \begin{equation*}
 h_2 \overline{h_1} = 6(2\cos(\beta)+1)(e^{i\beta}+1) =  12e^{i\frac{\beta}{2}}(2\cos(\beta)+1)\cos(\frac{\beta}{2})
\end{equation*}   
  
On the other hand, we can compute:
  \begin{eqnarray*}
  |h_1|^2 = 2(2\cos(\beta)+1)((\cos(\beta)+1)(5-2\cos(\beta))-4\sin(2\alpha_2)\sin(\beta)) \\
  |h_2|^2 = 2(2\cos(\beta)+1)((\cos(\beta)+1)(5-2\cos(\beta))+4\sin(2\alpha_2)\sin(\beta))
  \end{eqnarray*}
  
  We deduce immediately that
  \begin{equation*}
   |h_1|^2 + |h_2|^2 = 4(2\cos(\beta)+1)(5-2\cos(\beta))(\cos(\beta) +1)
  \end{equation*}
 and, by developing and writing $\sin^2(2\alpha_2)$ in terms of $\cos(\beta)$:
\begin{equation*}
 |h_1|^2|h_2|^2 = 72(2\cos(\beta)+1)^2(\cos(\beta) +1)
\end{equation*} 
  \end{proof}

We deduce that $\mathrm{Re}(e^{\frac{3\beta}{2}}h_2\overline{h_1}) = 12 \cos(2\beta)(2\cos(\beta)+1)\cos(\frac{\beta}{2})$. Remark that it is exactly the same term if we replace $\frac{3}{2}\beta$ by $-\frac{5}{2}\beta$. 
  Hence the polynomial $P_\beta(k)$ can be written in the following way:
  \begin{equation*}
  P_\beta(k) = |h_2|^2(1 - \frac{|h_1|^2}{18}) + 24 \cos(2\beta)(2\cos(\beta)+1)\cos(\frac{\beta}{2})k + |h_1|^2(1 - \frac{|h_2|^2}{18})k^2
\end{equation*}   
  Its discriminant $\Delta_\beta$ is then equal to:

  \begin{equation*}
  \Delta_\beta = (24 \cos(2\beta)(2\cos(\beta)+1)\cos(\frac{\beta}{2}))^2 - 4|h_1|^2|h_2|^2(1 - \frac{|h_1|^2}{18}) (1 - \frac{|h_2|^2}{18})
\end{equation*}  
  By writing the first term in terms of $\cos(\beta)$, we obtain:
  \begin{eqnarray*}
  \Delta_\beta &=& 288 (2\cos^2(\beta)-1)^2(2\cos(\beta)+1)^2(\cos(\beta)+1) - 4|h_1|^2|h_2|^2(1 - \frac{|h_1|^2}{18}) (1 - \frac{|h_2|^2}{18}) \\
   &=& 4|h_1|^2|h_2|^2 ((2\cos^2(\beta)-1)^2 - (1 - \frac{1}{18}|h_1|^2)(1 - \frac{1}{18}|h_2|^2))
\end{eqnarray*}
  If $\beta \in ]0 , \frac{2\pi}{3}[$, we know that $4|h_1|^2|h_2|^2 > 0$. By developing and writing the last factor in terms of $\cos(\beta)$, we obtain:
  \begin{eqnarray*}
 \frac{\Delta_\beta}{4|h_1|^2|h_2|^2} &=& (2\cos^2(\beta)-1)^2 - (1 - \frac{1}{18}|h_1|^2)(1 - \frac{1}{18}|h_2|^2) \\
   &=&  4\cos^2(\beta)(\cos^2(\beta)-1) + \frac{|h_1|^2 + |h_2|^2}{18} - \frac{|h_1|^2 |h_2|^2}{18^2}\\
   &=& \frac{4}{9} (1-\cos^2(\beta))(2 + 4 \cos(\beta) - 9\cos^2(\beta)) \\
   &=& \frac{4}{9} \sin^2(\beta)(2-4 \cos(\beta) - 9\cos^2(\beta))
\end{eqnarray*}
  The discriminant has the same sign as $2 + 4 \cos(\beta) - 9\cos^2(\beta)= -(\cos(\beta)- \frac{3}{4})(9\cos(\beta)+\frac{11}{4}) - \frac{1}{16}$.
 In particular, if $\cos(\beta) \geq \frac{3}{4}$, it is negative. Since $\cos(\frac{2\pi}{8}) < \frac{3}{4} < \cos(\frac{2\pi}{9})$, we are in this case if $\beta \leq \frac{2\pi}{9}$.
 
 We deduce that if $\beta \in ]0 , \frac{2\pi}{9}]$, the polynomial $P_\beta$ has no real roots. In order to show that it is negative, it is enough to show that its constant coefficient is negative. Since the discriminant $\Delta_\beta$ is negative for all $\beta \in ]0 , \frac{2\pi}{9}]$, this coefficient has constant sign. If $\beta = 0$, we have $P_0 = -36 + 72k -36k^2$, which allows us to complete the proof.
 
 The computation for the half-line of argument $-\frac{5}{2}\beta$ is identical.
\end{proof}

\subsubsection{Global intersections}
We are going to show Proposition \ref{prop_inter_biss_ell} using the projection defined above. Recall its statement:

 \begin{prop}\label{prop_inter_globale_ell}
  If $U$ is elliptic of order $n \geq 9$, then
\begin{enumerate}
 \item $\mathcal{J}_0^+$ intersects $\mathcal{J}_k^+$ if and only if $k \equiv -1,0,1 \mod n$.
 \item $\mathcal{J}_0^-$ intersects $\mathcal{J}_k^-$ if and only if $k \equiv -1,0,1 \mod n$.
  \item $\mathcal{J}_0^+$ intersects $\mathcal{J}_k^-$ if and only if $k \equiv -1,0,1 \mod n$. Furthermore, $\mathcal{J}_0^+ \cap \mathcal{J}_1^- = \{[Up_A]\}$ and $\mathcal{J}_0^+ \cap \mathcal{J}_{-1}^- = \{[U^{-1}p_B]\}$.
\end{enumerate}    
 \end{prop} 

 In order to complete the proof and refine the argument on the visual sphere, we will need the following lemma about the real angular diameter of a bisector in $\h2c$:
 
 \begin{lemme}\label{lemme_diam_ang_reel_ordre_geq_5}
  If $[U]$ is elliptic of order $\geq 5$, then the real angular diameter of $\mathcal{J}_k^{\pm}$ seen from $[p_U]$ is strictly less than $\frac{\pi}{3}$.
 \end{lemme} 
 \begin{proof}
 By symmetry, it is enough to show that the angular diameter of $\mathcal{J}_0^{+}$ seen from $[p_U]$ is strictly less than $\frac{\pi}{3}$.
 
  Let $\beta$ be the angle of rotation of $[U]$. We know that $\beta \leq \frac{2\pi}{5}$, that $\mathrm{tr}(U) = 1 + 2\cos(\beta)$ and that $\langle p_U , p_U \rangle = \langle p_V , p_V \rangle = \cos(\beta) - 1$.
 In this case, the quantity $\frac{\langle p_U, p_V \rangle \langle p_V ,p_U \rangle}{\langle p_U , p_U \rangle \langle p_V , p_V \rangle}$ equals $\frac{(3/2)^2}{(1-\cos(\beta))^2}$. Since $\beta \leq \frac{2\pi}{5}$, we have $\cos(\beta) \geq \cos(\frac{2\pi}{5}) = \frac{\sqrt{5}-1}{4} > \frac{1}{4}$, and 
 $\frac{(3/2)^2}{(1-\cos(\beta))^2} > 4$. By Corollary \ref{cor_diam_ang_reel_bisector} of Proposition \ref{prop_diam_ang_reel_bisector}, we know that the angular diameter of $\mathcal{J}_0^{+}$ seen from $[p_U]$ is strictly less than $\frac{\pi}{3}$.
 \end{proof}

\begin{notat}
 For $s \in \mathbb{R}$, let $U^s$ the element $U^s = \exp(s \Log(U)) \in \su21$. In this way, $U^sp_U = p_U$ , $U^sp'_U = e^{is\beta} p'_U$ and $U^sp''_U = e^{-is\beta} p''_U$. It acts on $L_{[p_U]}$ as a rotation of angle $s\beta$ in the chart $\psi_{p'_U,p''_U}$.
\end{notat}

\begin{proof}[Proof of Proposition \ref{prop_inter_globale_ell}]
 By symmetry, the points 1 and 2 are analogous. Furthermore, by Corollaries \ref{cor_incidences} and \ref{cor_tangence_bissec_ell}, the expected intersections and tangencies occur. 
 
 We begin by showing the first point of the statement. Let $\beta = \frac{2\pi}{n}$. We know that $D_0^{+}$ is a disk with angular diameter $< 4 \beta$. If $s \in [2,\frac{n}{2}-2] \cup [\frac{n}{2}+2 , n-2]$, the disk $D_k^{+}$ does not intersect its image by a rotation of angle $2s\beta$, since they are contained in disjoint angular sectors.
 Hence, it is enough to show that if $s \in ]\frac{n}{2}-2 , \frac{n}{2}+2[$, then $[U^s]\mathcal{J}_0^+$ does not intersect $\mathcal{J}_0^+$
 
 Notice first that $[U^{\frac{n}{2}}]$ is the reflection on $[p_U]$. We know that $[U^{\frac{n}{2}}]\mathcal{J}_0^+ = \mathfrak{B}([U^{\frac{n}{2}} p_V] , p_U)$ and $J_0^+$ are disjoint, since their orthogonal projections on the complex geodesic $l_{[p_U],[p_V]}$ are the (real) line bisector of $[p_U]$ and $[p_V]$, and its image by the refection on $[p_U]$.
 
 For $s \in \mathbb{R}$, let $C_s$ be the real cone on $[U^{s}]\mathcal{J}_0^+$ with vertex $[p_U]$. We know that $C_0$ and $C_{\frac{n}{2}}$ are disjoint and opposite with respect to $[p_U]$.
 Furthermore, the projection of $C_s$ on $L_{[p_U]}$ is the rotation of angle $2s$ of $D_0^+$.  Let $s \in ]\frac{n}{2}-2 , \frac{n}{2}+2[$, such that this projection intersects $D_0^+$.
 By Lemma \ref{lemme_diam_ang_reel_ordre_geq_5}, we know that the cone $C_s$ has an angular diameter strictly less than $\frac{\pi}{3}$.
 Hence it cannot intersect $C_0$ and $C_{\frac{n}{2}}$, because we would have a union of three cones with angular diameters $< \frac{\pi}{3}$ giving a connected set with angular diameter $\geq \pi$.
  Hence, the cone $C_s$ intersects either $C_0$, of $C_{\frac{n}{2}}$. By continuity of $s \mapsto [U^s]$, it only intersects $C_{\frac{n}{2}}$.
 In an analogous way, the intersections of $\mathcal{J}_0^+$ with the orbit by $[U]$ of $\mathcal{J}_0^-$ are exactly the expected ones.
\end{proof}

\break

 \bibliographystyle{alpha}
%\bibliography{./tex/biblio_these_updated}
\bibliography{./tex/ref_unif_cr}

\end{document}